\documentclass[reqno,11pt]{amsart}
\usepackage[a4paper,height=20cm,top=26mm,bottom=26mm,left=26mm,right=26mm]{geometry}

\usepackage{subfiles}
\usepackage{cleveref}

\usepackage{etex}
\usepackage{graphicx}
\usepackage{epic,eepic}
\usepackage{xypic}
\usepackage{amsmath,amssymb,amscd}
\usepackage[mathscr]{eucal}

\usepackage{here}

\usepackage{bm}
\usepackage[all]{xy}
\usepackage{tikz}
\usetikzlibrary{calc, math}

\numberwithin{figure}{section}
\numberwithin{table}{section}

\numberwithin{equation}{section}

\newtheorem{theorem}{Theorem}[section]
\newtheorem{lemma}[theorem]{Lemma}

\newtheorem{proposition}[theorem]{Proposition}

\theoremstyle{definition}

\newtheorem{remark}[theorem]{Remark}

\newcommand{\Z}{{\mathbb Z}}

\newcommand{\C}{{\mathbb C}}

\newcommand{\e}{{\mathrm e}}

\def\ve{{\varepsilon}}

\newcommand\trop{{\mathrm{trop}}}
\newcommand\rY{{\mathsf{Y}}}
\newcommand\scY{{\mathscr{Y}}}
\newcommand\Ad{{\mathrm{Ad}}}

\newcommand{\qR}{\mathcal{R}}
\newcommand{\qK}{\mathcal{K}}

\newcommand\qarrow[2]{\draw[->,shorten >=2pt,shorten <=2pt] (#1) -- (#2) [thick];} 
\newcommand\qdarrow[2]{\draw[->,dashed,shorten >=2pt,shorten <=2pt] (#1) -- (#2) [thick];} 

\newcommand\qsarrow[2]{\draw[->,shorten >=4pt,shorten <=4pt] (#1) -- (#2) [thick];}
\newcommand\qarrowsa[2]{\draw[->,shorten >=2pt,shorten <=4pt] (#1) -- (#2) [thick];} 
\newcommand\qarrowsb[2]{\draw[->,shorten >=4pt,shorten <=2pt] (#1) -- (#2) [thick];} 

\newcommand\maruni[1]{\path (#1) node[circle]{2}; \draw (#1) circle[radius=0.15];} 

\begin{document}


\title[3D reflection equation from symmetric butterfly quiver]
{Solutions of 3D Reflection Equation from 
\\
Quantum Cluster Algebra Associated with 
\\
Symmetric Butterfly Quiver}

\author[Rei Inoue]{Rei Inoue}
\address{Rei Inoue, Department of Mathematics and Informatics,
   Faculty of Science, Chiba University,
   Chiba 263-8522, Japan.}
\email{reiiy@math.s.chiba-u.ac.jp}

\author[Atsuo Kuniba]{Atsuo Kuniba}
\address{Atsuo Kuniba, Institute of Physics, Graduate School
of Arts and Sciences, University of Tokyo, Komaba, Tokyo, Japan.}
\email{atsuo.s.kuniba@gmail.com}

\dedicatory{Dedicated to the memory of Professor Rodney James Baxter}

\date{December 26, 2025}


\begin{abstract}
We construct a new solution $(R,K)$ to the three-dimensional
reflection equation, a boundary analogue of the tetrahedron equation.
The $R$-operator is the one obtained by Sun, Terashima, Yagi, and the
authors in 2024, involving four quantum dilogarithms with arguments in
the $q$-Weyl algebra.  The new $K$-operator similarly involves 
ten such quantum dilogarithms.
Our approach is based on
the quantum cluster algebra associated with the symmetric butterfly
quiver on the wiring diagram of type~C.
\end{abstract}



\maketitle

\section{Introduction and main result}

\subsection{Background}

In integrable quantum field theories in $(1+1)$-dimensional
space-time and in two-dimensional solvable lattice models of statistical
mechanics, a fundamental role is played by the Yang-Baxter equation
\cite{B82} in the bulk and by the reflection equation
\cite{Ch84,Sk88} at the boundary.  Their natural three-dimensional
generalizations are known as the tetrahedron equation (TE) \cite{Z80} and the
three-dimensional reflection equation (3DRE) \cite{IK97}.  In this
paper, we consider the following versions among their several
formulations:
\begin{align}
R_{456}\,R_{236}\,R_{135}\,R_{124}
&= R_{124}\,R_{135}\,R_{236}\,R_{456},
\label{te}
\\
R_{124}K_{1356}R_{178} R_{258}K_{2379} K_{4689}R_{457}
&=
R_{457}K_{4689} K_{2379} R_{258}R_{178}K_{1356}R_{124}.
\label{3dre}
\end{align}
Here $R_{ijk}$ and $K_{ijkl}$ denote operators $R$ and $K$ acting on
the three and four spaces specified by their indices,
respectively.\footnote{The subscripts in \eqref{te} and \eqref{3dre} are reset independently.
Thus, coincident indices across the two equations are not meant to
represent the same spaces.  See the explanation
following \eqref{prf} for a precise account.}
For geometric as well as algebraic interpretations, and applications of
the 3DRE, see \cite[Chaps.~4, 15, and~16]{K22}.\footnote{Equation~\eqref{3dre}
corresponds to \cite[eq.~(4.3)]{K22} with the indices
$1,2,3,4,5,6,7,8,9$ replaced by $9,8,6,7,5,4,3,2,1$, respectively.}

For the TE \eqref{te}, a number of remarkable solutions have been discovered; 
see, for example, 
\cite{BB92, KMS93, BS06, BMS10, KV94, SMS96, S99, BV15, KMY23, IKT1, IKSTY, PK24} 
and the references therein.
By contrast, for the 3DRE, only a few nontrivial solutions are known to date, 
namely those in \cite{KO12, KO13, Y21, IKT1}.

The objective of this paper is to construct a new $K$-operator that satisfies 
the 3DRE \eqref{3dre} together with the companion $R$-operator from \cite{IKSTY}, 
obtained by Sun, Terashima, Yagi, and the authors.
Our approach is based on quantum cluster algebras \cite{FG09}, 
initiated in \cite{SY22} and further developed in 
\cite{IKT1, IKT2, IKSTY}.

The quantum cluster algebra underlying the $R$-operator in \cite{IKSTY} 
is defined via the so-called \emph{symmetric butterfly} (SB) quiver, 
which is associated with the Weyl group of type~A.  
In this paper we extend it in a natural manner to type~C, 
and this enlarged structure turns out to accommodate the $K$-operator as well; 
see \eqref{SB-KC} and Figure~\ref{fig:RE}.  
For convenience, we also refer to these type~C variants as SB quivers.

\subsection{New solution to 3DRE}
Let us present our solution explicitly; its validity will be established in 
\S\ref{ss:mr} and \S\ref{ss:ef}.  
A guide to the derivation of the final forms given below is available in~\S\ref{ss:ef}.

For $i=1,2,\ldots,9$, let $(u_i,w_i)$ be canonical variables satisfying
\begin{align}\label{cv-C_3}
[u_i,w_j]=\hbar\,\gamma_i\,\delta_{ij},
\qquad
[u_i,u_j]=[w_i,w_j]=0,
\qquad
(\gamma_1,\ldots,\gamma_9)=(1,1,2,1,1,2,1,1,2),
\end{align}
where $\hbar$ is a parameter.  
We set $q=\e^\hbar$ and assume that $q$ is generic throughout.  
Let $(a_i,b_i,c_i,d_i,e_i)_{i=1,\ldots,9}$ be parameters subject to the following
constraints:
\begin{align}
&a_i +b_i+c_i+d_i+e_i=0 \quad (i=1,\ldots, 9),
\label{cd10}
\\
&a_i=c_i \quad (i=1,2,4,5,7,8),
\label{cd20}
\\
&\begin{pmatrix}
-a_3+c_3 \\
-a_6+c_6  \\
-a_9+c_9 
\end{pmatrix} = 
\begin{pmatrix}
 c_1+c_5  \\
 c_4+c_8  \\
 -c_2+c_7
 \end{pmatrix}
 =
 \begin{pmatrix} 
c_2+c_7\\
-c_1+c_5\\
-c_4+c_8
\end{pmatrix}.
\label{cd30}
\end{align}

The $R$-operator obtained in \cite{IKSTY} is given as follows (cf. \eqref{pfinal}, \eqref{Rf2}):
\footnote{The operator $R_{ijk}$ 
used here corresponds to $R_{kji}$ in \cite{IKSTY}.}
\begin{align}
R_{ijk}
&= \Psi_q(\e^{-d_k-c_j-b_i+u_i+u_k+w_i-w_j+w_k})^{-1}
\Psi_q(\e^{d_k+c_j+b_i+e_i-u_k+u_i-w_i+w_j-w_k})P_{ijk}
\nonumber \\
&\quad\cdot
\Psi_q(\e^{d_i+e_i+a_j+b_k+u_i-u_k-w_i+w_j-w_k})^{-1}
\Psi_q(\e^{-d_i-a_j-b_k+u_i+u_k-w_i+w_j-w_k}),
\label{rf}
\\
P_{ijk}
&=
\e^{\tfrac{1}{\hbar}(u_k-u_j)w_i}
\e^{\tfrac{e_j-e_k}{2\hbar}(w_k-w_j-w_i)}
\e^{\tfrac{-b_j+d_j+b_k-d_k}{2\hbar}(u_i-u_j)}\rho_{jk}.
\label{prf}
\end{align}
Here $\Psi_q$ denotes the quantum dilogarithm defined in \eqref{Psiq}.
The operator $\rho_{jk}$ exchanges the canonical variables according to
$\rho_{jk}(u_m,w_m)=(u_{m'},w_{m'})\,\rho_{jk}$, where $m=m'$ unless
$(m,m')=(j,k)$ or $(k,j)$, and leaves all parameters unchanged.  The
operator $\rho_{jl}$ appearing in \eqref{pkf} below is defined in the
same manner. 
The $R$-operator $R_{ijk}$ is known \cite{IKSTY} to satisfy the
tetrahedron equation 
$R_{lmn}\,R_{jkn}\,R_{ikm}\,R_{ijl}
 = R_{ijl}\,R_{ikm}\,R_{jkn}\,R_{lmn}$
 for any distinct indices $i,j,k,l,m,n$ from $\{1,2,4,5,7,8\}$. %
\footnote{This property already holds under the constraint \eqref{cd10} alone.
The additional conditions \eqref{cd20} and \eqref{cd30} are required only
for the 3DRE.}
The above $R_{ijk}$ corresponds to $R^{(+)}_{ijk}$ in
\eqref{R-+-+} and to
$\overline{R}^{(-)}_{ijk}$ in
\eqref{Rb-+-+}, which coincide under
the conditions \eqref{cd10}--\eqref{cd30}; see \eqref{srp2}.
A remarkable feature of $P_{ijk}$ is that it satisfies the tetrahedron equation by itself, 
$P_{lmn}\,P_{jkn}\,P_{ikm}\,P_{ijl}
 = P_{ijl}\,P_{ikm}\,P_{jkn}\,P_{lmn}$.

Let us now turn to the $K$-operator.  It involves ten quantum
dilogarithms, reflecting the ten mutations of the SB quiver in Figure \ref{K-SB-mutation}.
Explicitly, it is given as follows (cf.~\eqref{pf24}, \eqref{Kdf}):
\begin{align}
\begin{split}
K_{ijkl}&=~\Psi _q\left(\e^{-b_i-c_j-d_k+u_i+u_k+w_i-w_j+w_k}\right)^{-1}
\\ & \qquad \cdot 
\Psi_{q^2}\left(\e^{-a_k-b_j-d_l+u_j+u_l+w_j-2 w_k+w_l}\right)^{-1}
\\ & \qquad \cdot \Psi_q\left(\e^{2a_k-b_i+b_k-c_j+u_i-u_k+w_i-w_j+w_k}\right)^{-1}
\\ & \qquad \cdot \Psi_{q^2}\left(\e^{2a_k+a_i-2 c_j-d_j+d_l+u_j-u_l-w_j+2 w_k-w_l}\right)
\\ & \qquad \cdot \Psi_q\left(\e^{-2a_i+c_j-d_i+d_k+u_i-u_k-w_i+w_j-w_k}\right)
\\ & \qquad \cdot \Psi_{q^2}\left(
\e^{2a_k-2 b_i-b_j+2 b_k+b_l+a_i-2 c_j+2 c_l+2 u_i+u_j-2 u_k-u_l+2w_i-w_j+w_l}\right)^{-1}
\\ & \qquad \cdot \Psi_q\left(\e^{-b_i-b_j+b_k+b_l+a_i-c_j+2c_l+u_i+u_j-u_k-u_l+w_i-w_k+w_l}\right)^{-1}
\\ & \qquad \cdot \Psi _{q^2}\left(\e^{2a_k+a_i-2c_j-d_j+d_l+u_j-u_l-w_j+2 w_k-w_l}\right)^{-1}
\\ & \qquad \cdot \Psi_q\left(\e^{-a_i-b_k-c_j-d_i-d_j+d_l+u_i+u_j+u_k-u_l-w_i+w_k-w_l}\right)
\\ & \qquad \cdot \Psi_{q^2}\left(\e^{-a_k-b_j-d_l+u_j+u_l+w_j-2 w_k+w_l}\right)P^K_{ijkl},
\end{split}
\label{kf}\\
P^K_{ijkl} &= \e^{\frac{1}{2\hbar}\{(u_j-u_l)(a_j-a_l-2w_i)
-(b_j-d_j-b_l+d_l)(u_i-u_k)-\tfrac12(e_j-e_l)(w_j-w_l)\}}\rho_{jl}.
\label{pkf}
\end{align}
This corresponds to the specific expressions in \eqref{sol243} and
\eqref{K243} for the type $\rho_{24}$ solution in \S\ref{ad24}, with the
conditions \eqref{cd10}--\eqref{cd30} taken into account.

The  main result of this paper is the following (Theorem~\ref{th:main}):

\vspace{0.2cm}\noindent
\textbf{Theorem}.
{\em Under the conditions \eqref{cd10}--\eqref{cd30}, 
the $R$-operator \eqref{rf}--\eqref{prf} 
and the $K$-operator in \eqref{kf}--\eqref{pkf} satisfy the 
three-dimensional reflection equation\eqref{3dre}.}

\vspace{0.2cm}
Each operator $R_{ijk}$ and $K_{ijkl}$ decomposes into a dilogarithm
part and a monomial part, denoted by $P_{ijk}$ and $P^K_{ijkl}$,
respectively.  The monomial factors $P_{ijk}$ and $P^K_{ijkl}$ can be
moved to any position among the dilogarithm factors by transforming the
canonical variables according to $\eta^{(+)}_{ijk}$~\eqref{uw-2} and
$\eta^K_{ijkl}$~\eqref{tau24uw} under their adjoint actions.
As in the case of the TE, the 3DRE admits corresponding decomposition:
it separates into identities for the dilogarithm parts and for the
monomial parts.  These identities make sense both in the set of
formal Laurent series introduced in \S\ref{sec:qW} and in the
group $N(C_3)\rtimes S(C_3)$ described around
\eqref{NdC3}, respectively.
The monomial parts $P_{ijk}$ and $P^K_{ijkl}$ satisfy the 3DRE among
themselves (Proposition~\ref{prop:RE-P}).
The remaining dilogarithm identity involves
$3\times 10 + 4\times 4 = 46$ quantum dilogarithms on each side,
with $31$ given by $\Psi_q$ and the other $15$ by $\Psi_{q^2}$.
Among the $5 \times 9 = 45$ parameters $(a_i,b_i,c_i,d_i,e_i)$ there are $21$ relations in \eqref{cd10}--\eqref{cd30}.  
Hence our solution to the 3DRE contains 
$5\times 9 - 21 = 24$ free parameters, modulo constant shifts of the 
canonical variables $u_i$ and $w_i$.

Our solution can be evaluated in various representations of the canonical variables.
In fact, the corresponding images of the $R$-operator, including its quantum double
version, are known to cover the significant solutions to the TE obtained in 
\cite{KV94, BS06, BMS10, KMY23, IKT1} through suitable specializations of the parameters;  
see \cite[Table~1.1]{IKSTY}.

For the $K$-operator newly constructed in this paper, it contains the one in \cite{IKT1} 
as a special case.  
This will be demonstrated in \S\ref{sec:fg} as a consequence of the reduction from the 
SB quiver to the Fock-Goncharov (FG) quiver used in \cite{IKT1}.  
We further expect that, by an appropriate tuning of the parameters, our $K$-operator 
also reproduces the solution in \cite{KO12}, which was the first nontrivial solution 
to the 3DRE derived from the representation theory of the quantized coordinate ring 
$A_q(C_2)$.

\subsection{Method}
Our strategy for constructing the $K$-operator parallels the earlier
works \cite{IKT1,IKT2,IKSTY}, and may be sketched as follows.
We begin with the quantum cluster algebra associated with the SB quiver,
whose vertices lie on all of the crossings, the reflection points and the faces of 
the wiring diagram for the Weyl group $W(C_2)$ \eqref{SB-KC}.  It  is
naturally embedded into the one associated with $W(C_3)$ shown in
Figure \ref{fig:RE}.

Corresponding to the sequence of transformations that carries a reduced
expression of the longest element of $W(C_3)$ to the ``most distant"
reversed one, we obtain a nontrivial identity among the cluster
transformations of the quantum $Y$-variables.  It already takes the
form of the 3DRE \eqref{REiso}, whose building blocks
$\widehat{R}_{ijk}$ and $\widehat{K}_{ijkl}$ correspond to the cubic
and the quartic Coxeter relations, respectively.

Reflecting the decomposition of mutations (cf.~\eqref{mu-pm}), 
the cluster transformations are decomposed as
$\widehat{R}_{ijk}=\mathrm{Ad}(R^\Psi_{ijk})\circ\tau_{ijk}$ and
$\widehat{K}_{ijkl}=\mathrm{Ad}(K^\Psi_{ijkl})\circ\tau^K_{ijkl}$,
where $R^\Psi_{ijk}$ and $K^\Psi_{ijkl}$ are products of quantum
dilogarithms in quantum $Y$-variables, and $\tau_{ijk}$ and $\tau^K_{ijkl}$ are the monomial
parts (see \eqref{tau-+-+} and \eqref{tauK24}).

A key to the next step is to translate the quantum $Y$-variables into the
canonical variables \eqref{cv-C_3} via a ring homomorphism.  
We then seek a distinguished situation in which the monomial parts
$\tau_{ijk}$ and $\tau^K_{ijkl}$ are realized as adjoint actions
$\mathrm{Ad}(P_{ijk})$ and $\mathrm{Ad}(P^K_{ijkl})$ for suitable
$P_{ijk}$ and $P^K_{ijkl}$, respectively.
This effectively turns the cluster transformations
$\widehat{R}_{ijk}$ and $\widehat{K}_{ijkl}$
into totally adjoint operators 
$\mathrm{Ad}(R^\Psi_{ijk}  P_{ijk})$ and
$\mathrm{Ad}(K^\Psi_{ijkl} P^K_{ijkl})$,
where $R^\Psi_{ijk}$ and $K^\Psi_{ijkl}$ are now written in terms of the canonical variables.
We furthermore postulate that the combined operators 
$R_{ijk}:=R^\Psi_{ijk}P_{ijk}$ and $K_{ijkl}:=K^\Psi_{ijkl}P^K_{ijkl}$
{\em themselves} satisfy the 3DRE\ (!).
These requirements are highly nontrivial, and there is no general result
ensuring either the existence of such $P_{ijk}$ and $P^K_{ijkl}$, or the
well-definedness of the products $R^\Psi_{ijk}$ and $K^\Psi_{ijkl}$ in
the framework of quantum cluster algebras.
A central highlight of our analysis is to accomplish this by fully
exploiting the optional sign degrees of freedom in decomposing the
mutation in \eqref{mu-pm}.
The detailed procedure is summarized in the introduction of
\S\ref{sec:main}.

\subsection{Outlook}
There are a number of future problems stemming from the result of this
paper.

One natural direction is to compute matrix elements of the $K$-operator
either in the $u$-diagonal or $w$-diagonal representations, or in their
quantum-double versions, in parallel with the analysis carried out
for the $R$-operator in \cite{IKSTY}.  In such representations, the
equality of the two sides of the 3DRE is expected to admit an
interpretation as a duality involving $46$ $q$-factorials 
corresponding to \eqref{L46} = \eqref{R46}.  
See \cite[Th.~6.2]{IKT1} for a basic instance of this type of computation.

The $R$-operator is also characterized as the unique (up to normalization)
solution to the so-called $RLLL$ relation 
$R(L_{12}L_{13}L_{23})=(L_{23}L_{13}L_{12})R$ \cite[Th.~7.1]{IKSTY}.
Here $L$ denotes a three-dimensional $L$-operator \cite{BS06,BMS10},
which may be viewed as defining a quantized six-vertex model \cite{KMY23,IKTY}.
In a similar spirit, it is natural to expect that the $K$-operator 
constructed in this paper is characterized by a {\em quantized reflection equation}
(cf.~\cite{KP18}, \cite[Sec.~4.4]{K22}) of the form 
$(L_{12}G_2L_{21}G_1)\,K = K\,(G_1L_{12}G_2L_{21})$ for a suitable operator $G$.
Once such a quantized reflection equation is obtained,
it yields an infinite family of solutions to the {\em usual} reflection equation in two
dimensions, presented in a {\em matrix product form} \cite{KP18}.

In \S \ref{sec:k}, we have obtained the $K$-operators of types $\rho_{24}$ and
$\rho_{13}$.  The subsequent analysis focuses on the former case, whose
associated SB quiver \eqref{SB-KC} is of type~C and contains
weight-two nodes along the top boundary edges.
The $K$-operator of type
$\rho_{13}$ corresponds instead to type~B and warrants a parallel
investigation. See Remark \ref{re:b}. 
In particular, an appropriate specialization of this
operator is expected to reproduce the solution derived from the
quantized coordinate ring $A_q(B_2)$ in \cite{KO13}.

\subsection{Outline of the paper}

In \S \ref{sec:qca}, 
we summarize the basic facts on quantum cluster algebra.
A key role is played by Theorem \ref{thm:period} 
concerning the cluster transformations along the mutation sequences.

In \S \ref{sec:cm}, 
we recall the SB quiver used in \cite{IKSTY} for the
$R$-operator, and introduce its extension adapted to the $K$-operator.
The associated cluster transformations and their 3DRE are presented.

In \S \ref{sec:Rop}, 
we review the construction of the $R$-operator in
\cite{IKSTY}, where Proposition \ref{prop:ns} and Theorem \ref{thm:TE} 
provide a slight refinement of its description in terms of $N(A_3)\rtimes S(A_3)$, 
$N'(A_3)\rtimes S(A_3)$ and the set of Laurent series $\mathcal{L}_\gamma$ introduced in \S \ref{sec:qW}.

In \S \ref{sec:k}, 
we obtain the $K$-operators of types $\rho_{24}$ and
$\rho_{13}$ by analyzing the condition under which $\tau^K_{ijkl}$ is
realized as an adjoint operator. The rest of this paper focuses on the case
$\rho_{24}$; see Remark \ref{re:b}.

In \S \ref{sec:main}, 
building on the results in the preceding
sections, we establish the 3DRE under several additional constraints
required for the parameters.  We also give a detailed proof of the
well-definedness of the products of quantum dilogarithms appearing in
the 3DRE, both as Laurent series in quantum $Y$-variables and in terms
of $q$-Weyl algebra generators.  The derivation of the explicit formulas
\eqref{rf}--\eqref{pkf} is illustrated in \S \ref{ss:ef}.

In \S \ref{sec:fg}, we show how the $K$-operator in \cite{IKT1},
associated with the FG quiver, is recovered by an
appropriate specialization of the parameters.

Appendix \ref{app:pK} provides a proof of Proposition \ref{prop:K24}.
Appendix \ref{app2} presents explicit formulas of the $K$-operators 
for general allowed choices of sign sequence.
Due to Proposition \ref{prop:K24}, they are just different guises of the same operator.
Appendix~\ref{app:ff} collects the data used in the proof of the
well-definedness of the dilogarithm identities.
Appendix \ref{app:R}, parallel to \S \ref{sec:fg}, explains how the $R$-operator
for the FG quiver \cite{IKT1} is obtained as a specialization of the one in this paper.

\subsection*{Acknowledgement}
The authors thank Xiaoyue Sun, Yuji Terashima and Junya Yagi for collaborations in the previous works.
RI is supported by JSPS KAKENHI Grant Number 23K03048. 
AK is supported by JSPS KAKENHI Grant Number 24K06882.

\section{Basics of quantum cluster mutations and notations}\label{sec:qca}

\subsection{Quantum cluster mutation}

We recall the definition of quantum cluster mutation introduced by
Fock and Goncharov~\cite{FG06}, along with an important property
of periodic mutation sequences used in this paper.

Let $I$ be a finite index set, and let
$B=(b_{ij})_{i,j\in I}$ be a skew-symmetrizable exchange matrix with
entries in $\tfrac{1}{2}\mathbb{Z}$; that is, there exists a diagonal matrix $d=\mathrm{diag}(d_j)_{j\in I}$ of positive integers such that
$\widehat{B}=(\widehat{b}_{ij})_{i,j\in I}:=B\,d=(b_{ij}d_j)_{i,j\in I}$
is skew-symmetric. We assume $\gcd(d_j\mid j\in I)=1$, so that $d$ is
uniquely determined by~$B$.

Define the subset
$I_0 := \{\,i\in I \mid b_{ij}\notin\mathbb{Z}\text{ or }b_{ji}\notin\mathbb{Z}\,\}$.
Let $\mathcal{Y}(B)$ denote the skewfield generated by the quantum
$Y$-variables $Y=(Y_i)_{i\in I}$ subject to the $q$-commutation
relations
\begin{align}\label{q-Y}
  Y_i Y_j = q^{2\widehat{b}_{ij}}\,Y_j Y_i.  
\end{align}
We write $q_i:=q^{d_i}$.  
The data $(B,d,Y)$ is called a \emph{quantum $Y$-seed}, and each $Y_i$
is referred to as a \emph{quantum $y$-variable}.

For $k \in I \setminus I_0$, the quantum mutation $\mu_k$ transforms
$(B,d,Y)$ into $(B',d',Y') := \mu_k(B,d,Y)$ defined by
\begin{align}\label{q-mu}
  b_{ij}' &=
  \begin{cases}
    -b_{ij}, & i=k \text{ or } j=k,\\[2mm]
    \displaystyle b_{ij}
      + \frac{|b_{ik}|\,b_{kj} + b_{ik}\,|b_{kj}|}{2},
      & \text{otherwise},
  \end{cases}\\[2mm]
  \label{d-mu}
  d_i' &= d_i,\\[2mm]
  \label{Y-mu}
  Y_i' &=
  \begin{cases}
    Y_k^{-1}, & i=k,\\[2mm]
    \displaystyle
    Y_i \prod_{j=1}^{|b_{ik}|}
    \bigl(1 + q_k^{\,2j-1} Y_k^{-\mathrm{sgn}(b_{ik})}\bigr)^{-\mathrm{sgn}(b_{ik})},
      & i \neq k.
  \end{cases}
\end{align}
In what follows, we abbreviate $(B,d, Y)$ to $(B,Y)$, since $d$ is
invariant under mutations. The mutation $\mu_k$ induces an isomorphism of
skewfields $\mu_k^\ast:\mathcal{Y}(B') \to \mathcal{Y}(B)$, where
$\mathcal{Y}(B')$ is the skewfield generated by $Y'_i$ subject to the
relations $Y'_i Y'_j = q^{2\widehat{b}'_{ij}} Y'_j Y'_i$.

To visualize the exchange matrix $B$, we use weighted quivers.
For the pair $(B,d)$, define the skew-symmetric matrix
$\sigma = (\sigma_{ij})_{i,j \in I}$ by
$\sigma_{ij} = b_{ij}\,\gcd(d_i,d_j)/d_i$.
In this paper, we only encounter cases where $\sigma_{ij}$ is integral
or $\pm \tfrac{1}{2}$.
We determine the weighted quiver $Q=(\sigma,d)$, without one-loops or
two-cycles, as follows. The vertex set of $Q$ is $I$, and each vertex
$i \in I$ carries a weight~$d_i$.
A vertex of weight~$1$ is represented by a circle, while a vertex of
weight~$d_i>1$ is represented by a circle containing~$d_i$ inside.
When $\sigma_{ij}$ is integral, we draw ordinary arrows
$\longrightarrow$ so that
$\sigma_{ij} = \#\{\text{arrows from $i$ to $j$}\}
 - \#\{\text{arrows from $j$ to $i$}\}$.
When $\sigma_{ij}=1/2$ (resp.~$\sigma_{ij}=-1/2$),
we draw a dashed arrow $\dashrightarrow$ from $i$ to $j$
(resp.~from $j$ to $i$).
When $B$ is skew-symmetric, we have $d_i=1$ for all $i\in I$.

\subsubsection*{Decomposition of quantum mutation}

The quantum mutation can be decomposed into two parts --- a monomial part
and an automorphism part~\cite{FG09} --- in two equivalent ways
as described in~\cite{Ke11}.
For $\ve \in \{+,-\}$, define an isomorphism
$\tau_{k,\ve}$ of skewfields by
\begin{align}\label{mu-mono}
\tau_{k,\ve} :~ \mathcal{Y}(B') \to \mathcal{Y}(B);
\quad
Y'_i \mapsto
\begin{cases}
  Y_k^{-1}, & i = k,\\[2mm]
  q^{-\widehat{b}_{ik}[\ve b_{ik}]_+}\,
  Y_i\,Y_k^{[\ve b_{ik}]_+}, & i \neq k,
\end{cases}
\end{align}
where $[a]_+ := \max(0,a)$.
Then we have
\begin{align}\label{mu-pm}
\mu_k^\ast
 = \mathrm{Ad}\,\Psi_{q_i}(Y_k)\circ\tau_{k,+}(Y'_i)
 = \mathrm{Ad}\,\Psi_{q_i}(Y_k^{-1})^{-1}\circ\tau_{k,-}(Y'_i),
\end{align}
where $\mathrm{Ad}(X)(Y) = XYX^{-1}$,
and $\Psi_q(U)$ denotes the quantum dilogarithm, defined by
\begin{align}\label{Psiq}
\Psi_q(U) = \frac{1}{(-qU;q^2)_\infty},
\qquad
(z;q)_\infty = \prod_{k=0}^{\infty}(1-zq^k).
\end{align}
The function $\Psi_q(U)$ admits the series expansion
\begin{align}\label{dilog-sum}
\Psi_q(U)
 = \sum_{n=0}^\infty \frac{(-qU)^n}{(q^2;q^2)_n},
\qquad
\Psi_q(U)^{-1}
 = \sum_{n=0}^\infty \frac{q^{n^2}U^n}{(q^2;q^2)_n},
\qquad
(z;q)_n = \frac{(z;q)_\infty}{(zq^n;q)_\infty}\quad (n\in\mathbb{Z}),
\end{align}
and it is an element of the formal power series algebra
$\mathbb{Q}(q)[[U]]$.
Equation~\eqref{mu-pm} follows from one of the fundamental properties of
$\Psi_q(U)$,
\begin{align}\label{dilog-rec}
\Psi_q(q^2U)\,\Psi_q(U)^{-1} = 1 + qU.
\end{align}

\subsubsection*{Quantum torus algebra}

The quantum torus algebra $\mathcal{T}(B)$ associated with $B$ is the
$\mathbb{Q}(q)$-algebra generated by noncommutative variables
$\rY^\alpha~(\alpha \in \mathbb{Z}^I)$ satisfying the relations
\begin{align}
q^{\langle \alpha,\beta \rangle}\,\rY^\alpha \rY^\beta
 = \rY^{\alpha+\beta},
\end{align}
where $\langle~,~\rangle$ is the skew-symmetric bilinear form defined by
$\langle \alpha,\beta\rangle = -\langle \beta,\alpha\rangle
 = -\,\alpha \cdot \widehat{B}\beta$.
Let $e_i$ be the standard unit vector of $\mathbb{Z}^I$, and write
$\rY_i$ for $\rY^{e_i}$. Then
$\rY_i \rY_j = q^{2\widehat{b}_{ij}} \rY_j \rY_i$.
Identifying $\rY_i$ with $Y_i$, we recover~\eqref{q-Y}.
The monomial part of $\mu_k^\ast$, $\tau_{k,\ve}$ in~\eqref{mu-mono},
naturally induces an isomorphism of quantum torus algebras, written as
\begin{align}\label{torus-iso}
\tau_{k,\ve} : \mathcal{T}(B') \to \mathcal{T}(B);
\quad
\rY'_i \mapsto
\begin{cases}
\rY_k^{-1}, & i = k,\\[2mm]
\rY^{e_i + e_k [\ve b_{ik}]_+}, & i \neq k.
\end{cases}
\end{align}
 
\subsection{Periodicity and quantum dilogarithm identity}

Let $\mathbb{P}(u)=\mathbb{P}_\trop(u_1,u_2,\ldots,u_p)
 := \{\prod_{i=1}^{p} u_i^{a_i}\mid a_i \in \mathbb{Z}\}$
be the tropical semifield of rank~$p$, equipped with addition~$\oplus$
and multiplication~$\cdot$ defined by
$$
\prod_{i=1}^{p} u_i^{a_i} \oplus \prod_{i=1}^{p} u_i^{b_i}
 = \prod_{i=1}^{p} u_i^{\min(a_i,b_i)},
 \qquad
\prod_{i=1}^{p} u_i^{a_i} \cdot \prod_{i=1}^{p} u_i^{b_i}
 = \prod_{i=1}^{p} u_i^{a_i+b_i}.
$$
For $v = \prod_{i\in I} u_i^{a_i} \in \mathbb{P}(u)$, we write
$v = u^\alpha$ with $\alpha=(a_i)_{i\in I}\in\mathbb{Z}^{I}$.
If $\alpha \in \mathbb{Z}_{\ge0}^{I}$ (resp.~$\alpha \in \mathbb{Z}_{\le0}^{I}$),
we say that $v$ is \emph{positive} (resp.~\emph{negative}).

For a finite set $I$, let $\mathbb{P}(u)$ be the tropical semifield of
rank~$|I|$.
For a tropical $y$-seed $(B,y)$, where
$B=(b_{ij})_{i,j\in I}$ and $y=(y_i)_{i\in I}\in\mathbb{P}(u)^{I}$,
and for $k\in I\setminus I_0$, the mutation
$\mu_k(B,y) = (B',y')$ is defined by~\eqref{q-mu} together with
\begin{equation}\label{trop-mu}
y_i' =
  \begin{cases}
    y_k^{-1}, & i = k,\\[2mm]
    y_i \cdot (1 \oplus y_k^{-\mathrm{sgn}(b_{ik})})^{-b_{ik}}, & i \neq k.
  \end{cases}
\end{equation}
An important property of tropical $y$-variables is
\emph{sign coherence}:
for any tropical $y$-variable $y_i' = u^{\alpha'}$
obtained by mutating the initial seed $(B,y)$
with $y = (u_i)_{i\in I}$,
the exponent vector $\alpha'$ is either positive or negative.
The sign of $\alpha'$ is called the \emph{tropical sign} of~$y_i'$. 

The symmetric group $\mathfrak{S}_I$ naturally acts on tropical
$y$-seeds by
$$
  \sigma : (b_{ij},y_i) \mapsto
  (b_{\sigma^{-1}(i),\sigma^{-1}(j)},\,y_{\sigma^{-1}(i)}),
  \qquad \sigma \in \mathfrak{S}_I,
$$
and acts similarly on quantum $Y$-seeds.
For a sequence
$\mathbf{i} = (i_1,i_2,\ldots,i_L) \in I^L$, define the composition of
mutations
$\mu_{\mathbf{i}} := \mu_{i_L}\mu_{i_{L-1}}\cdots\mu_{i_2}\mu_{i_1}$,
and consider the resulting sequences of tropical $y$-seeds and quantum
$y$-seeds starting from $(B,u)$ and $(B,Y)$, respectively:
\begin{align}
\label{seq-ty}
&(B,u) = (B^{(1)},y^{(1)}) \xrightarrow{\mu_{i_1}}
(B^{(2)},y^{(2)}) \xrightarrow{\mu_{i_2}}
\cdots \xrightarrow{\mu_{i_L}} (B^{(L+1)},y^{(L+1)}),\\
\label{seq-Y}
&(B,Y) = (B^{(1)},Y^{(1)}) \xrightarrow{\mu_{i_1}}
(B^{(2)},Y^{(2)}) \xrightarrow{\mu_{i_2}}
\cdots \xrightarrow{\mu_{i_L}} (B^{(L+1)},Y^{(L+1)}).
\end{align}
For $\sigma \in \mathfrak{S}_I$, we say that the sequence
$\mathbf{i}$ is a \emph{$\sigma$-period} of $(B,u)$ if
$\sigma^{-1}(b^{(L+1)},y^{(L+1)}) = (b,y)$.
The $\sigma$-period of $(B,Y)$ is defined analogously.
For an exchange matrix $B$, we refer to a sequence of mutations for~$B$
together with a permutation of~$I$ as a \emph{mutation sequence}.
  
The following theorem is obtained by combining the synchronicity \cite{N21} among $x$-seeds, 
$y$-seeds and tropical $y$-seeds, together with the synchronicity 
between classical and quantum seeds \cite[Lemma 2.22]{FG09b}, \cite[Proposition 3.4]{KN11}.
 
\begin{theorem}\label{thm:period}
For an exchange matrix $B$ and a mutation sequence $\nu$ for $B$, the following two statements are equivalent:
\begin{itemize}

\item[(1)]
For a tropical $y$-seed $(B,y)$, it holds that $\nu (B,y) = (B,y)$. 

\item[(2)]
For a quantum $Y$-seed $(B,Y)$, it holds that $\nu (B,Y) = (B,Y)$. 

\end{itemize}
\end{theorem}

Establishing the periodicity of a tropical $y$-seed is much
easier than that of its quantum counterpart. In this paper, we employ
the theorem so that statement~(2) follows from~(1).

For $t=1,\ldots,L+1$, let $\rY_i^{(t)}~(i\in I)$ denote the generators
of the quantum torus $\mathcal{T}(B^{(t)})$.
The quantum $y$-variables in~\eqref{seq-Y} can be expressed as
\begin{align}\label{ad-tau-decomp}
\begin{split}
Y_i^{(t+1)}
&= \Ad\!\bigl(\Psi_{q_{i_1}}(\rY_{i_1}^{(1)\delta_1})^{\delta_1}\bigr)
   \tau_{i_1,\delta_1}\cdots
   \Ad\!\bigl(\Psi_{q_{i_t}}(\rY_{i_t}^{(t)\delta_t})^{\delta_t}\bigr)
   \tau_{i_t,\delta_t}(\rY_i^{(t+1)})\\
&=\Ad\!\bigl(
   \Psi_{q_{i_1}}(\rY^{\delta_1\beta_1})^{\delta_1}
   \cdots
   \Psi_{q_{i_t}}(\rY^{\delta_t\beta_t})^{\delta_t}
   \bigr)
   \circ
   \tau_{i_1,\delta_1}\cdots\tau_{i_t,\delta_t}(\rY_i^{(t+1)}),
\end{split}
\end{align}
where $\beta_r \in \mathbb{Z}^I$ is determined by
$\rY^{\beta_r}
 = \tau_{i_1,\delta_1}\cdots\tau_{i_{r-1},\delta_{r-1}}
   (\rY_{i_r}^{(r)})$.
In particular, $\beta_1 = e_{i_1}$.
In this way, the isomorphism
$\mu_{\mathbf{i}}^\ast : \mathcal{Y}(B^{(L+1)}) \to \mathcal{Y}(B^{(1)})$
is decomposed into the monomial part
\begin{align}\label{mono-p}
\tau_{i_1,\delta_1}\tau_{i_2,\delta_2}\cdots\tau_{i_L,\delta_L}
 : \mathcal{Y}(B^{(L+1)}) \to \mathcal{Y}(B^{(1)}),
\end{align}
and the dilogarithm part
\begin{align}\label{dilog-p}
\Ad\!\bigl(
 \Psi_{q_{i_1}}(\rY^{\delta_1\beta_1})^{\delta_1}
 \Psi_{q_{i_2}}(\rY^{\delta_2\beta_2})^{\delta_2}
 \cdots
 \Psi_{q_{i_L}}(\rY^{\delta_L\beta_L})^{\delta_L}
 \bigr)
 : \mathcal{Y}(B^{(1)}) \to \mathcal{Y}(B^{(1)}).
\end{align}

Suppose that $\mathbf{i} = (i_1,i_2,\ldots,i_L)$ is a $\sigma$-period of $(B,Y)$.
For any sign sequence $(\delta_t)_{t=1,\ldots,L}$ with
$\delta_t \in \{+,-\}$, we have
\begin{align}\label{s-per}
\begin{split}
\Ad\!\bigl(
\Psi_{q_{i_1}}(\rY^{\delta_1\beta_1})^{\delta_1}
\Psi_{q_{i_2}}(\rY^{\delta_2\beta_2})^{\delta_2}
\cdots
\Psi_{q_{i_L}}(\rY^{\delta_L\beta_L})^{\delta_L}
\bigr)
\circ
\tau_{i_1,\delta_1}\tau_{i_2,\delta_2}\cdots\tau_{i_L,\delta_L}\sigma
= \mathrm{id}.
\end{split}
\end{align}
If each $\delta_t$ is chosen to be the tropical sign of
$y_{i_t}^{(t)}$, then by Theorem~\ref{thm:period} and the discussion in
\cite{Ke11,KN11}, equation~\eqref{s-per} decomposes into two identities:
\begin{align}
&\tau_{i_1,\delta_1}\tau_{i_2,\delta_2}\cdots\tau_{i_L,\delta_L}\sigma
 = \mathrm{id},\\[2mm]
&\Psi_{q_{i_1}}(\rY^{\delta_1\beta_1})^{\delta_1}
 \Psi_{q_{i_2}}(\rY^{\delta_2\beta_2})^{\delta_2}
 \cdots
 \Psi_{q_{i_L}}(\rY^{\delta_L\beta_L})^{\delta_L}
 = 1.
\end{align}
For a general sign sequence, there is no guarantee that
\eqref{s-per} can be separated into two such identities.

\subsection{$q$-Weyl algebras}\label{sec:qW}

Fix a positive integer $p$ and a nonzero complex number~$\hbar$.
For $\gamma = (\gamma_i)_{i=1,\ldots,p} \in \mathbb{Z}_{>0}^p$,
let $(u_i,w_i)~(i=1,2,\ldots,p)$ be canonical pairs satisfying
\begin{align}\label{uw-gamma}
[u_i,w_j] = \hbar\,\gamma_i\,\delta_{ij},
\qquad
[u_i,u_j] = [w_i,w_j] = 0.
\end{align}
We write
$\mathbf{u} = (u_1,u_2,\ldots,u_p,w_1,w_2,\ldots,w_p)$.
Set $q = \e^{\hbar}$, and let $\mathcal{W}_\gamma$ denote the algebra
over~$\mathbb{C}(q^{1/2})$ generated by the $q$-Weyl pairs
$\mathrm{e}^{\pm u_i}, \mathrm{e}^{\pm w_i}$ $(i=1,2,\ldots,p)$,
satisfying the $q$-commutation relations
\begin{align}\label{qW-gamma}
\begin{split}
&\mathrm{e}^{u_i}\mathrm{e}^{w_j}
 = q^{\gamma_i\,\delta_{ij}}\,
   \mathrm{e}^{w_i}\mathrm{e}^{u_i},\\
&\mathrm{e}^{u_i}\mathrm{e}^{u_j}
 = \mathrm{e}^{u_j}\mathrm{e}^{u_i},
\qquad
\mathrm{e}^{w_i}\mathrm{e}^{w_j}
 = \mathrm{e}^{w_j}\mathrm{e}^{w_i}.
\end{split}
\end{align}

We denote by $\mathcal{L}_\gamma$ the set of formal Laurent series in
$\mathrm{e}^{u_i}$ and $\mathrm{e}^{w_i}$ satisfying~\eqref{qW-gamma},
expressed as
\begin{align}\label{fLs}
\mathcal{L}_\gamma
 = \left\{
   \sum_{\mathbf{m}\in\mathbb{Z}^{2p}}
   f(\mathbf{m})\,\mathrm{e}^{\mathbf{m}\cdot\mathbf{u}}
   \;\middle|\;
   f(\mathbf{m})\in\mathbb{C}(q^{1/2})
   \right\},
\end{align}
with the relation
$\mathrm{e}^{u_i}\mathrm{e}^{w_j}
 = \mathrm{e}^{\frac{1}{2}[u_i,w_j]}\mathrm{e}^{u_i+w_j}$.
Note that $\mathcal{L}_\gamma$ is not closed under multiplication.

We identify $\gamma$ with the diagonal matrix
$\mathrm{diag}(\gamma_1,\gamma_2,\ldots,\gamma_p)\in GL_p(\mathbb{Z})$.
For $\alpha\in\mathbb{C}^{2p}$ and an integral matrix
$A\in GL_p(\mathbb{Z})$ such that
$\gamma^{-1}A\gamma\in GL_p(\mathbb{Z})$ and $\det A=\pm1$,
define an affine transformation $\eta_{A,\alpha}$ of $\mathbf{u}$ by
\begin{align}
\eta_{A,\alpha}:\;
\mathbf{u}\mapsto\tilde{A}\mathbf{u}+\alpha;
\qquad
\tilde{A}
 = \begin{pmatrix}
     A & O\\[1mm]
     O & \gamma(A^{-1})^{\!T}\gamma^{-1}
   \end{pmatrix}
   \in SL_{2p}(\mathbb{Z}).
\end{align}
It is easy to verify that $\eta_{A,\alpha}$ acts as a canonical
transformation of variables, preserving the commutation relations
\eqref{uw-gamma}.
This $\eta_{A,\alpha}$ defines an isomorphism of
$\mathcal{W}_\gamma$, and we denote this isomorphism by the same symbol.

\begin{lemma}\label{lem:L}
For
$F=\sum_{\mathbf{m}\in\mathbb{Z}^{2p}}
 f(\mathbf{m})\,\mathrm{e}^{\mathbf{m}\cdot\mathbf{u}}
 \in\mathcal{L}_\gamma$,
define the induced transformation $\eta_{A,\alpha}^\ast$ by
$$
  \eta_{A,\alpha}^\ast(F)
  = \sum_{\mathbf{m}\in\mathbb{Z}^{2p}}
     f(\mathbf{m})\,
     \mathrm{e}^{\mathbf{m}\cdot(\tilde{A}\mathbf{u}+\alpha)}.
$$
Then it holds that $\eta_{A,\alpha}^\ast(F)\in\mathcal{L}_\gamma$.
In particular, $\eta_{A,\alpha}^\ast$ acts on $\mathcal{L}_\gamma$.
\end{lemma}

\begin{proof}
The claim follows from the fact that the matrix~$\tilde{A}$
defines an injection
$\tilde{A}:\mathbb{Z}^{2p}\to\mathbb{Z}^{2p}$ by
$\mathbf{m}\mapsto\mathbf{m}\tilde{A}$.
\end{proof}

\subsection{Other notations}

For a simple Lie algebra~$\mathfrak{g}$ of rank~$\ell$, let
$W(\mathfrak{g})$ denote the corresponding Weyl group generated by
simple reflections $s_1,\ldots,s_\ell$.
A reduced expression $s_{i_1}s_{i_2}\cdots s_{i_p}$ of an element of
$W(\mathfrak{g})$ will be abbreviated as $i_1 i_2 \cdots i_p$.

When $\mathfrak{g}=A_\ell$, we consider a wiring diagram with
$\ell+1$ wires, where $s_k$ ($k=1,\ldots,\ell$) interchanges the
$k$-th and $(k+1)$-th wires counted from the bottom.

When $\mathfrak{g}=C_\ell$ or $B_\ell$, 
we consider a wiring diagram
with $\ell$ wires and a wall which reflects the wires, where
$s_k$ ($k=1,\ldots,\ell-1$) interchanges the $k$-th and $(k+1)$-th
wires from the bottom, and $s_\ell$ reflects the $\ell$-th wire at
the wall.

\section{Tetrahedron and 3D reflection equations from cluster mutations}\label{sec:cm}

\subsection{$R$-operator}
Recall the $R$-operator for the symmetric butterfly (SB) quiver
introduced in~\cite{SY22,IKSTY}.
It is constructed from the following transformation of the wiring
diagrams (shown in red) and the corresponding quivers (shown in black),
which are associated with the two reduced expressions~$121$ and~$212$
of the longest element in the Weyl group~$W(A_2)$.
\begin{align}\label{SB-R}
\begin{tikzpicture}
\begin{scope}[>=latex,xshift=0pt]
{\color{red}
\fill (1,0.5) circle(2pt) coordinate(A) node[above right]{$1$};
\fill (2,1.5) circle(2pt) coordinate(B) node[below right]{$2$};
\fill (3,0.5) circle(2pt) coordinate(C) node[above right]{$3$};
\draw [-] (0,2) to [out = 0, in = 135] (B);
\draw [-] (B) -- (C); 
\draw [-] (C) to [out = -45, in = 180] (4,0);
\draw [-] (0,1) to [out = 0, in = 135] (A); 
\draw [-] (A) to [out = -45, in = -135] (C);
\draw [-] (C) to [out = 45, in = 180] (4,1);
\draw [-] (0,0) to [out = 0, in = -135] (A); 
\draw [-] (A) -- (B);
\draw [-] (B) to [out = 45, in = 180] (4,2);
}
%
\draw (2,-0.5) circle(2pt) coordinate(1) node[below]{\scriptsize$9$};
\draw (0,0.5) circle(2pt) coordinate(2) node[below left]{\scriptsize$4$};
\draw (1,0.5) circle(2pt) coordinate(3) node[below=1pt]{\scriptsize$5$};
\draw (2,0.5) circle(2pt) coordinate(4) node[above left]{\scriptsize$6$};
\draw (3,0.5) circle(2pt) coordinate(5) node[below=1pt]{\scriptsize$7$};
\draw (4,0.5) circle(2pt) coordinate(6) node[below right]{\scriptsize$8$};
\draw (1,1.5) circle(2pt) coordinate(7) node[above left]{\scriptsize$1$};
\draw (2,1.5) circle(2pt) coordinate(8); 
\draw (1.9,1.75) node {\scriptsize$2$};
\draw (3,1.5) circle(2pt) coordinate(9) node[above right]{\scriptsize$3$};
\draw (2,2.5) circle(2pt) coordinate(10) node[above]{\scriptsize$0$};
\draw[->,dashed, shorten >=2pt,shorten <=2pt] (1) to [out = 165, in = -60] (2) [thick];
\draw[->,shorten >=2pt,shorten <=2pt] (3) to [out = -60, in = 150] (1) [thick];
\qarrow{1}{4}
\draw[->,shorten >=2pt,shorten <=2pt] (5) to [out = -120, in = 30] (1) [thick];
\draw[->,dashed, shorten >=2pt,shorten <=2pt] (1) to [out = 15, in = -120] (6) [thick];
\qarrow{6}{5}
\qarrow{4}{5}
\qarrow{4}{3}
\qarrow{2}{3}
\qdarrow{7}{2}
\qarrow{3}{7}
\qarrow{8}{4}
\qarrow{5}{9}
\qdarrow{9}{6}
\qarrow{9}{8}
\qarrow{7}{8}
\qdarrow{10}{7}
\qarrow{8}{10}
\qdarrow{10}{9}
\draw [->] (5.3,1) to (6.3,1);
\draw (5.8,1) circle(0pt) node[above]{$\qR_{123}$};
\draw (2,-1.2) circle(0pt) node {$B(A_2)$};
\end{scope}
\begin{scope}[>=latex,xshift=215pt]
\coordinate (A) at (3,1.5);
\coordinate (B) at (2,0.5);
\coordinate (C) at (1,1.5);
{\color{red}
\fill (3,1.5) circle(2pt) coordinate(A) node[below right]{$1$};
\fill (2,0.5) circle(2pt) coordinate(B) node[above right]{$2$};
\fill (1,1.5) circle(2pt) coordinate(C) node[below right]{$3$};
\draw [-] (0,0) to [out = 0, in = -135] (B);
\draw [-] (B) -- (A); 
\draw [-] (A) to [out = 45, in = 180] (4,2);
\draw [-] (0,1) to [out = 0, in = -135] (C); 
\draw [-] (C) to [out = 45, in = 135] (A);
\draw [-] (A) to [out = -45, in = 180] (4,1);
\draw [-] (0,2) to [out = 0, in = 135] (C); 
\draw [-] (C) -- (B);
\draw [-] (B) to [out = -45, in = 180] (4,0);
}
%
%
\draw (2,-0.5) circle(2pt) coordinate(1) node[below]{\scriptsize$9$};
\draw (1,0.5) circle(2pt) coordinate(2) node[below left]{\scriptsize$4$};
\draw (3,1.5) circle(2pt) coordinate(3) node[above]{\scriptsize$5$};
\draw (2,1.5) circle(2pt) coordinate(4) node[below left]{\scriptsize$6$};
\draw (1,1.5) circle(2pt) coordinate(5) node[above]{\scriptsize$7$};
\draw (3,0.5) circle(2pt) coordinate(6) node[below right]{\scriptsize$8$};
\draw (0,1.5) circle(2pt) coordinate(7) node[above left]{\scriptsize$1$};
\draw (2,0.5) circle(2pt) coordinate(8);
\draw (1.9,0.2) node {\scriptsize$2$};
\draw (4,1.5) circle(2pt) coordinate(9) node[above right]{\scriptsize$3$};
\draw (2,2.5) circle(2pt) coordinate(10) node[above]{\scriptsize$0$};
\draw[->,dashed, shorten >=2pt,shorten <=2pt] (10) to [out = -165, in = 60] (7) [thick];
\draw[->,shorten >=2pt,shorten <=2pt] (5) to [out = 60, in = -150] (10) [thick];\qarrow{10}{4}
\draw[->,shorten >=2pt,shorten <=2pt] (3) to [out = 120, in = -30] (10) [thick];
\draw[->,dashed, shorten >=2pt,shorten <=2pt] (10) to [out = -15, in = 120] (9) [thick];
\qarrow{7}{5}
\qarrow{4}{5}
\qarrow{4}{3}
\qarrow{9}{3}
\qdarrow{2}{7}
\qarrow{8}{4}
\qarrow{5}{2}
\qarrow{3}{6}
\qdarrow{6}{9}
\qarrow{2}{8}
\qarrow{6}{8}
\qdarrow{1}{2}
\qarrow{8}{1}
\qdarrow{1}{6}
\draw (2,-1.2) circle(0pt) node {$B'(A_2)$};
\end{scope}
\end{tikzpicture}
\end{align}
For both transformations of wiring diagrams and quivers we use the same
notation~$\qR_{123}$,
where the quiver transformation is represented as the mutation sequence  
$\qR_{123} = \sigma_{5,7}\sigma_{2,6}\mu_2 \mu_7 \mu_5 \mu_6$.  

From the wiring diagrams and quivers shown in
Figure~\ref{figTE},
which correspond to the reduced expressions
$123121$ and $321323$ of the longest element in the Weyl group~$W(A_3)$,
we obtain the tetrahedron equation for the quiver transformation~$\qR_{ijk}$:
$$
  \sigma_{7,12}\qR_{456}\qR_{236}\qR_{135}\qR_{124}(B(A_3))
  = \sigma_{7,14}\qR_{124}\qR_{135}\qR_{236}\qR_{456}(B(A_3))
  = B'(A_3),
$$
and for the quantum $Y$-seed,
\begin{align}
\label{TE}
\sigma_{7,12}\qR_{456}\qR_{236}\qR_{135}\qR_{124}(B(A_3),Y)
 = \sigma_{7,14}\qR_{124}\qR_{135}\qR_{236}\qR_{456}(B(A_3),Y).
\end{align}

\begin{figure}[H]
\scalebox{0.9}{
\begin{tikzpicture}
\begin{scope}[>=latex,xshift=0pt]
{\color{red}
\fill (1,0.5) circle(2pt) coordinate(A) node[above right]{$1$};
\fill (2,1.5) circle(2pt) coordinate(B) node[above right]{$2$};
\fill (3,0.5) circle(2pt) coordinate(C) node[above right]{$4$};
\fill (3,2.5) circle(2pt) coordinate(D) node[above right]{$3$};
\fill (4,1.5) circle(2pt) coordinate(E) node[above right]{$5$};
\fill (5,0.5) circle(2pt) coordinate(F) node[above right]{$6$};
\draw [-] (0,3) to [out = 0, in = 135] (D);
\draw [-] (D) -- (F); 
\draw [-] (F) to [out = -45, in = 180] (6,0);
\draw [-] (0,2) to [out = 0, in = 135] (B);
\draw [-] (B) -- (C); 
\draw [-] (C) to [out = -45, in = -135] (F);
\draw [-] (F) to [out = 45, in = 180] (6,1);
\draw [-] (0,1) to [out = 0, in = 135] (A); 
\draw [-] (A) to [out = -45, in = -135] (C);
\draw [-] (C) -- (E);
\draw [-] (E) to [out = 45, in = 180] (6,2);
\draw [-] (0,0) to [out = 0, in = -135] (A); 
\draw [-] (A) -- (D);
\draw [-] (D) to [out = 45, in = 180] (6,3);
}
\draw (3,-0.5) circle(2pt) coordinate(Q17) node[below]{\scriptsize $17$};
\draw (0,0.5) circle(2pt) coordinate(Q10) node[left]{\scriptsize $10$};
\draw (1,0.5) circle(2pt) coordinate(Q11) node[below]{\scriptsize $11$};
\draw (2,0.5) circle(2pt) coordinate(Q12) node[below left]{\scriptsize $12$};
\draw (3,0.5) circle(2pt) coordinate(Q13) node[below left]{\scriptsize $13$};
\draw (4,0.5) circle(2pt) coordinate(Q14) node[below right]{\scriptsize $14$};
\draw (5,0.5) circle(2pt) coordinate(Q15) node[below]{\scriptsize $15$};
\draw (6,0.5) circle(2pt) coordinate(Q16) node[right]{\scriptsize $16$};
\draw (1,1.5) circle(2pt) coordinate(Q5) node[left]{\scriptsize $5$};
\draw (2,1.5) circle(2pt) coordinate(Q6) node[below left]{\scriptsize $6$};
\draw (3,1.5) circle(2pt) coordinate(Q7) node[below left ]{\scriptsize $7$};
\draw (4,1.5) circle(2pt) coordinate(Q8) node[below left]{\scriptsize $8$};
\draw (5,1.5) circle(2pt) coordinate(Q9) node[right]{\scriptsize $9$};
\draw (2,2.5) circle(2pt) coordinate(Q2) node[left]{\scriptsize $2$};
\draw (3,2.5) circle(2pt) coordinate(Q3) node[below left]{\scriptsize $3$};
\draw (4,2.5) circle(2pt) coordinate(Q4) node[right]{\scriptsize $4$};
\draw (3,3.5) circle(2pt) coordinate(Q1) node[above]{\scriptsize $1$};
\qarrow{Q3}{Q1}
\qarrow{Q2}{Q3} \qarrow{Q4}{Q3}
\qarrow{Q3}{Q7} \qarrow{Q6}{Q2} \qarrow{Q8}{Q4}
\qarrow{Q5}{Q6} \qarrow{Q7}{Q6} \qarrow{Q7}{Q8} \qarrow{Q9}{Q8}
\qarrow{Q11}{Q5} \qarrow{Q6}{Q12} \qarrow{Q13}{Q7} \qarrow{Q8}{Q14} \qarrow{Q15}{Q9}
\qarrow{Q10}{Q11} \qarrow{Q12}{Q11} \qarrow{Q12}{Q13} \qarrow{Q14}{Q13} \qarrow{Q14}{Q15} \qarrow{Q16}{Q15}
\qarrow{Q13}{Q17} \qarrow{Q17}{Q14} \qarrow{Q17}{Q12}
\draw [->, shorten >=2pt,shorten <=2pt] (Q15) to [out = -135, in = 15] (Q17) [thick];
\draw [<-, dashed, shorten >=2pt,shorten <=2pt] (Q16) to [out = -135, in = 0] (Q17) [thick];
\draw [->, shorten >=2pt,shorten <=2pt] (Q11) to [out = -45, in = 165] (Q17) [thick];
\draw [<-, dashed, shorten >=2pt,shorten <=2pt] (Q10) to [out = -45, in = 180] (Q17) [thick];
\qdarrow{Q1}{Q2}  \qdarrow{Q2}{Q5} \qdarrow{Q5}{Q10}  
\qdarrow{Q1}{Q4}  \qdarrow{Q4}{Q9} \qdarrow{Q9}{Q16}  
\draw (3,-1.5) node{$B(A_3)$};
\end{scope}
\begin{scope}[>=latex,xshift=230pt]
{\color{red}
\fill (5,2.5) circle(2pt) coordinate(A) node[above]{$1$};
\fill (4,1.5) circle(2pt) coordinate(B) node[above right]{$2$};
\fill (3,2.5) circle(2pt) coordinate(C) node[above right]{$4$};
\fill (3,0.5) circle(2pt) coordinate(D) node[above right]{$3$};
\fill (2,1.5) circle(2pt) coordinate(E) node[above right]{$5$};
\fill (1,2.5) circle(2pt) coordinate(F) node[above]{$6$};
\draw [-] (0,3) to [out = 0, in = 135] (F);
\draw [-] (F) -- (D); 
\draw [-] (D) to [out = -45, in = 180] (6,0);
\draw [-] (0,2) to [out = 0, in = -135] (F) to [out = 45, in = 135] (C) -- (B) to [out = -45, in = 180] (6,1);
\draw [-] (0,1) to [out = 0, in = -135] (E) -- (C) to [out = 45, in = 135] (A) to [out = -45, in = 180] (6,2);
\draw [-] (0,0) to [out = 0, in = -135] (D) -- (A) to [out = 45, in = 180] (6,3);
}
\draw (3,3.5) circle(2pt) coordinate(Q1) node[above]{\scriptsize $1$};
\draw (0,2.5) circle(2pt) coordinate(Q2) node[left]{\scriptsize $2$};
\draw (1,2.5) circle(2pt) coordinate(Q15) node[below left]{\scriptsize $15$};
\draw (2,2.5) circle(2pt) coordinate(Q12) node[below left]{\scriptsize $12$};
\draw (3,2.5) circle(2pt) coordinate(Q13) node[below left]{\scriptsize $13$};
\draw (4,2.5) circle(2pt) coordinate(Q14) node[below left]{\scriptsize $14$};
\draw (5,2.5) circle(2pt) coordinate(Q11) node[below left]{\scriptsize $11$};
\draw (6,2.5) circle(2pt) coordinate(Q4) node[right]{\scriptsize $4$};
\draw (1,1.5) circle(2pt) coordinate(Q5) node[left]{\scriptsize $5$};
\draw (2,1.5) circle(2pt) coordinate(Q8) node[below left]{\scriptsize $8$};
\draw (3,1.5) circle(2pt) coordinate(Q7) node[below left]{\scriptsize $7$};
\draw (4,1.5) circle(2pt) coordinate(Q6) node[below left]{\scriptsize $6$};
\draw (5,1.5) circle(2pt) coordinate(Q9) node[right]{\scriptsize $9$};
\draw (2,0.5) circle(2pt) coordinate(Q10) node[left]{\scriptsize $10$};
\draw (3,0.5) circle(2pt) coordinate(Q3) node[below left]{\scriptsize $3$};
\draw (4,0.5) circle(2pt) coordinate(Q16) node[right]{\scriptsize $16$};
\draw (3,-0.5) circle(2pt) coordinate(Q17) node[below]{\scriptsize $17$};
\qarrow{Q3}{Q17}
\qarrow{Q10}{Q3} \qarrow{Q16}{Q3}
\qarrow{Q3}{Q7} \qarrow{Q6}{Q16} \qarrow{Q8}{Q10}
\qarrow{Q5}{Q8} \qarrow{Q7}{Q8} \qarrow{Q7}{Q6} \qarrow{Q9}{Q6}
\qarrow{Q13}{Q7} \qarrow{Q6}{Q14} \qarrow{Q11}{Q9} \qarrow{Q8}{Q12} \qarrow{Q15}{Q5}
\qarrow{Q4}{Q11} \qarrow{Q12}{Q13} \qarrow{Q14}{Q11} \qarrow{Q14}{Q13} \qarrow{Q12}{Q15} \qarrow{Q2}{Q15}
\qarrow{Q13}{Q1} \qarrow{Q1}{Q14} \qarrow{Q1}{Q12}
\draw [->, shorten >=2pt,shorten <=2pt] (Q15) to [out = 45, in = 195] (Q1) [thick];
\draw [<-, dashed, shorten >=2pt,shorten <=2pt] (Q2) to [out = 45, in = 180] (Q1) [thick];
\draw [->, shorten >=2pt,shorten <=2pt] (Q11) to [out = 135, in = -15] (Q1) [thick];
\draw [<-, dashed, shorten >=2pt,shorten <=2pt] (Q4) to [out = 135, in = 0] (Q1) [thick];
\qdarrow{Q10}{Q17} \qdarrow{Q2}{Q5} \qdarrow{Q5}{Q10}  
\qdarrow{Q16}{Q17}  \qdarrow{Q4}{Q9} \qdarrow{Q9}{Q16}  
\draw (3,-1.5) node{$B'(A_3)$};
\end{scope}
\end{tikzpicture}
}
\caption{The tetrahedron relation for SB.}
\label{figTE}
\end{figure}
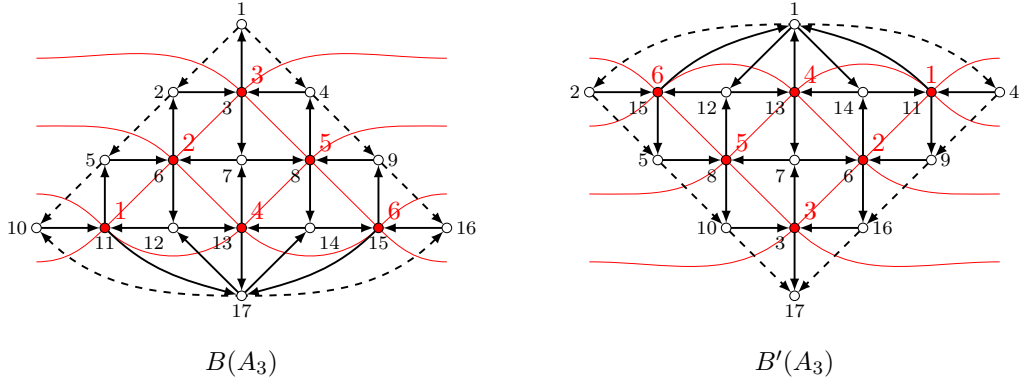

Corresponding to the mutation sequence 
$\qR_{123} = \sigma_{5,7}\sigma_{2,6}\mu_2 \mu_7 \mu_5 \mu_6$,
we define a sequence of quantum $Y$-seeds as
\begin{align}\label{R123-qY}
\begin{split}
(B(A_2),Y) &= (B^{(1)},Y^{(1)}) 
 \stackrel{\mu_6}{\longrightarrow} (B^{(2)},Y^{(2)}) 
 \stackrel{\mu_5}{\longrightarrow} (B^{(3)},Y^{(3)}) 
 \stackrel{\mu_7}{\longrightarrow} (B^{(4)},Y^{(4)}) 
\\
&\qquad\stackrel{\mu_2}{\longrightarrow} (B^{(5)},Y^{(5)}) 
 \stackrel{\sigma_{5,7}\sigma_{2,6}}{\longrightarrow} 
 (B^{(6)},Y^{(6)}) = (B'(A_2),Y').
\end{split}
\end{align}
This mutation sequence induces an isomorphism of skewfields
$\widehat{R}_{123} : \mathcal{Y}(B'(A_2)) \to \mathcal{Y}(B(A_2))$.
For a sign sequence $\ve = (\ve_1,\ve_2,\ve_3,\ve_4) \in \{-1,1\}^4$,
$\widehat{R}_{123}$ can be expressed as
\begin{equation}\label{Rhat}
\begin{split}
\widehat{R}_{123} 
&= 
\mathrm{Ad}\!\left(\Psi_q((Y^{(1)}_6)^{\varepsilon_1})^{\varepsilon_1}\right)\tau_{6,\varepsilon_1}
\mathrm{Ad}\!\left(\Psi_q((Y^{(2)}_5)^{\varepsilon_2})^{\varepsilon_2}\right)\tau_{5,\varepsilon_2}
\\
&\quad\cdot
\mathrm{Ad}\!\left(\Psi_q((Y^{(3)}_7)^{\varepsilon_3})^{\varepsilon_3}\right)\tau_{7,\varepsilon_3}
\mathrm{Ad}\!\left(\Psi_q((Y^{(4)}_2)^{\varepsilon_4})^{\varepsilon_4}\right)\tau_{2,\varepsilon_4}
\sigma_{5,7}\sigma_{2,6}.
\end{split}
\end{equation}
We define the monomial part of $\widehat{R}_{123}$ as
\begin{align}\label{Rmono}
\tau_{123|\ve} =
\tau_{6,\varepsilon_1}\,
\tau_{5,\varepsilon_2}\,
\tau_{7,\varepsilon_3}\,
\tau_{2,\varepsilon_4}\,
\sigma_{5,7}\sigma_{2,6}
:\;
\mathcal{Y}(B'(A_2)) \longrightarrow \mathcal{Y}(B(A_2)).
\end{align}

Both $\widehat{R}_{123}$ and $\tau_{123|\ve}$ can be extended to
$\widehat{R}_{ijk}$ and $\tau_{ijk|\ve}$ corresponding to the
transformations~$\qR_{ijk}$ appearing in the tetrahedron
equation~\eqref{TE}.  
In~\cite{IKSTY}, we studied a {\it homogeneous} tetrahedron
equation, in which all the operators~$\widehat{R}_{ijk}$ share a common
sign sequence~$\ve$, and the following
conditions are satisfied:
\begin{itemize}\label{ve-con}
\item[(i)] 
The tetrahedron equation for the operators~$\widehat{R}_{ijk}$ can be
{\it decomposed} into two parts: a tetrahedron equation for their
monomial parts~$\tau_{ijk|\ve}$ with the same sign sequence~$\ve$,
\begin{align*}
\tau_{124|\ve}\tau_{135|\ve}\tau_{236|\ve}\tau_{456|\ve}\sigma_{7,12}
=
\tau_{456|\ve}\tau_{236|\ve}\tau_{135|\ve}\tau_{124|\ve}\sigma_{7,14},
\end{align*}
and a corresponding dilogarithm identity.
\item[(ii)] 
Through a ring homomorphism from quantum $Y$ vartiables to $q$-Weyl algebras, 
each monomial part~$\tau_{ijk|\ve}$ of~$\widehat{R}_{ijk}$ associated
with the sign sequence~$\ve$ is realized as the adjoint action of a
suitable operator.
\end{itemize}

It turned out that there are two sign sequences which satisfy the condition (i):
$\ve = (-,-,+,+)$ and $(-,+,-,+)$ \cite[Proposition 3.8 and 3.9]{IKSTY}.
For these sign sequences, \eqref{Rhat} and \eqref{Rmono} take the forms 
\begin{align}
\label{hatR--++}
&\widehat{R}_{123} = 
\mathrm{Ad}\bigl(\Psi_q(Y_6^{-1})^{-1} \Psi_q(q Y_5^{-1} Y_6^{-1})^{-1} 
\Psi_q(q^{-1} Y_6 Y_7) \Psi_q(q^{-2}Y_2 Y_6 Y_7) \bigr) \tau_{123 | --++},
\\ \displaybreak[0]
\label{tau--++}
&\tau_{123 | --++}: 
\begin{cases}
Y_0' \mapsto Y_0, &Y_5' \mapsto Y_2,   
\\
Y_1' \mapsto q^2 Y_1 Y_5 Y_6, &Y_6' \mapsto q^{-2} Y_2^{-1} Y_6^{-1} Y_7^{-1}, 
\\
Y_2' \mapsto Y_5, &Y_7' \mapsto Y_2 Y_5^{-1} Y_7,  
\\
Y_3' \mapsto Y_3, &Y_8' \mapsto Y_6 Y_7 Y_8,
\\
Y_4' \mapsto Y_4, &Y_9' \mapsto Y_9,
\end{cases} 
\end{align}
for $\ve =(-,-,+,+)$, and 
\begin{align}
\label{hatR-+-+}
&\widehat{R}_{123} = 
\mathrm{Ad}\bigl(\Psi_q(Y_6^{-1})^{-1} \Psi_q(q Y_5 Y_6) 
\Psi_q(q^{-1} Y_6^{-1} Y_7^{-1})^{-1} \Psi_q(Y_2 Y_5 Y_6) \bigr) \tau_{123 | -+-+},
\\ \displaybreak[0]
\label{tau-+-+}
&\tau_{123 | -+-+}:
\begin{cases}
Y_0' \mapsto Y_0, &Y_5' \mapsto Y_2 Y_5 Y_7^{-1},   
\\
Y_1' \mapsto Y_1, &Y_6' \mapsto Y_2^{-1} Y_5^{-1} Y_6^{-1}, 
\\
Y_2' \mapsto Y_7, &Y_7' \mapsto Y_2,  
\\
Y_3' \mapsto Y_3 Y_6 Y_7, &Y_8' \mapsto Y_8,
\\
Y_4' \mapsto Y_4 Y_5 Y_6, &Y_9' \mapsto Y_9,
\end{cases} 
\end{align}
for $\ve =(-,+,-,+)$.
We will explain in \S \ref{sec:Rop} that these sign sequences satisfy the condition (ii).

For later use, we also introduce a transformation $\overline{\qR}_{123}$:
\begin{align}\label{SB-Rb}
\begin{tikzpicture}
\begin{scope}[>=latex,xshift=0pt]
{\color{red}
\fill (1,0.5) circle(2pt) coordinate(A) node[above right]{\scriptsize$3$};
\fill (2,1.5) circle(2pt) coordinate(B) node[below right]{\scriptsize$2$};
\fill (3,0.5) circle(2pt) coordinate(C) node[above left]{\scriptsize$1$};
\draw [-] (0,2) to [out = 0, in = 135] (B);
\draw [-] (B) -- (C); 
\draw [-] (C) to [out = -45, in = 180] (4,0);
\draw [-] (0,1) to [out = 0, in = 135] (A); 
\draw [-] (A) to [out = -45, in = -135] (C);
\draw [-] (C) to [out = 45, in = 180] (4,1);
\draw [-] (0,0) to [out = 0, in = -135] (A); 
\draw [-] (A) -- (B);
\draw [-] (B) to [out = 45, in = 180] (4,2);
}
%
\draw (2,-0.5) circle(2pt) coordinate(1) node[below]{\scriptsize$9$};
\draw (0,0.5) circle(2pt) coordinate(2) node[below left]{\scriptsize$4$};
\draw (1,0.5) circle(2pt) coordinate(3) node[below=1pt]{\scriptsize$5$};
\draw (2,0.5) circle(2pt) coordinate(4) node[above left]{\scriptsize$6$};
\draw (3,0.5) circle(2pt) coordinate(5) node[below=1pt]{\scriptsize$7$};
\draw (4,0.5) circle(2pt) coordinate(6) node[below right]{\scriptsize$8$};
\draw (1,1.5) circle(2pt) coordinate(7) node[above left]{\scriptsize$1$};
\draw (2,1.5) circle(2pt) coordinate(8); 
\draw (1.9,1.75) node {\scriptsize$2$};
\draw (3,1.5) circle(2pt) coordinate(9) node[above right]{\scriptsize$3$};
\draw (2,2.5) circle(2pt) coordinate(10) node[above]{\scriptsize$0$};
\draw[->,dashed, shorten >=2pt,shorten <=2pt] (1) to [out = 165, in = -60] (2) [thick];
\draw[->,shorten >=2pt,shorten <=2pt] (3) to [out = -60, in = 150] (1) [thick];
\qarrow{1}{4}
\draw[->,shorten >=2pt,shorten <=2pt] (5) to [out = -120, in = 30] (1) [thick];
\draw[->,dashed, shorten >=2pt,shorten <=2pt] (1) to [out = 15, in = -120] (6) [thick];
\qarrow{6}{5}
\qarrow{4}{5}
\qarrow{4}{3}
\qarrow{2}{3}
\qdarrow{7}{2}
\qarrow{3}{7}
\qarrow{8}{4}
\qarrow{5}{9}
\qdarrow{9}{6}
\qarrow{9}{8}
\qarrow{7}{8}
\qdarrow{10}{7}
\qarrow{8}{10}
\qdarrow{10}{9}
\draw [<-] (5.3,1) to (6.3,1);
\draw (5.8,1) node[above]{$\overline{\qR}_{123}$};
\draw (2,-1.2) circle(0pt) node {$B(A_2)$};
\end{scope}
\begin{scope}[>=latex,xshift=215pt]
\coordinate (A) at (3,1.5);
\coordinate (B) at (2,0.5);
\coordinate (C) at (1,1.5);
{\color{red}
\fill (3,1.5) circle(2pt) coordinate(A) node[below left]{\scriptsize$3$};
\fill (2,0.5) circle(2pt) coordinate(B) node[above right]{\scriptsize$2$};
\fill (1,1.5) circle(2pt) coordinate(C) node[below right]{\scriptsize$1$};
\draw [-] (0,0) to [out = 0, in = -135] (B);
\draw [-] (B) -- (A); 
\draw [-] (A) to [out = 45, in = 180] (4,2);
\draw [-] (0,1) to [out = 0, in = -135] (C); 
\draw [-] (C) to [out = 45, in = 135] (A);
\draw [-] (A) to [out = -45, in = 180] (4,1);
\draw [-] (0,2) to [out = 0, in = 135] (C); 
\draw [-] (C) -- (B);
\draw [-] (B) to [out = -45, in = 180] (4,0);
}
%
%
\draw (2,-0.5) circle(2pt) coordinate(1) node[below]{\scriptsize$9$};
\draw (1,0.5) circle(2pt) coordinate(2) node[below left]{\scriptsize$4$};
\draw (3,1.5) circle(2pt) coordinate(3) node[above]{\scriptsize$5$};
\draw (2,1.5) circle(2pt) coordinate(4) node[below left]{\scriptsize$6$};
\draw (1,1.5) circle(2pt) coordinate(5) node[above]{\scriptsize$7$};
\draw (3,0.5) circle(2pt) coordinate(6) node[below right]{\scriptsize$8$};
\draw (0,1.5) circle(2pt) coordinate(7) node[above left]{\scriptsize$1$};
\draw (2,0.5) circle(2pt) coordinate(8);
\draw (1.9,0.2) node {\scriptsize$2$};
\draw (4,1.5) circle(2pt) coordinate(9) node[above right]{\scriptsize$3$};
\draw (2,2.5) circle(2pt) coordinate(10) node[above]{\scriptsize$0$};
\draw[->,dashed, shorten >=2pt,shorten <=2pt] (10) to [out = -165, in = 60] (7) [thick];
\draw[->,shorten >=2pt,shorten <=2pt] (5) to [out = 60, in = -150] (10) [thick];\qarrow{10}{4}
\draw[->,shorten >=2pt,shorten <=2pt] (3) to [out = 120, in = -30] (10) [thick];
\draw[->,dashed, shorten >=2pt,shorten <=2pt] (10) to [out = -15, in = 120] (9) [thick];
\qarrow{7}{5}
\qarrow{4}{5}
\qarrow{4}{3}
\qarrow{9}{3}
\qdarrow{2}{7}
\qarrow{8}{4}
\qarrow{5}{2}
\qarrow{3}{6}
\qdarrow{6}{9}
\qarrow{2}{8}
\qarrow{6}{8}
\qdarrow{1}{2}
\qarrow{8}{1}
\qdarrow{1}{6}
\draw (2,-1.2) circle(0pt) node {$B'(A_2)$};
\end{scope}
\end{tikzpicture}
\end{align}
Note that the numberings of the crossings of the wiring diagrams
\eqref{SB-R} and~\eqref{SB-Rb} are different, whereas the numberings of the quiver vertices 
in \eqref{SB-R} and~\eqref{SB-Rb} are the same.
Hence as a transformation of wiring diagram, $\overline{\qR}_{123}$ is not the inverse of $\qR_{123}$.
On the other hand, as a mutation sequence, $\overline{\qR}_{123}$ coincides with~${\qR}_{123}$, 
and induces an isomorphism of skewfields $\mathcal{Y}(B(A_2)) \to \mathcal{Y}(B'(A_2))$, which is the inverse of $\qR_{123}$.

In analogy with~$\qR_{123}$, we define a sequence of quantum
$Y$-seeds by
\begin{align}
\begin{split}
({B}'(A_2),\overline{Y}) 
 &= (\overline{B}^{(1)},\overline{Y}^{(1)}) 
 \stackrel{\mu_6}{\longrightarrow} 
 (\overline{B}^{(2)},\overline{Y}^{(2)}) 
 \stackrel{\mu_5}{\longrightarrow}
 (\overline{B}^{(3)},\overline{Y}^{(3)}) 
 \stackrel{\mu_7}{\longrightarrow}
 (\overline{B}^{(4)},\overline{Y}^{(4)}) 
\\
&\qquad
\stackrel{\mu_2}{\longrightarrow}
 (\overline{B}^{(5)},\overline{Y}^{(5)}) 
 \stackrel{\sigma_{5,7}\sigma_{2,6}}{\longrightarrow}
 (\overline{B}^{(6)},\overline{Y}^{(6)}) 
 = ({B}(A_2),\overline{Y}'),
\end{split}
\end{align}
where we write $\overline{Y}=(\overline{Y}_i)$ for the quantum $Y$-variables to emphasize the difference from \eqref{R123-qY}.

We also define the monomial part~$\overline{\tau}_{123|\ve}$ of the induced
isomorphism of skewfields
$\widehat{\overline{R}}_{123} : \mathcal{Y}(B(A_2)) \to
\mathcal{Y}(B'(A_2))$
for a sign sequence $\ve = (\ve_1,\ve_2,\ve_3,\ve_4) \in \{-1,1\}^4$ by
\begin{align}\label{Rbmono}
\overline{\tau}_{123|\ve} =
\tau_{6,\varepsilon_1}\,
\tau_{5,\varepsilon_2}\,
\tau_{7,\varepsilon_3}\,
\tau_{2,\varepsilon_4}\,
\sigma_{5,7}\sigma_{2,6}
:\;
\mathcal{Y}(B(A_2)) \longrightarrow \mathcal{Y}(B'(A_2)).
\end{align}
For the sign sequences $\ve = (-,-,+,+)$ and $\ve = (-,+,-,+)$ introduced above,  
the isomorphism~$\widehat{\overline{R}}_{123}$ and its monomial
part~\eqref{Rbmono} are given by
\begin{align}
&\widehat{\overline{R}}_{123} = 
\mathrm{Ad}\bigl(\Psi_q(\overline{Y}_6^{-1})^{-1} \Psi_q(q \overline{Y}_5^{-1} 
\overline{Y}_6^{-1})^{-1} \Psi_q(q^{-1} \overline{Y}_6 \overline{Y}_7) 
\Psi_q(q^{-2}\overline{Y}_2 \overline{Y}_6 \overline{Y}_7) \bigr) \overline{\tau}_{123 | --++},
\label{hbr}\\ \displaybreak[0]
\label{taub--++}
&\overline{\tau}_{123 | --++}: 
\begin{cases}
\overline{Y}_0' \mapsto \overline{Y}_0, &\overline{Y}_5' \mapsto \overline{Y}_2,\\
\overline{Y}_1' \mapsto q^{-2} \overline{Y}_1 \overline{Y}_6 \overline{Y}_7,
 &\overline{Y}_6' \mapsto q^{-2} \overline{Y}_2^{-1} \overline{Y}_6^{-1} \overline{Y}_7^{-1}, 
\\
\overline{Y}_2' \mapsto \overline{Y}_5, &\overline{Y}_7' \mapsto \overline{Y}_2 \overline{Y}_5^{-1} \overline{Y}_7,  
\\
\overline{Y}_3' \mapsto \overline{Y}_3, &\overline{Y}_8' \mapsto  \overline{Y}_5 \overline{Y}_6 \overline{Y}_8,
\\
\overline{Y}_4' \mapsto \overline{Y}_4, &\overline{Y}_9' \mapsto \overline{Y}_9,
\end{cases} 
\end{align}
for $\ve =(-,-,+,+)$, and 
\begin{align}
&\widehat{\overline{R}}_{123} = 
\mathrm{Ad}\bigl(\Psi_q(\overline{Y}_6^{-1})^{-1} \Psi_q(q \overline{Y}_5 \overline{Y}_6) \Psi_q(q^{-1} \overline{Y}_6^{-1} \overline{Y}_7^{-1})^{-1} \Psi_q(\overline{Y}_2 \overline{Y}_5 \overline{Y}_6) \bigr) \overline{\tau}_{123 | -+-+},
\\ \displaybreak[0]
\label{taub-+-+}
&\overline{\tau}_{123 | -+-+}:
\begin{cases}
\overline{Y}_0' \mapsto \overline{Y}_0, &\overline{Y}_5' \mapsto \overline{Y}_2 \overline{Y}_5 \overline{Y}_7^{-1},   
\\
\overline{Y}_1' \mapsto \overline{Y}_1, &\overline{Y}_6' \mapsto \overline{Y}_2^{-1} \overline{Y}_5^{-1} \overline{Y}_6^{-1}, 
\\
\overline{Y}_2' \mapsto \overline{Y}_7, &\overline{Y}_7' \mapsto \overline{Y}_2,\\
\overline{Y}_3' \mapsto \overline{Y}_3 \overline{Y}_5 \overline{Y}_6,   
 &\overline{Y}_8' \mapsto \overline{Y}_8,
\\
\overline{Y}_4' \mapsto \overline{Y}_4 \overline{Y}_6 \overline{Y}_7,
&\overline{Y}_9' \mapsto \overline{Y}_9,
\end{cases} 
\end{align}
for $\ve =(-,+,-,+)$.
Note that for either sign sequences $\ve$, the $\overline{\tau}_{123|\ve}$ is the inverse map of $\tau_{123|\ve}$. 

\subsection{$K$-operator}

We now introduce $B(C_2)$ and $B'(C_2)$ as the SB
quivers corresponding to the wiring diagrams associated with the two
reduced expressions $1212$ and $2121$ of the longest element in the
Weyl group~$W(C_2)$.  
Let~$\qK_{1234}$ denote the transformation of the wiring diagrams (shown
in red) and the quivers (shown in black) defined as follows. 
\begin{align}
\label{SB-KC}
\begin{split}
\begin{tikzpicture}
\begin{scope}[>=latex,xshift=0pt]
{\color{red}
\fill (2,0) circle(2pt) coordinate(A) node[above right]{\scriptsize $1$};
\fill (4,0) circle(2pt) coordinate(C) node[above right]{\scriptsize $3$};
\fill (3,1) circle(2pt) coordinate(B) node[below right]{\scriptsize $2$};
\fill (5,1) circle(2pt) coordinate(D) node[below right]{\scriptsize $4$};
\draw [->] (1,0.5)  to [out = 0, in = 135] (A) to [out = -45, in = -135] (C) -- (D) to [out = -45, in = 180] (6,0.5);
\draw [->] (1,-0.5) to [out = 0, in = -135] (A) -- (B) -- (C) to [out = -45, in = 180] (6,-0.5);
}
%
\path (2,1) coordinate(A1) node[above=0.2em]{\scriptsize $1$};
\maruni{A1} 
\path (3,1) coordinate(A2) node[above=0.2em]{\scriptsize $2$};
\maruni{A2}
\path (4,1) coordinate(A3) node[above=0.2em]{\scriptsize $3$};
\maruni{A3}
\path (5,1) coordinate(A4) node[above=0.2em]{\scriptsize $4$};
\maruni{A4}
\path (6,1) coordinate(A5) node[above=0.2em]{\scriptsize $5$};
\maruni{A5}
\draw (1,0) circle(2pt) coordinate(A6) node[below left]{\scriptsize $6$};
\draw (2,0) circle(2pt) coordinate(A7) node[below]{\scriptsize $7$};
\draw (3,0) circle(2pt) coordinate(A8) node[below left]{\scriptsize $8$};
\draw (4,0) circle(2pt) coordinate(A9) node[below]{\scriptsize $9$};
\draw (5,0) circle(2pt) coordinate(A10) node[below right]{\scriptsize $10$};
\draw (3,-1) circle(2pt) coordinate(A11) node[below]{\scriptsize $11$};
\qsarrow{A1}{A2}
\qsarrow{A3}{A2}
\qsarrow{A3}{A4}
\qsarrow{A5}{A4}
\qarrow{A6}{A7}
\qarrow{A8}{A7}
\qarrow{A8}{A9}
\qarrow{A10}{A9}
\draw[->,shorten >=2pt,shorten <=4pt] (A4) -- (A10) [thick];
\draw[->,shorten >=4pt,shorten <=2pt] (A9) -- (A3) [thick];
\draw[->,shorten >=2pt,shorten <=4pt] (A2) -- (A8) [thick];
\draw[->,shorten >=4pt,shorten <=2pt] (A7) -- (A1) [thick];
\draw[->,dashed,shorten >=4pt,shorten <=2pt] (A10) -- (A5) [thick];
\draw[->,dashed,shorten >=2pt,shorten <=4pt] (A1) -- (A6) [thick];
\draw[->,dashed, shorten >=2pt,shorten <=2pt] (A11) to [out = 165, in = -60] (A6) [thick];
\draw[->,shorten >=2pt,shorten <=2pt] (A7) to [out = -60, in = 150] (A11) [thick];
\qarrow{A11}{A8}
\draw[->,shorten >=2pt,shorten <=2pt] (A9) to [out = -120, in = 30] (A11) [thick];
\draw[->,dashed, shorten >=2pt,shorten <=2pt] (A11) to [out = 15, in = -120] (A10) [thick];
\draw[->] (6.8,0)--(7.8,0); \draw (7.3,0) node[above] {$\qK_{1234}$};
\path (3,-1.9) node {$B(C_2)$};
\end{scope}
\begin{scope}[>=latex,xshift=225pt]
{\color{red}
\fill (5,0) circle(2pt) coordinate(A) node[above right]{\scriptsize $1$};
\fill (3,0) circle(2pt) coordinate(C) node[above right]{\scriptsize $3$};
\fill (4,1) circle(2pt) coordinate(B) node[below right]{\scriptsize $2$};
\fill (2,1) circle(2pt) coordinate(D) node[below right]{\scriptsize $4$};
\draw [->] (1,0.5) to [out = 0, in = -135] (D) -- (C) to [out = -45, in = -135] (A) to [out = 45, in = 180] (6,0.5);
\draw [->] (1,-0.5) to [out = 0, in = -135] (C) -- (B) -- (A) to [out = -45, in = 180] (6,-0.5);
}
%
%
\path (1,1) coordinate(A1) node[above=0.2em]{\scriptsize $1$};
\maruni{A1} 
\path (2,1) coordinate(A2) node[above=0.2em]{\scriptsize $2$};
\maruni{A2}
\path (3,1) coordinate(A3) node[above=0.2em]{\scriptsize $3$};
\maruni{A3}
\path (4,1) coordinate(A4) node[above=0.2em]{\scriptsize $4$};
\maruni{A4}
\path (5,1) coordinate(A5) node[above=0.2em]{\scriptsize $5$};
\maruni{A5}
\draw (2,0) circle(2pt) coordinate(A6) node[below left]{\scriptsize $6$};
\draw (3,0) circle(2pt) coordinate(A7) node[below]{\scriptsize $7$};
\draw (4,0) circle(2pt) coordinate(A8) node[below left]{\scriptsize $8$};
\draw (5,0) circle(2pt) coordinate(A9) node[below]{\scriptsize $9$};
\draw (6,0) circle(2pt) coordinate(A10) node[below right]{\scriptsize $10$};
\draw (4,-1) circle(2pt) coordinate(A11) node[below]{\scriptsize $11$};
\qsarrow{A1}{A2}
\qsarrow{A3}{A2}
\qsarrow{A3}{A4}
\qsarrow{A5}{A4}
\qarrow{A6}{A7}
\qarrow{A8}{A7}
\qarrow{A8}{A9}
\qarrow{A10}{A9}
\draw[->,shorten >=4pt,shorten <=2pt] (A9) -- (A5) [thick];
\draw[->,shorten >=2pt,shorten <=4pt] (A4) -- (A8) [thick];
\draw[->,shorten >=4pt,shorten <=2pt] (A7) -- (A3) [thick];
\draw[->,shorten >=2pt,shorten <=4pt] (A2) -- (A6) [thick];
\draw[->,dashed,shorten >=2pt,shorten <=4pt] (A5) -- (A10) [thick];
\draw[->,dashed,shorten >=4pt,shorten <=2pt] (A6) -- (A1) [thick];
\draw[->,dashed, shorten >=2pt,shorten <=2pt] (A11) to [out = 165, in = -60] (A6) [thick];
\draw[->,shorten >=2pt,shorten <=2pt] (A7) to [out = -60, in = 150] (A11) [thick];
\qarrow{A11}{A8}
\draw[->,shorten >=2pt,shorten <=2pt] (A9) to [out = -120, in = 30] (A11) [thick];
\draw[->,dashed, shorten >=2pt,shorten <=2pt] (A11) to [out = 15, in = -120] (A10) [thick];
\path (4,-1.9) node {$B'(C_2)$};
\end{scope}
\end{tikzpicture}
\end{split}
\end{align}
The exchange matrices for these quivers are not skew symmetric. 
Using the diagonal matrix
$d = (2,2,2,2,2,1,1,1,1,1,1)$,
we symmetrize the exchange matrix
$B = ({b}_{ij})$ for $B(C_2)$ to obtain
\begin{align}\label{Bhat-K}
\widehat{B}:= B\, d = 
\begin{bmatrix}
0 & 2 & 0 & 0 & 0 & 1 & -2 & 0 & 0 & 0 & 0\\
-2 & 0 & -2 & 0 & 0 & 0 & 0 & 2 & 0 & 0 & 0\\
0 & 2 & 0 & 2 & 0 & 0 & 0 & 0 & -2 & 0 & 0\\
0 & 0 & -2 & 0 & -2 & 0 & 0 & 0 & 0 & 2 & 0\\
0 & 0 & 0 & 2 & 0 & 0 & 0 & 0 & 0 & -1 & 0\\
-1 & 0 & 0 & 0 & 0 & 0 & 1 & 0 & 0 & 0 & -\frac{1}{2}\\
2 & 0 & 0 & 0 & 0 & -1 & 0 & -1 & 0 & 0 & 1\\
0 & -2 & 0 & 2 & 0 & 0 & 1 & 0 & 1 & 0 & -1\\
0 & 0 & 2 & 0 & 0 & 0 & 0 & -1 & 0 & -1 & 1\\
0 & 0 & 0 & -2 & 1 & 0 & 0 & 0 & 1 & 0 & -\frac{1}{2}\\
0 & 0 & 0 & 0 & 0 & \frac{1}{2} & -1 & 1 & -1 & \frac{1}{2} & 0\\
\end{bmatrix}.
\end{align} 
This determines the $q$-commutation relations among the quantum
$Y$-variables in the $Y$-seed $(B(C_2), Y)$ as
$Y_i Y_j = q^{2 \widehat{b}_{ij}} Y_j Y_i$.
See \eqref{q-Y}.
If there are no arrows between vertices~$i$ and~$j$ of a quiver
(i.e.,~$b_{ij}=0$), we write $\mu_{i,j}$ for the corresponding pair of
commuting mutations $\mu_i \mu_j \,(= \mu_j \mu_i)$.

\begin{remark}
The rank of the matrix $\widehat{B}$ \eqref{Bhat-K} is $8$.
As a result, the skewfield $\mathcal{Y}(B(C_2))$ generated by
$Y_1,\ldots,Y_{11}$ has a center generated by the following three elements:
\begin{align}
&Y_1 Y_3^{-1} Y_5 Y_6^4 Y_7^2 Y_{11}^2,
\quad
Y_1^2 Y_2 Y_4^{-1} Y_5^{-2} Y_6^2 Y_7^2 Y_8^2 Y_{10}^{-2},
\quad
Y_1 Y_2 Y_3 Y_5^{-1} Y_7 Y_8^2 Y_9.
\end{align}
We note that this rank coincides with the number of
canonical variables $u_i, w_i$ $(i=1,\ldots,4)$, which will be introduced in
\S\ref{ss:Kop}.
A similar feature also holds for the quiver $B(A_2)$ in \eqref{SB-Rb},
as well as for the square quiver studied in \cite{IKT2}.
\end{remark}

\begin{lemma}
The relation 
$\mu_{3,8}\mu_{2,9}\mu_{4,7}\mu_{2,9}\mu_{3,8}(B(C_2)) = B'(C_2)$ 
is satisfied.
Equivalently, as a quiver transformation,
$\qK_{1234} = \mu_{3,8}\mu_{2,9}\mu_{4,7}\mu_{2,9}\mu_{3,8}$.
\end{lemma}

\begin{figure}[H]
\begin{tikzpicture}
\begin{scope}[>=latex,xshift=0pt]
\draw[->] (7,0.3) -- (8,0.3);
\draw (7.5,0.3) circle(0pt) node[below] {$\mu_{3,8} \mu_{2,9} \mu_{4,7} \mu_{2,9} \mu_{3,8}$};
\path (2,1) coordinate(A1) node[above=0.2em]{\scriptsize $1$};
\maruni{A1} 
\path (3,1) coordinate(A2) node[above=0.2em]{\scriptsize $2$};
\maruni{A2}
\path (4,1) coordinate(A3) node[above=0.2em]{\scriptsize $3$};
\maruni{A3}
\path (5,1) coordinate(A4) node[above=0.2em]{\scriptsize $4$};
\maruni{A4}
\path (6,1) coordinate(A5) node[above=0.2em]{\scriptsize $5$};
\maruni{A5}
\draw (1,0) circle(2pt) coordinate(A6) node[below left]{\scriptsize $6$};
\draw (2,0) circle(2pt) coordinate(A7) node[below]{\scriptsize $7$};
\draw (3,0) circle(2pt) coordinate(A8) node[below left]{\scriptsize $8$};
\draw (4,0) circle(2pt) coordinate(A9) node[below]{\scriptsize $9$};
\draw (5,0) circle(2pt) coordinate(A10) node[below right]{\scriptsize $10$};
\draw (3,-1) circle(2pt) coordinate(A11) node[below]{\scriptsize $11$};
\qsarrow{A1}{A2}
\qsarrow{A3}{A2}
\qsarrow{A3}{A4}
\qsarrow{A5}{A4}
\qarrow{A6}{A7}
\qarrow{A8}{A7}
\qarrow{A8}{A9}
\qarrow{A10}{A9}
\draw[->,shorten >=2pt,shorten <=4pt] (A4) -- (A10) [thick];
\draw[->,shorten >=4pt,shorten <=2pt] (A9) -- (A3) [thick];
\draw[->,shorten >=2pt,shorten <=4pt] (A2) -- (A8) [thick];
\draw[->,shorten >=4pt,shorten <=2pt] (A7) -- (A1) [thick];
\draw[->,dashed,shorten >=4pt,shorten <=2pt] (A10) -- (A5) [thick];
\draw[->,dashed,shorten >=2pt,shorten <=4pt] (A1) -- (A6) [thick];
\draw[->,dashed, shorten >=2pt,shorten <=2pt] (A11) to [out = 165, in = -60] (A6) [thick];
\draw[->,shorten >=2pt,shorten <=2pt] (A7) to [out = -60, in = 150] (A11) [thick];
\qarrow{A11}{A8}
\draw[->,shorten >=2pt,shorten <=2pt] (A9) to [out = -120, in = 30] (A11) [thick];
\draw[->,dashed, shorten >=2pt,shorten <=2pt] (A11) to [out = 15, in = -120] (A10) [thick];
\draw[->] (3.5,-1.5) -- (3.5,-2.2);
\path (3.5,-1.9) node[right]{$\mu_{3,8}$};
\end{scope}
\begin{scope}[>=latex,xshift=0pt,yshift=-110pt]
\path (2,1) coordinate(A1) node[above=0.2em]{\scriptsize $1$};
\maruni{A1} 
\path (3,1) coordinate(A2) node[above=0.2em]{\scriptsize $2$};
\maruni{A2}
\path (4,1) coordinate(A3) node[above=0.2em]{\scriptsize $3$};
\maruni{A3}
\path (5,1) coordinate(A4) node[above=0.2em]{\scriptsize $4$};
\maruni{A4}
\path (6,1) coordinate(A5) node[above=0.2em]{\scriptsize $5$};
\maruni{A5}
\draw (1,0) circle(2pt) coordinate(A6) node[below left]{\scriptsize $6$};
\draw (2,0) circle(2pt) coordinate(A7) node[below]{\scriptsize $7$};
\draw (3,0) circle(2pt) coordinate(A8) node[below left]{\scriptsize $8$};
\draw (4,0) circle(2pt) coordinate(A9) node[below]{\scriptsize $9$};
\draw (5,0) circle(2pt) coordinate(A10) node[below right]{\scriptsize $10$};
\draw (3,-1) circle(2pt) coordinate(A11) node[below]{\scriptsize $11$};
\qsarrow{A1}{A2}
\qsarrow{A2}{A3}
\qsarrow{A4}{A3}
\qsarrow{A5}{A4}
\qarrow{A6}{A7}
\qarrow{A7}{A8}
\qarrow{A9}{A8}
\qarrow{A10}{A9}
\draw[->,shorten >=2pt,shorten <=4pt] (A4) -- (A10) [thick];
\draw[->,shorten >=2pt,shorten <=4pt] (A3) -- (A9) [thick];
\draw[->,shorten >=4pt,shorten <=2pt] (A8) -- (A2) [thick];
\draw[->,shorten >=4pt,shorten <=2pt] (A7) -- (A1) [thick];
\draw[->,dashed,shorten >=4pt,shorten <=2pt] (A10) -- (A5) [thick];
\draw[->,dashed,shorten >=2pt,shorten <=4pt] (A1) -- (A6) [thick];
\draw[->,dashed, shorten >=2pt,shorten <=2pt] (A11) to [out = 165, in = -60] (A6) [thick];
\qarrow{A8}{A11}
\draw[->,dashed, shorten >=2pt,shorten <=2pt] (A11) to [out = 15, in = -120] (A10) [thick];
\draw[->] (3.5,-1.5) -- (3.5,-2.2);
\path (3.5,-1.9) node[right]{$\mu_{2,9}$};
\end{scope}
\begin{scope}[>=latex,xshift=0pt,yshift=-220pt]
\path (2,1) coordinate(A1) node[above=0.2em]{\scriptsize $1$};
\maruni{A1} 
\path (3,1) coordinate(A2) node[above=0.2em]{\scriptsize $2$};
\maruni{A2}
\path (4,1) coordinate(A3) node[above=0.2em]{\scriptsize $3$};
\maruni{A3}
\path (5,1) coordinate(A4) node[above=0.2em]{\scriptsize $4$};
\maruni{A4}
\path (6,1) coordinate(A5) node[above=0.2em]{\scriptsize $5$};
\maruni{A5}
\draw (1,0) circle(2pt) coordinate(A6) node[below left]{\scriptsize $6$};
\draw (2,0) circle(2pt) coordinate(A7) node[below]{\scriptsize $7$};
\draw (3,0) circle(2pt) coordinate(A8) node[below left]{\scriptsize $8$};
\draw (4,0) circle(2pt) coordinate(A9) node[below]{\scriptsize $9$};
\draw (5,0) circle(2pt) coordinate(A10) node[below right]{\scriptsize $10$};
\draw (3,-1) circle(2pt) coordinate(A11) node[below]{\scriptsize $11$};
\qsarrow{A2}{A1}
\qsarrow{A3}{A2}
\qsarrow{A3}{A4}
\qsarrow{A5}{A4}
\qarrow{A6}{A7}
\qarrow{A8}{A7}
\qarrow{A8}{A9}
\qarrow{A9}{A10}
\draw[->,shorten >=2pt,shorten <=4pt] (A4) -- (A9) [thick];
\draw[->,shorten >=4pt,shorten <=2pt] (A9) -- (A3) [thick];
\draw[->,shorten >=2pt,shorten <=4pt] (A2) -- (A8) [thick];
\draw[->,shorten >=4pt,shorten <=2pt] (A7) -- (A2) [thick];
\draw[->,dashed,shorten >=4pt,shorten <=2pt] (A10) -- (A5) [thick];
\draw[->,dashed,shorten >=2pt,shorten <=4pt] (A1) -- (A6) [thick];
\draw[->, shorten >=4pt,shorten <=4pt] (A1) to [out = 45, in = 135] (A3) [thick];
\draw[->,dashed, shorten >=2pt,shorten <=2pt] (A11) to [out = 165, in = -60] (A6) [thick];
\qarrow{A8}{A11}
\draw[->,dashed, shorten >=2pt,shorten <=2pt] (A11) to [out = 15, in = -120] (A10) [thick];
\draw[->, shorten >=2pt,shorten <=2pt] (A10) to [out = -135, in = -45] (A8) [thick];
\draw[->] (7,0.5) -- (8,0.5);
\path (7.5,0.5) node[below]{$\mu_{4,7}$};
\end{scope}
\begin{scope}[>=latex,xshift=225pt,yshift=-220pt]
\path (2,1) coordinate(A1) node[above=0.2em]{\scriptsize $1$};
\maruni{A1} 
\path (3,1) coordinate(A2) node[above=0.2em]{\scriptsize $2$};
\maruni{A2}
\path (4,1) coordinate(A3) node[above=0.2em]{\scriptsize $3$};
\maruni{A3}
\path (5,1) coordinate(A4) node[above=0.2em]{\scriptsize $4$};
\maruni{A4}
\path (6,1) coordinate(A5) node[above=0.2em]{\scriptsize $5$};
\maruni{A5}
\draw (1,0) circle(2pt) coordinate(A6) node[below left]{\scriptsize $6$};
\draw (2,0) circle(2pt) coordinate(A7) node[below]{\scriptsize $7$};
\draw (3,0) circle(2pt) coordinate(A8) node[below left]{\scriptsize $8$};
\draw (4,0) circle(2pt) coordinate(A9) node[below]{\scriptsize $9$};
\draw (5,0) circle(2pt) coordinate(A10) node[below right]{\scriptsize $10$};
\draw (3,-1) circle(2pt) coordinate(A11) node[below]{\scriptsize $11$};
\qsarrow{A2}{A1}
\qsarrow{A3}{A2}
\qsarrow{A4}{A3}
\qsarrow{A4}{A5}
\qarrow{A7}{A6}
\qarrow{A7}{A8}
\qarrow{A8}{A9}
\qarrow{A9}{A10}
\draw[->,shorten >=4pt,shorten <=2pt] (A9) -- (A4) [thick];
\draw[->,shorten >=2pt,shorten <=4pt] (A2) -- (A7) [thick];
\draw[->,shorten >=2pt,shorten <=4pt] (A5) -- (A9) [thick];
\draw[->,shorten >=4pt,shorten <=2pt] (A6) -- (A2) [thick];
\draw[->,dashed,shorten >=4pt,shorten <=2pt] (A10) -- (A5) [thick];
\draw[->,dashed,shorten >=2pt,shorten <=4pt] (A1) -- (A6) [thick];
\draw[->, shorten >=4pt,shorten <=4pt] (A1) to [out = 45, in = 135] (A3) [thick];
\draw[->,dashed, shorten >=2pt,shorten <=2pt] (A11) to [out = 165, in = -60] (A6) [thick];
\qarrow{A8}{A11}
\draw[->,dashed, shorten >=2pt,shorten <=2pt] (A11) to [out = 15, in = -120] (A10) [thick];
\draw[->, shorten >=2pt,shorten <=2pt] (A10) to [out = -135, in = -45] (A8) [thick];
\end{scope}
\begin{scope}[>=latex,xshift=225pt,yshift=-110pt]
\path (2,1) coordinate(A1) node[above=0.2em]{\scriptsize $1$};
\maruni{A1} 
\path (3,1) coordinate(A2) node[above=0.2em]{\scriptsize $2$};
\maruni{A2}
\path (4,1) coordinate(A3) node[above=0.2em]{\scriptsize $3$};
\maruni{A3}
\path (5,1) coordinate(A4) node[above=0.2em]{\scriptsize $4$};
\maruni{A4}
\path (6,1) coordinate(A5) node[above=0.2em]{\scriptsize $5$};
\maruni{A5}
\draw (1,0) circle(2pt) coordinate(A6) node[below left]{\scriptsize $6$};
\draw (2,0) circle(2pt) coordinate(A7) node[below]{\scriptsize $7$};
\draw (3,0) circle(2pt) coordinate(A8) node[below left]{\scriptsize $8$};
\draw (4,0) circle(2pt) coordinate(A9) node[below]{\scriptsize $9$};
\draw (5,0) circle(2pt) coordinate(A10) node[below right]{\scriptsize $10$};
\draw (3,-1) circle(2pt) coordinate(A11) node[below]{\scriptsize $11$};
\qsarrow{A1}{A2}
\qsarrow{A2}{A3}
\qsarrow{A4}{A3}
\qsarrow{A5}{A4}
\qarrow{A6}{A7}
\qarrow{A7}{A8}
\qarrow{A9}{A8}
\qarrow{A10}{A9}
\draw[->,shorten >=2pt,shorten <=4pt] (A4) -- (A9) [thick];
\draw[->,shorten >=4pt,shorten <=2pt] (A7) -- (A2) [thick];
\draw[->,shorten >=2pt,shorten <=4pt] (A3) -- (A7) [thick];
\draw[->,shorten >=4pt,shorten <=2pt] (A9) -- (A5) [thick];
\draw[->,shorten >=4pt,shorten <=2pt] (A8) -- (A4) [thick];
\draw[->,shorten >=2pt,shorten <=4pt] (A2) -- (A6) [thick];
\draw[->,dashed,shorten >=2pt,shorten <=4pt] (A5) -- (A10) [thick];
\draw[->,dashed,shorten >=4pt,shorten <=2pt] (A6) -- (A1) [thick];
\draw[->,dashed, shorten >=2pt,shorten <=2pt] (A11) to [out = 165, in = -60] (A6) [thick];
\qarrow{A8}{A11}
\draw[->,dashed, shorten >=2pt,shorten <=2pt] (A11) to [out = 15, in = -120] (A10) [thick];
\draw[<-] (3.5,-1.5) -- (3.5,-2.2);
\path (3.5,-1.9) node[right]{$\mu_{2,9}$};
\end{scope}
\begin{scope}[>=latex,xshift=225pt]
\path (1,1) coordinate(A1) node[above=0.2em]{\scriptsize $1$};
\maruni{A1} 
\path (2,1) coordinate(A2) node[above=0.2em]{\scriptsize $2$};
\maruni{A2}
\path (3,1) coordinate(A3) node[above=0.2em]{\scriptsize $3$};
\maruni{A3}
\path (4,1) coordinate(A4) node[above=0.2em]{\scriptsize $4$};
\maruni{A4}
\path (5,1) coordinate(A5) node[above=0.2em]{\scriptsize $5$};
\maruni{A5}
\draw (2,0) circle(2pt) coordinate(A6) node[below left]{\scriptsize $6$};
\draw (3,0) circle(2pt) coordinate(A7) node[below]{\scriptsize $7$};
\draw (4,0) circle(2pt) coordinate(A8) node[below left]{\scriptsize $8$};
\draw (5,0) circle(2pt) coordinate(A9) node[below]{\scriptsize $9$};
\draw (6,0) circle(2pt) coordinate(A10) node[below right]{\scriptsize $10$};
\draw (4,-1) circle(2pt) coordinate(A11) node[below]{\scriptsize $11$};
\qsarrow{A1}{A2}
\qsarrow{A3}{A2}
\qsarrow{A3}{A4}
\qsarrow{A5}{A4}
\qarrow{A6}{A7}
\qarrow{A8}{A7}
\qarrow{A8}{A9}
\qarrow{A10}{A9}
\draw[->,shorten >=4pt,shorten <=2pt] (A9) -- (A5) [thick];
\draw[->,shorten >=2pt,shorten <=4pt] (A4) -- (A8) [thick];
\draw[->,shorten >=4pt,shorten <=2pt] (A7) -- (A3) [thick];
\draw[->,shorten >=2pt,shorten <=4pt] (A2) -- (A6) [thick];
\draw[->,dashed,shorten >=2pt,shorten <=4pt] (A5) -- (A10) [thick];
\draw[->,dashed,shorten >=4pt,shorten <=2pt] (A6) -- (A1) [thick];
\draw[->,dashed, shorten >=2pt,shorten <=2pt] (A11) to [out = 165, in = -60] (A6) [thick];
\draw[->,shorten >=2pt,shorten <=2pt] (A7) to [out = -60, in = 150] (A11) [thick];
\qarrow{A11}{A8}
\draw[->,shorten >=2pt,shorten <=2pt] (A9) to [out = -120, in = 30] (A11) [thick];
\draw[->,dashed, shorten >=2pt,shorten <=2pt] (A11) to [out = 15, in = -120] (A10) [thick];
\draw[<-] (3.5,-1.5) -- (3.5,-2.2);
\path (3.5,-1.9) node[right]{$\mu_{3,8}$};
\end{scope}
\end{tikzpicture}
\caption{The mutation sequence for $\qK_{1234}$.}
\label{K-SB-mutation}
\end{figure}
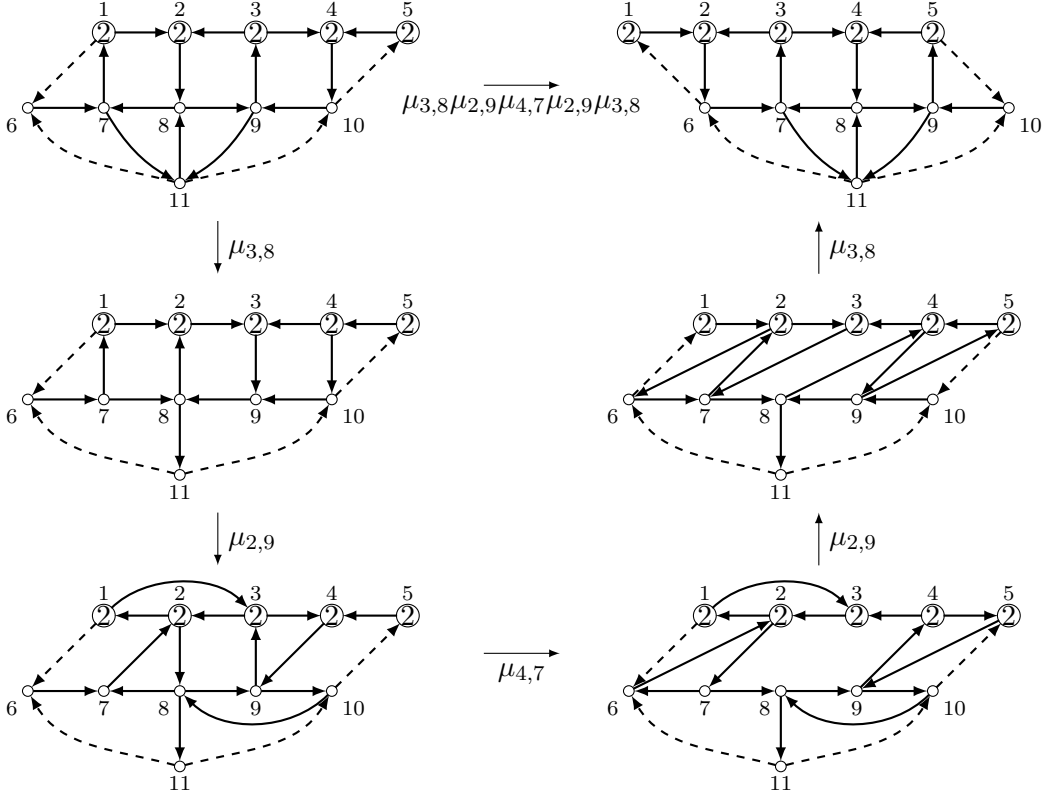

Correspondingly, we define the following sequence of quantum $Y$-seeds:
\begin{align}\label{K-Y-seq}
\begin{split}
(B(C_2),Y) &= (B^{(1)},Y^{(1)}) \stackrel{\mu_8}{\longrightarrow} 
(B^{(2)},Y^{(2)}) \stackrel{\mu_3}{\longrightarrow}
(B^{(3)},Y^{(3)}) \stackrel{\mu_9}{\longrightarrow}
(B^{(4)},Y^{(4)}) \stackrel{\mu_2}{\longrightarrow}
(B^{(5)},Y^{(5)}) 
\\
& \qquad 
\stackrel{\mu_7}{\longrightarrow}
(B^{(6)},Y^{(6)}) \stackrel{\mu_4}{\longrightarrow}
(B^{(7)},Y^{(7)}) \stackrel{\mu_9}{\longrightarrow}
(B^{(8)},Y^{(8)}) \stackrel{\mu_2}{\longrightarrow}
(B^{(9)},Y^{(9)}) 
\\
& \qquad 
\stackrel{\mu_8}{\longrightarrow} 
(B^{(10)},Y^{(10)}) \stackrel{\mu_3}{\longrightarrow}
(B^{(11)},Y^{(11)})= (B'(C_2),Y').
\end{split}
\end{align}
The mutation sequence $\qK_{1234}$ induces the isomorphism of skewfields 
$\widehat{K}_{1234} : \mathcal{Y}(B'(C_2),Y') \to \mathcal{Y}(B(C_2),Y)$.
For a sign sequence $\ve=(\ve_k)_{k=1,\ldots,10} \in \{1,-1\}^{10}$ we set
\begin{align}\label{Khat}
\begin{split}
\widehat{K}_{1234} = & \;
\mathrm{Ad}\bigl(\Psi_q(Y^{(1)\varepsilon_1}_8)^{\varepsilon_1}\bigr) \tau_{8,\varepsilon_1}
\mathrm{Ad}\bigl(\Psi_{q^2}(Y^{(2)\varepsilon_2}_3)^{\varepsilon_2}\bigr) \tau_{3,\varepsilon_2}
\\
& \quad \cdot \mathrm{Ad}\bigl(\Psi_{q}(Y^{(3)\varepsilon_3}_9)^{\varepsilon_3}\bigr) \tau_{9,\varepsilon_3}
\mathrm{Ad}\bigl(\Psi_{q^2}(Y^{(4)\varepsilon_4}_2)^{\varepsilon_4}\bigr) \tau_{2,\varepsilon_4}
\\ 
& \quad \cdot \mathrm{Ad}\bigl(\Psi_q(Y^{(5)\varepsilon_5}_7)^{\varepsilon_5}\bigr) \tau_{7,\varepsilon_5}
\mathrm{Ad}\bigl(\Psi_{q^2}(Y^{(6)\varepsilon_6}_4)^{\varepsilon_6}\bigr) \tau_{4,\varepsilon_6}
\\
& \quad \cdot \mathrm{Ad}\bigl(\Psi_{q}(Y^{(7)\varepsilon_7}_9)^{\varepsilon_7}\bigr) \tau_{9,\varepsilon_7}
\mathrm{Ad}\bigl(\Psi_{q^2}(Y^{(8)\varepsilon_8}_2)^{\varepsilon_8}\bigr) \tau_{2,\varepsilon_8}
\\
& \quad \cdot \mathrm{Ad}\bigl(\Psi_q(Y^{(9)\varepsilon_9}_8)^{\varepsilon_9}\bigr) \tau_{8,\varepsilon_9}
\mathrm{Ad}\bigl(\Psi_{q^2}(Y^{(10)\varepsilon_{10}}_3)^{\varepsilon_{10}}\bigr)\tau_{3,\varepsilon_{10}}.
\end{split}
\end{align}

We define the monomial part $\tau_{1234|\ve}^K$ of $\widehat{K}_{1234}$ by 
\begin{align}\label{Kmono}
\tau_{1234|\ve}^K= \;
\tau_{8, \varepsilon_1}\tau_{3, \varepsilon_2} \tau_{9, \varepsilon_3}\tau_{2, \varepsilon_4}
\tau_{7, \varepsilon_5}\tau_{4, \varepsilon_6}
\tau_{9, \varepsilon_7 }\tau_{2, \varepsilon_8}\tau_{8, \varepsilon_9}
\tau_{3, \varepsilon_{10}}: \mathcal{Y}(B'(C_2),Y') \to \mathcal{Y}(B(C_2),Y).
\end{align}

\subsection{3D reflection equation}

The transformations~$\qR_{ijk}, \overline{\qR}_{ijk}$ and~$\qK_{ijkl}$
satisfy the 3D reflection equation,  
which is realized as an equality between two transformations of the
SB quiver~$B(C_3)$ associated with the longest element
$123123123$ in the Weyl group~$W(C_3)$ into the quiver~$B'(C_3)$
corresponding to $321321321$.
See Figure~\ref{fig:RE}.

\begin{figure}[H]
\scalebox{0.8}{
\begin{tikzpicture}
\begin{scope}[>=latex]
{\color{red}
\fill (1,0) circle(2pt) coordinate(A) node[above right]{$1$};
\fill (2,1) circle(2pt) coordinate(B) node[above right]{$2$};
\fill (3,2) circle(2pt) coordinate(C) node[above=0.2em]{$3$};
\fill (3,0) circle(2pt) coordinate(D) node[above right]{$4$};
\fill (4,1) circle(2pt) coordinate(E) node[above right]{$5$};
\fill (5,2) circle(2pt) coordinate(F) node[above=0.2em]{$6$};
\fill (5,0) circle(2pt) coordinate(G) node[above right]{$7$};
\fill (6,1) circle(2pt) coordinate(H) node[above right]{$8$};
\fill (7,2) circle(2pt) coordinate(I) node[above=0.2em]{$9$};
\draw [-] (0,1.5) to [out = 0, in = 135] (B);
\draw [-] (B) -- (D);
\draw [-] (D) to [out = -45, in = -135] (G); 
\draw [-] (G) -- (I);
\draw [->] (I) to [out = -45, in = 180] (8,1.5);
\draw [-] (0,0.5) to [out = 0, in = 135] (A); 
\draw [-] (A) to [out = -45, in = -135] (D);
\draw [-] (D) -- (F);
\draw [-] (F) -- (H);
\draw [->] (H) to [out = -45, in = 180] (8,0.5);
\draw [-] (0,-0.5) to [out = 0, in = -135] (A); 
\draw [-] (A) -- (C);
\draw [-] (C) -- (G);
\draw [->] (G) to [out = -45, in = 180] (8,-0.5);
}
\path (2,2) node[circle]{2} coordinate(J1) node[above=0.2em]{\scriptsize$1$};
\draw (2,2) circle[radius=0.15];
\path (3,2) node[circle]{2} coordinate(J2) node[above right]{\scriptsize$2$};
\draw (3,2) circle[radius=0.15];
\path (4,2) node[circle]{2} coordinate(J3) node[above=0.2em]{\scriptsize$3$};
\draw (4,2) circle[radius=0.15];
\path (5,2) node[circle]{2} coordinate(J4) node[above right]{\scriptsize$4$};
\draw (5,2) circle[radius=0.15];
\path (6,2) node[circle]{2} coordinate(J5) node[above=0.2em]{\scriptsize$5$};
\draw (6,2) circle[radius=0.15];
\path (7,2) node[circle]{2} coordinate(J6) node[above right]{\scriptsize$6$};
\draw (7,2) circle[radius=0.15];
\path (8,2) node[circle]{2} coordinate(J7) node[above=0.2em]{\scriptsize$7$};
\draw (8,2) circle[radius=0.15];
\qsarrow{J1}{J2}
\qsarrow{J3}{J2}
\qsarrow{J3}{J4}
\qsarrow{J5}{J4}
\qsarrow{J5}{J6}
\qsarrow{J7}{J6}
\draw (1,1) circle(2pt) coordinate(J8) node[left]{\scriptsize$8$};
\draw (2,1) circle(2pt) coordinate(J9) node[below left]{\scriptsize$9$};
\draw (3,1) circle(2pt) coordinate(J10) node[below left]{\scriptsize$10$};
\draw (4,1) circle(2pt) coordinate(J11) node[below left]{\scriptsize$11$};
\draw (5,1) circle(2pt) coordinate(J12) node[below left]{\scriptsize$12$};
\draw (6,1) circle(2pt) coordinate(J13) node[below left]{\scriptsize$13$};
\draw (7,1) circle(2pt) coordinate(J14) node[right]{\scriptsize$14$};
\qarrow{J8}{J9}
\qarrow{J10}{J9}
\qarrow{J10}{J11}
\qarrow{J12}{J11}
\qarrow{J12}{J13}
\qarrow{J14}{J13}
\qarrowsb{J9}{J1}
\qarrowsa{J2}{J10}
\qarrowsb{J11}{J3}
\qarrowsa{J4}{J12}
\qarrowsb{J13}{J5}
\qarrowsa{J6}{J14}
\draw (0,0) circle(2pt) coordinate(J15) node[left]{\scriptsize$15$};
\draw (1,0) circle(2pt) coordinate(J16) node[below]{\scriptsize$16$};
\draw (2,0) circle(2pt) coordinate(J17) node[below left]{\scriptsize$17$};
\draw (3,0) circle(2pt) coordinate(J18) node[below left]{\scriptsize$18$};
\draw (4,0) circle(2pt) coordinate(J19) node[below right]{\scriptsize$19$};
\draw (5,0) circle(2pt) coordinate(J20) node[below]{\scriptsize$20$};
\draw (6,0) circle(2pt) coordinate(J21) node[right]{\scriptsize$21$};
\qarrow{J15}{J16}
\qarrow{J17}{J16}
\qarrow{J17}{J18}
\qarrow{J19}{J18}
\qarrow{J19}{J20}
\qarrow{J21}{J20}
\qarrow{J16}{J8}
\qarrow{J9}{J17}
\qarrow{J18}{J10}
\qarrow{J11}{J19}
\qarrow{J20}{J12}
\qarrow{J13}{J21}
\draw (3,-1) circle(2pt) coordinate(J22) node[below]{\scriptsize$22$};
\qarrow{J18}{J22}
\qarrow{J22}{J19}
\qarrow{J22}{J17}
\draw [->, shorten >=2pt,shorten <=2pt] (J20) to [out = -135, in = 15] (J22) [thick];
\draw [<-, dashed, shorten >=2pt,shorten <=2pt] (J21) to [out = -135, in = 0] (J22) [thick];
\draw [->, shorten >=2pt,shorten <=2pt] (J16) to [out = -45, in = 165] (J22) [thick];
\draw [<-, dashed, shorten >=2pt,shorten <=2pt] (J15) to [out = -45, in = 180] (J22) [thick];
\qdarrow{J21}{J14}
\qdarrow{J8}{J15}
\draw[->,dashed,shorten >=4pt,shorten <=2pt] (J14) -- (J7) [thick];
\draw[->,dashed,shorten >=2pt,shorten <=4pt] (J1) -- (J8) [thick];
\draw (4,-2) node {$B(C_3)$};
\end{scope}
\begin{scope}[>=latex,xshift=270pt]
{\color{red}
\fill (7,0) circle(2pt) coordinate(A) node[above right]{$1$};
\fill (6,1) circle(2pt) coordinate(B) node[above right]{$2$};
\fill (5,2) circle(2pt) coordinate(C) node[above=0.2em]{$3$};
\fill (5,0) circle(2pt) coordinate(D) node[above right]{$4$};
\fill (4,1) circle(2pt) coordinate(E) node[above right]{$5$};
\fill (3,2) circle(2pt) coordinate(F) node[above=0.2em]{$6$};
\fill (3,0) circle(2pt) coordinate(G) node[above right]{$7$};
\fill (2,1) circle(2pt) coordinate(H) node[above right]{$8$};
\fill (1,2) circle(2pt) coordinate(I) node[above=0.2em]{$9$};
\draw [-] (0,1.5) to [out = 0, in = -135] (I);
\draw [-] (I) -- (G);
\draw [-] (G) to [out = -45, in = -135] (D);
\draw [-] (D) -- (B);
\draw [->] (B) to [out = 45, in = 180] (8,1.5);
\draw [-] (0,0.5) to [out = 0, in = -135] (H);
\draw [-] (H) -- (F); 
\draw [-] (F) to [out = 45, in = -135] (F);
\draw [-] (F) -- (D);
\draw [-] (D) to [out = -45, in = -135] (A);
\draw [->] (A) to [out = 45, in = 180] (8,0.5);
\draw [-] (0,-0.5) to [out = 0, in = -135] (G); 
\draw [-] (G) -- (C);
\draw [-] (C) -- (A);
\draw [->] (A) to [out = -45, in = 180] (8,-0.5);
}
\path (0,2) node[circle]{2} coordinate(J1) node[above=0.2em]{\scriptsize$1$};
\draw (0,2) circle[radius=0.15];
\path (1,2) node[circle]{2} coordinate(J2) node[above right]{\scriptsize$2$};
\draw (1,2) circle[radius=0.15];
\path (2,2) node[circle]{2} coordinate(J3) node[above=0.2em]{\scriptsize$3$};
\draw (2,2) circle[radius=0.15];
\path (3,2) node[circle]{2} coordinate(J4) node[above right]{\scriptsize$4$};
\draw (3,2) circle[radius=0.15];
\path (4,2) node[circle]{2} coordinate(J5) node[above=0.2em]{\scriptsize$5$};
\draw (4,2) circle[radius=0.15];
\path (5,2) node[circle]{2} coordinate(J6) node[above right]{\scriptsize$6$};
\draw (5,2) circle[radius=0.15];
\path (6,2) node[circle]{2} coordinate(J7) node[above=0.2em]{\scriptsize$7$};
\draw (6,2) circle[radius=0.15];
\qsarrow{J1}{J2}
\qsarrow{J3}{J2}
\qsarrow{J3}{J4}
\qsarrow{J5}{J4}
\qsarrow{J5}{J6}
\qsarrow{J7}{J6}
\draw (1,1) circle(2pt) coordinate(J8) node[left]{\scriptsize$8$};
\draw (2,1) circle(2pt) coordinate(J9) node[below left]{\scriptsize$18$};
\draw (3,1) circle(2pt) coordinate(J10) node[below left]{\scriptsize$17$};
\draw (4,1) circle(2pt) coordinate(J11) node[below left]{\scriptsize$16$};
\draw (5,1) circle(2pt) coordinate(J12) node[below left]{\scriptsize$10$};
\draw (6,1) circle(2pt) coordinate(J13) node[below left]{\scriptsize$20$};
\draw (7,1) circle(2pt) coordinate(J14) node[right]{\scriptsize$14$};
\qarrow{J8}{J9}
\qarrow{J10}{J9}
\qarrow{J10}{J11}
\qarrow{J12}{J11}
\qarrow{J12}{J13}
\qarrow{J14}{J13}
\qarrowsa{J2}{J8}
\qarrowsb{J9}{J3}
\qarrowsa{J4}{J10}
\qarrowsb{J11}{J5}
\qarrowsa{J6}{J12}
\qarrowsb{J13}{J7}
\draw (2,0) circle(2pt) coordinate(J15) node[left]{\scriptsize$15$};
\draw (3,0) circle(2pt) coordinate(J16) node[below]{\scriptsize$9$};
\draw (4,0) circle(2pt) coordinate(J17) node[below left]{\scriptsize$19$};
\draw (5,0) circle(2pt) coordinate(J18) node[below left]{\scriptsize$13$};
\draw (6,0) circle(2pt) coordinate(J19) node[below right]{\scriptsize$12$};
\draw (7,0) circle(2pt) coordinate(J20) node[below]{\scriptsize$11$};
\draw (8,0) circle(2pt) coordinate(J21) node[right]{\scriptsize$21$};
\qarrow{J15}{J16}
\qarrow{J17}{J16}
\qarrow{J17}{J18}
\qarrow{J19}{J18}
\qarrow{J19}{J20}
\qarrow{J21}{J20}
\qarrow{J9}{J15}
\qarrow{J16}{J10}
\qarrow{J11}{J17}
\qarrow{J18}{J12}
\qarrow{J13}{J19}
\qarrow{J20}{J14}
\draw (5,-1) circle(2pt) coordinate(J22) node[below]{\scriptsize$22$};
\qarrow{J18}{J22}
\qarrow{J22}{J19}
\qarrow{J22}{J17}
\draw [->, shorten >=2pt,shorten <=2pt] (J20) to [out = -135, in = 15] (J22) [thick];
\draw [<-, dashed, shorten >=2pt,shorten <=2pt] (J21) to [out = -135, in = 0] (J22) [thick];
\draw [->, shorten >=2pt,shorten <=2pt] (J16) to [out = -45, in = 165] (J22) [thick];
\draw [<-, dashed, shorten >=2pt,shorten <=2pt] (J15) to [out = -45, in = 180] (J22) [thick];
\qdarrow{J14}{J21}
\qdarrow{J15}{J8}
\draw[<-,dashed,shorten >=4pt,shorten <=2pt] (J14) -- (J7) [thick];
\draw[<-,dashed,shorten >=2pt,shorten <=4pt] (J1) -- (J8) [thick];
\draw (4,-2) node {$B'(C_3)$};
\end{scope}
\end{tikzpicture}
}
\caption{SB quivers connected by the two sides of the 3D reflection equation.}
\label{fig:RE}
\end{figure}
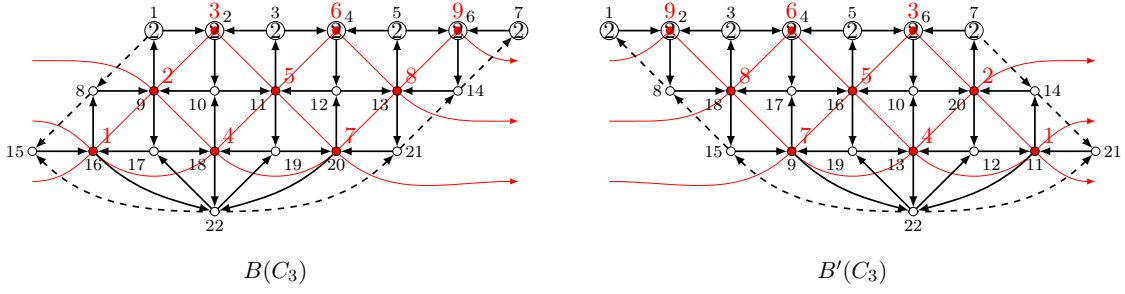

\begin{proposition}\label{pr:resb}
The transformations $\qR_{ijk}$, $\overline{\qR}_{ijk}$ and $\qK_{ijkl}$ satisfy
the 3D reflection equation:
\begin{align}\label{RE-SB}
\begin{split}
&\overline{\qR}_{457} \qK_{4689} \qK_{2379} \qR_{258} \overline{\qR}_{178} \qK_{1356}\qR_{124}(B(C_3),Y)
\\
& \quad = \overline{\qR}_{124} \qK_{1356} \qR_{178} \overline{\qR}_{258} \qK_{2379} \qK_{4689} \qR_{457} (B(C_3),Y).
\end{split}
\end{align}
Accordingly, we have the 3DRE as the isomorphisms of the skewfields from $\mathcal{Y}(B'(C_3))$ to $\mathcal{Y}(B(C_3))$, 
consisiting of $\widehat{R}_{ijk}$, $\widehat{\overline{R}}_{ijk}$ and $\widehat{K}_{ijkl}$ as 
\begin{align}\label{REiso}
\begin{split}
\widehat{R}_{124} \widehat{K}_{1356} \widehat{\overline{R}}_{178}  \widehat{R}_{258} \widehat{K}_{2379} \widehat{K}_{4689}  \widehat{\overline{R}}_{457} 
= \widehat{R}_{457} \widehat{K}_{4689} \widehat{K}_{2379} \widehat{\overline{R}}_{258} \widehat{R}_{178} \widehat{K}_{1356} \widehat{\overline{R}}_{124}.
\end{split}
\end{align}
\end{proposition}  

\begin{proof}
From Figures~\ref{fig:RE1} and~\ref{fig:RE2}, one observes that the
quiver mutations on both sides of~\eqref{RE-SB} coincide.  
For the quantum $Y$-seeds, the equality can be verified at the level of
tropical $y$-variables, and the claim then follows from
Theorem~\ref{thm:period}.
Correspondingly, \eqref{REiso} is obtained as an identity of isomorphisms of skewfields from 
$\mathcal{Y}(B'(C_3))$  to $\mathcal{Y}(B(C_3))$. 
\end{proof}

\begin{figure}[H]
\hspace*{-1.5cm}
\includegraphics[clip,scale=0.9]{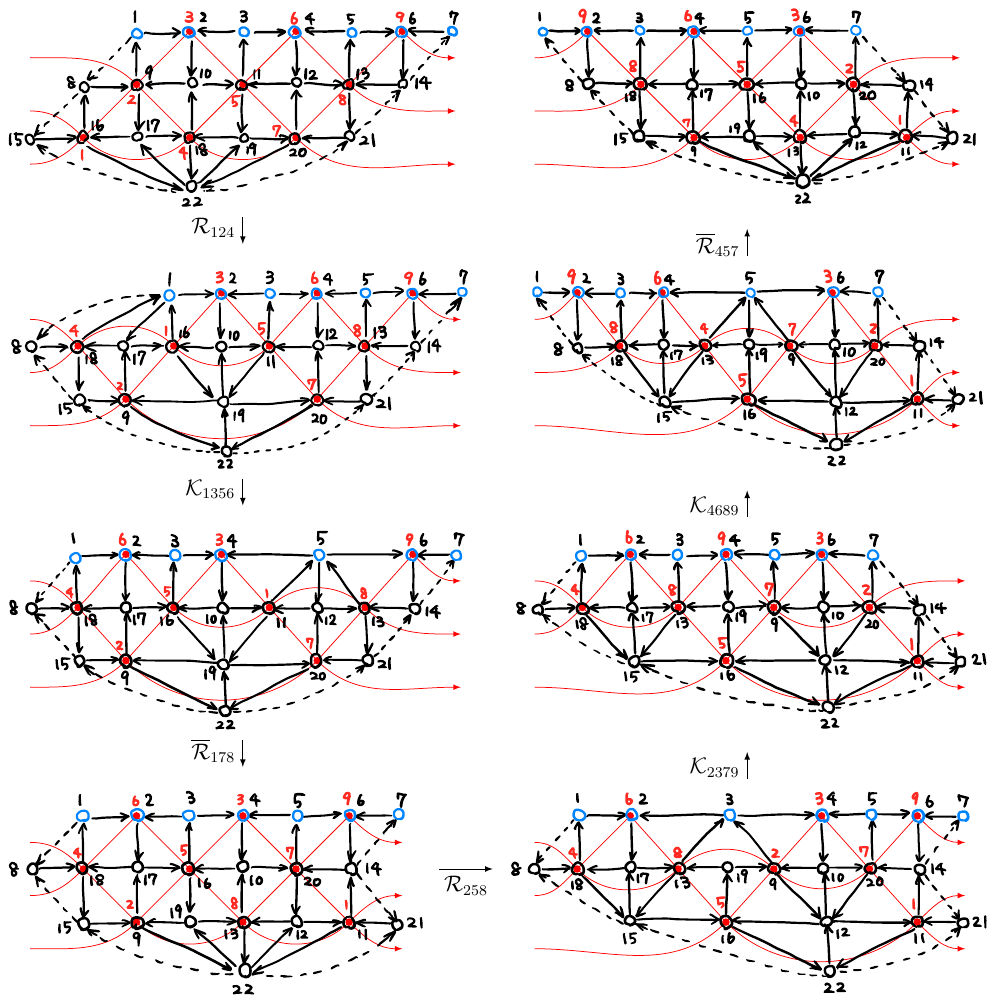}
\vspace*{-9cm}
\caption{The LHS of the 3D reflection equation \eqref{RE-SB}. Blue vertices have weight two.}
\label{fig:RE1}
\end{figure}

\begin{figure}[H]
\hspace*{-1.5cm}
\includegraphics[clip,scale=0.90]{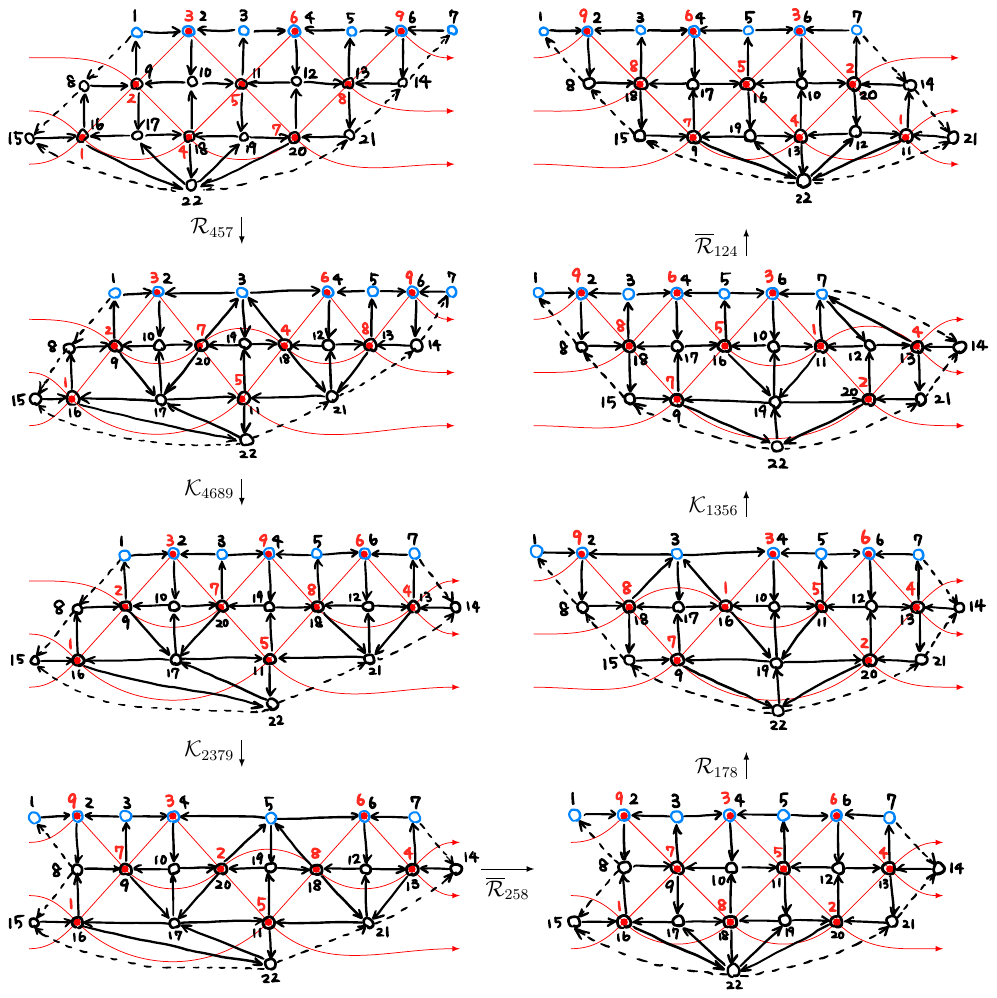}
\vspace*{-9cm}
\caption{The RHS of the 3D reflection equation \eqref{RE-SB}. Blue vertices have weight two.}
\label{fig:RE2}
\end{figure}

\section{$R$-operator from $q$-Weyl algebra}
\label{sec:Rop}

We set $\gamma = (1,1,1)$ and consider canonical variables
$(u_i, w_i)$ satisfying
\begin{align}
\label{uwA}
[u_i, w_j] = \hbar\,\delta_{ij}, 
\qquad 
[u_i, u_j] = [w_i, w_j] = 0,
\end{align}
for $i,j = 1,2,3$.
Let $\mathcal{W}(A_3) := \mathcal{W}_\gamma$ denote the $q$-Weyl algebra
defined in~\S\ref{sec:qW}, and let
$\mathrm{Frac}\,\mathcal{W}(A_2)$ be its noncommutative field of fractions.
Let $\mathscr{P}_i := (a_i,b_i,c_i,d_i,e_i) \in \C^5$ be a tuple of parameters
satisfying
\begin{align}\label{econ-a}
a_i + b_i + c_i + d_i + e_i = 0 \qquad (i=1,2,3).
\end{align}

For the left quiver in~\eqref{SB-R}, 
we define a ring homomorphism of skewfields
$\phi \colon \mathcal{Y}(B(A_2)) \to \mathrm{Frac}\,\mathcal{W}(A_2)$
according to the graphical rule shown in Figure~\ref{fig:para}.
\begin{align}
\phi:
\begin{cases}
Y_0 \mapsto \e^{a_2+w_2},&Y_5 \mapsto \e^{e_1+2u_1}, 
\\
Y_1 \mapsto \e^{a_1+d_2+w_1-u_2-w_2},&Y_6  \mapsto \e^{b_1+c_2+d_3-u_1-w_1+w_2-u_3-w_3},
\\
Y_2 \mapsto \e^{e_2+2u_2},&Y_7 \mapsto \e^{e_3+2u_3},
\\
Y_3 \mapsto \e^{b_2+a_3-u_2-w_2+w_3}, &Y_8 \mapsto \e^{b_3-u_3-w_3},
\\
Y_4 \mapsto \e^{d_1-u_1-w_1}, &Y_9 \mapsto \e^{c_1+c_3+w_1+w_3}. 
\end{cases}
\label{Yw}
\end{align}
Similarly, for the right quiver in \eqref{SB-R}, a ring homomorphism 
$\phi' \colon \mathcal{Y}(B'(A_3)) \rightarrow \mathrm{Frac}\,\mathcal{W}(A_3)$ is defined by 
\begin{align}
\phi':
\begin{cases}
Y'_{0}  \mapsto \e^{a_1+a_3+w_1+w_3},
&Y'_5  \mapsto \e^{e_1+2u_1}, 
\\
Y'_{1}  \mapsto \e^{d_3-u_3-w_3},
&Y'_{6} \mapsto \e^{d_1+a_2+b_3-u_1-w_1+w_2-u_3-w_3},
\\
Y'_{2}  \mapsto \e^{e_2+2u_2},
&Y'_{7} \mapsto \e^{e_3+2u_3},
\\
Y'_{3}  \mapsto \e^{b_1-u_1-w_1},
&Y'_{8}  \mapsto \e^{c_1+b_2+w_1-u_2-w_2},
\\
Y'_{4}  \mapsto \e^{d_2+c_3-u_2-w_2+w_3},
&Y'_{9}  \mapsto \e^{c_2+w_2}.
\end{cases}
\label{Ypw}
\end{align}

\begin{figure}[H]
\begin{tikzpicture}
\begin{scope}[>=latex,xshift=0pt]
{\color{red}
\draw [-] (0,2) to [out = 0, in = 135] (1.8,1.2);
\draw [-] (0,0) to [out = 0, in = -135] (1.8,0.8);
\draw [-] (2.2,0.8) to [out = -45, in = 180] (4,0);
\draw [-] (2.2,1.2) to [out = 45, in = 180] (4,2);
}
%
\draw (0.7,1) circle(2pt) coordinate(A) node[left]{\color{red} \scriptsize{$-u_i-w_i+d_i$}};
\draw (3.3,1) circle(2pt) coordinate(B) node[right]{\color{red} \scriptsize{$-u_i-w_i+b_i$}};
\draw (2,2.3) circle(2pt) coordinate(C) node[above]{\color{red} \scriptsize{$w_i+a_i$}};
\draw (2,-0.3) circle(2pt) coordinate(D) node[below]{\color{red} \scriptsize{$w_i+c_i$}};
\draw (2,1) node {\color{red} \scriptsize{$2u_i+e_i$}};
\draw[->,shorten >=2pt] (2,1.3) -- (C) [thick];
\draw[->,shorten >=2pt] (2,0.7) -- (D) [thick];
\draw[->,shorten <=2pt ] (A) -- (1.4,1) [thick];
\draw[->,shorten <=2pt] (B) -- (2.6,1) [thick];
\qdarrow{C}{A}
\qdarrow{C}{B}
\qdarrow{D}{A}
\qdarrow{D}{B}
\end{scope}
\end{tikzpicture}
\caption{Graphical rule for parametrizing the $Y$-variables 
in terms of the $q$-Weyl algebra generators 
near the crossing $i$ (center) of the wiring diagram (in red).}
\label{fig:para}
\end{figure}
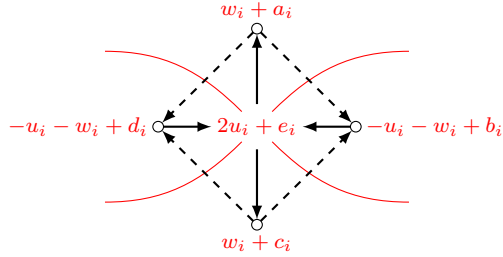

Define the transformations $\eta_{123}^{(-)}$ and $\eta_{123}^{(+)}$ of canonical variables by 
\begin{align}
&\eta_{123}^{(-)}:\; 
\begin{cases}
u_1 \mapsto u_2 +\lambda_0, 
\qquad \qquad \qquad \;\;
w_1  \mapsto w_2-w_3+\lambda_2,
\\
u_2  \mapsto u_1-\lambda_0, 
\qquad \qquad \qquad \;\;
w_2  \mapsto w_1+w_3+\lambda_1,
\\
u_3  \mapsto -u_1+u_2+u_3+\lambda_0, 
\quad \;
w_3  \mapsto w_3+\lambda_3,
\end{cases}
\label{uw-1}
\\[1mm]
&\eta_{123}^{(+)}:\; 
\begin{cases}
u_1 \mapsto u_1+u_2-u_3+\kappa_0, 
\quad \;\;
w_1  \mapsto w_1+ \kappa_1,
\\
u_2  \mapsto u_3-\kappa_0, 
\qquad \qquad \qquad 
w_2  \mapsto w_1+w_3+\kappa_3,
\\
u_3  \mapsto u_2+\kappa_0, 
\qquad \qquad \qquad 
w_3  \mapsto -w_1+w_2+\kappa_2,
\end{cases}
\label{uw-2}
\end{align}
where $\lambda_r=\lambda_r(\mathscr{P}_1,\mathscr{P}_2,\mathscr{P}_3)$
and $\kappa_r=\kappa_r(\mathscr{P}_1,\mathscr{P}_2,\mathscr{P}_3)$
for $r=0$, $1$, $2$, $3$ are defined, under the condition \eqref{econ-a},
by
\begin{align}\label{lad}
\lambda_0 = \frac{e_2-e_1}{2},\quad 
\lambda_1 = c_1-c_2+c_3, \quad
\lambda_2 = -a_3-b_2+b_1-\lambda_0,\quad
\lambda_3 = a_2-a_1+b_2-b_1 + \lambda_0.
\end{align}
and 
\begin{align}\label{kad}
\kappa_0 = \frac{e_2-e_3}{2},\quad 
\kappa_1 =  b_3+c_3-b_2-c_2-\kappa_0, \quad
\kappa_2 =d_3-d_2-a_1 - \kappa_0,\quad
\kappa_3 = c_1-c_2+c_3.
\end{align}
These transformations induce isomorphisms of the algebra~$\mathcal{W}(A_2)$.
They act naturally on $\mathrm{Frac}\,\mathcal{W}(A_2)$,
and we denote this induced action by~$\eta_{123}^{(\pm)}$ as well.

\begin{proposition}(Cf. \cite[\S 4.2]{IKSTY})
\label{prop-Rcom}
(i) The following diagrams are commutative.
\begin{align}\label{Rcom1}
\xymatrix{
\mathcal{Y}(B'(A_2)) \ar[r]^{\phi'} \ar[d]_{\tau_{123| --++}} & \mathrm{Frac}\,\mathcal{W}(A_2) \ar[d]_{\eta_{123}^{(-)}}
\\
\mathcal{Y}(B(A_2)) \ar[r]^{\phi}& \mathrm{Frac}\,\mathcal{W}(A_2)
}
\qquad 
\xymatrix{
\mathcal{Y}(B'(A_2)) \ar[r]^{\phi'} \ar[d]_{\tau_{123| -+-+}} & \mathrm{Frac}\,\mathcal{W}(A_2) \ar[d]_{\eta_{123}^{(+)}}
\\
\mathcal{Y}(B(A_2)) \ar[r]^{\phi}& \mathrm{Frac}\,\mathcal{W}(A_2)
}
\end{align}
(ii) For $\delta = +$ and $-$, the isomorphism $\eta^{(\delta)}_{123}$ 
is realized as $\eta^{(\delta)}_{123} = \mathrm{Ad} P_{123}^{(\delta)}$ by
\begin{align}
\label{P--++}
P_{123}^{(-)}
&=
\e^{\tfrac{1}{\hbar}(u_1-u_2)w_3}
\e^{\tfrac{\lambda_0}{\hbar}(-w_3-w_2+w_1)}
\e^{\tfrac{1}{\hbar}(\lambda_1u_1+\lambda_2u_2+\lambda_3u_3)}\rho_{12},
\\
\label{P-+-+}
P_{123}^{(+)}&= 
\e^{\tfrac{1}{\hbar}(u_3-u_2)w_1}
\e^{\tfrac{\kappa_0}{\hbar}(w_3-w_2-w_1)}
\e^{\tfrac{1}{\hbar}(\kappa_1u_1+\kappa_2u_2+\kappa_3u_3)}\rho_{23}.
\end{align}
Here, $\rho_{ij} \in \mathfrak{S}_3$ acts on~$\mathcal{W}(A_2)$ by adjoint
action, permuting the indices of the canonical variables; for example,
$\rho_{ij} u_i = u_j \rho_{ij}$, and similarly for the other variables.
\end{proposition}

\begin{proof}
(i) By direct computation we check 
$\eta^{(\delta)}_{123} \circ \phi' (Y'_i) = \phi \circ \tau_{123|\ve} (Y'_i)$ for 
$\ve=(-,\delta,-\delta,+)$ with $\delta= \pm$ and $i=0,1,\ldots,9$.
We demonstrate the case of $\delta=+$ and $i=4$, using 
notations $Y_i = \e^{y_i}$ and $Y_i' = \e^{y_i'}$. Using \eqref{tau-+-+}, \eqref{Yw}, \eqref{Ypw} and \eqref{uw-2} we have 
\begin{align*}
y_4' &\stackrel{\phi'}{\longmapsto} d_2+c_3-u_2-w_2+w_3 
\stackrel{\eta^{(+)}_{123}}{\longmapsto} d_2+c_3 -(u_3-\kappa_0) - (w_1+w_3+\kappa_3) +(-w_1+w_2+\kappa_2)
\\
& \qquad = \underline{d_2+c_3+\kappa_0+\kappa_2-\kappa_3} -2w_1+w_2 -u_3-w_3
\\
y_4' &\stackrel{\tau_{123|\ve}}{\longmapsto} y_4+y_5+y_6 
\stackrel{\phi}{\mapsto}(d_1-u_1-w_1) + (e_1+2u_1) + (b_1+c_2+d_3-u_1-w_1+w_2-u_3-w_3)
\\ & \qquad = \underline{d_1+b_1+e_1+c_2+d_3} -2w_1+w_2 -u_3-w_3,  
\end{align*}
where the underlined parts are the same due to \eqref{econ-a} and \eqref{kad}.
\\
(ii) It is proved by computing $\mathrm{Ad} P^{(\delta)}_{123} (Y_i)$ applying the BCH formula.
\end{proof}

Consequently,
for the two sign sequences $\ve = (-, \delta, -\delta, +)$ with $\delta \in \{+,-\}$, the operators $R_{123}^{(\delta)}$ realizing 
$\widehat{R}_{123}$ through 
$\phi \circ \widehat{R}_{123} = \mathrm{Ad}\,R_{123}^{(\delta)} \circ \phi'$ 
are obtained as follows.
\begin{align}
\begin{split}
\label{R--++}
R_{123}^{(-)}
&= 
\Psi_q(\e^{-d_3-c_2-b_1+u_1+u_3+w_1-w_2+w_3})^{-1}
\Psi_q(\e^{-d_3-c_2-b_1-e_1+u_3-u_1+w_1-w_2+w_3})^{-1}
\\
&\quad \cdot
\Psi_q(\e^{d_3+e_3+c_2+b_1+u_3-u_1-w_1+w_2-w_3})
\Psi_q(\e^{d_3+e_3+c_2+e_2+b_1+u_3+2u_2-u_1-w_1+w_2-w_3})
\\
&\quad \cdot P_{123}^{(-)},
\end{split}
\\
\begin{split}
\label{R-+-+}
R_{123}^{(+)}
&= \Psi_q(\e^{-d_3-c_2-b_1+u_1+u_3+w_1-w_2+w_3})^{-1}
\Psi_q(\e^{d_3+c_2+b_1+e_1-u_3+u_1-w_1+w_2-w_3})
\\
& \quad \cdot 
\Psi_q(\e^{-d_3-e_3-c_2-b_1-u_3+u_1+w_1-w_2+w_3})^{-1}
\Psi_q(\e^{d_3+c_2+e_2+b_1+e_1-u_3+2u_2+u_1-w_1+w_2-w_3})
\\
& \quad \cdot P_{123}^{(+)}.
\end{split}
\end{align}

\begin{remark}
In \eqref{R-+-+}, only the last dilogarithm contains $\e^{2u_2}$.
This asymmetry can be remedied by placing $P^{(+)}_{123}$ in the middle,
using $\mathrm{Ad}(P^{(+)}_{123})=\eta^{(+)}_{123}$ together with
\eqref{Rcom1}.  This leads to the expression
\begin{equation}\label{rsigma}
\begin{split}
R^{(+)}_{123}
&=
\Psi_q(\e^{-d_3-c_2-b_1+u_1+u_3+w_1-w_2+w_3})^{-1}
\Psi_q(\e^{d_3+c_2+b_1+e_1-u_3+u_1-w_1+w_2-w_3})P^{(+)}_{123}
\\
&\quad \cdot
\Psi_q(\e^{d_1+e_1+a_2+b_3+u_1-u_3-w_1+w_2-w_3})^{-1}
\Psi_q(\e^{-d_1-a_2-b_3+u_1+u_3-w_1+w_2-w_3}).
\end{split}
\end{equation}
This agrees with \cite[eq.\ (B.2)]{IKSTY} after interchanging the indices
$1$ and $3$.  An analogous rewriting applies to the other $R$-operators
presented in this paper.
\end{remark}

In a similar manner, associated with the transformation~$\overline{R}_{123}$ in~\eqref{SB-Rb}, 
we define ring homomorphisms 
$\overline{\phi}: \mathcal{Y}(B'(A_2)) \to \mathrm{Frac}\,\mathcal{W}(A_2)$ 
and 
$\overline{\phi}': \mathcal{Y}(B(A_2)) \to \mathrm{Frac}\,\mathcal{W}(A_2)$ 
by interchanging the subscripts~$1$ and~$3$ of all canonical variables~$u_i, w_i$ 
and parameters~$\mathscr{P}_i$ in~$\phi'$ and~$\phi$, respectively.
Namely, $\overline{\phi}$ is given by
\begin{align}
\overline{\phi}:
\begin{cases}
\overline{Y}_{0}  \mapsto \e^{a_1+a_3+w_1+w_3},
&\overline{Y}_5  \mapsto \e^{e_3+2u_3}, 
\\
\overline{Y}_{1}  \mapsto \e^{d_1-u_1-w_1},
&\overline{Y}_{6} \mapsto \e^{d_3+a_2+b_1-u_1-w_1+w_2-u_3-w_3},
\\
\overline{Y}_{2}  \mapsto \e^{e_2+2u_2},
&\overline{Y}_{7} \mapsto \e^{e_1+2u_1},
\\
\overline{Y}_{3}  \mapsto \e^{b_3-u_3-w_3},
&\overline{Y}_{8}  \mapsto \e^{c_3+b_2+w_3-u_2-w_2},
\\
\overline{Y}_{4}  \mapsto \e^{d_2+c_1-u_2-w_2+w_1},
&\overline{Y}_{9}  \mapsto \e^{c_2+w_2},
\end{cases}
\label{Ypwb}
\end{align}
and $\overline{\phi}'$ is given by
\begin{align}
\overline{\phi}':
\begin{cases}
\overline{Y}'_0 \mapsto \e^{a_2+w_2},&\overline{Y}'_5 \mapsto \e^{e_3+2u_3}, 
\\
\overline{Y}'_1 \mapsto \e^{a_3+d_2+w_3-u_2-w_2},&\overline{Y}'_6  \mapsto \e^{b_3+c_2+d_1-u_1-w_1+w_2-u_3-w_3},
\\
\overline{Y}'_2 \mapsto \e^{e_2+2u_2},&\overline{Y}'_7 \mapsto \e^{e_1+2u_1},
\\
\overline{Y}'_3 \mapsto \e^{b_2+a_1-u_2-w_2+w_1}, &\overline{Y}'_8 \mapsto \e^{b_1-u_1-w_1},
\\
\overline{Y}'_4 \mapsto \e^{d_3-u_3-w_3}, &\overline{Y}'_9 \mapsto \e^{c_1+c_3+w_1+w_3}. 
\end{cases}
\label{Ywb}
\end{align}
Define the transformations $\overline{\eta}_{123}^{(-)}$ and $\overline{\eta}_{123}^{(+)}$ of canonical variables by
\begin{align}
\label{uwb-1}
&\overline{\eta}_{123}^{(-)}:\;
\begin{cases}
u_1 \mapsto u_1+u_2-u_3+\overline{\lambda}_0, 
\quad \;\;
w_1  \mapsto w_1+ \overline{\lambda}_1,
\\
u_2  \mapsto u_3-\overline{\lambda}_0, 
\qquad \qquad \qquad 
w_2  \mapsto w_1+w_3+\overline{\lambda}_3,
\\
u_3  \mapsto u_2+\overline{\lambda}_0, 
\qquad \qquad \qquad 
w_3  \mapsto -w_1+w_2+\overline{\lambda}_2,
\end{cases}
\\[1mm]
\label{uwb-2}
&\overline{\eta}_{123}^{(+)}:\; 
\begin{cases}
u_1 \mapsto u_2 +\overline{\kappa}_0, 
\qquad \qquad \qquad \;\;
w_1  \mapsto w_2-w_3+\overline{\kappa}_2,
\\
u_2  \mapsto u_1-\overline{\kappa}_0, 
\qquad \qquad \qquad \;\;
w_2  \mapsto w_1+w_3+\overline{\kappa}_1,
\\
u_3  \mapsto -u_1+u_2+u_3+\overline{\kappa}_0, 
\quad \;
w_3  \mapsto w_3+\overline{\kappa}_3,
\end{cases}
\end{align} 
where $\overline{\lambda}_r$ and $\overline{\kappa}_r$ are given by 
\begin{align}\label{lbad}
\overline{\lambda}_0 = \frac{e_2-e_3}{2},\quad 
\overline{\lambda}_1 = c_2-c_3+d_2-d_3 + \overline{\lambda}_0,\quad
\overline{\lambda}_2 = -c_1-d_2+d_3-\overline{\lambda}_0, \quad
\overline{\lambda}_3 = a_1-a_2+a_3.
\end{align}
and
\begin{align}\label{kbad}
\overline{\kappa}_0 = \frac{e_2-e_1}{2},\quad 
\overline{\kappa}_1 = a_1-a_2+a_3,\quad
\overline{\kappa}_2 = -c_3+b_1-b_2 - \overline{\kappa}_0,\quad
\overline{\kappa}_3 =  d_1+a_1-d_2-a_2-\overline{\kappa}_0.
\end{align}

Similarly to~$\eta_{123}^{(\pm)}$, these transformations induce
isomorphisms of~$\mathcal{W}(A_2)$.  
We denote by~$\overline{\eta}_{123}^{(\pm)}$ their corresponding natural
actions on~$\mathrm{Frac}\,\mathcal{W}(A_2)$ as well.
The following proposition is proved in the same manner as
Proposition~\ref{prop-Rcom}.

\begin{proposition}
(i) The following diagrams are commutative.
\begin{align}\label{Rcom2}
\xymatrix{
\mathcal{Y}(B(A_2)) \ar[r]^{\overline{\phi}'} \ar[d]_{\overline{\tau}_{123|--++}} & 
\mathrm{Frac}\,\mathcal{W}(A_2) \ar[d]_{\overline{\eta}_{123}^{(-)}}
\\
\mathcal{Y}(B'(A_2)) \ar[r]^{\overline{\phi}}& \mathrm{Frac}\,\mathcal{W}(A_2)
}
\qquad 
\xymatrix{
\mathcal{Y}(B(A_2)) \ar[r]^{\overline{\phi}'} \ar[d]_{\overline{\tau}_{123|-+-+}} & 
\mathrm{Frac}\,\mathcal{W}(A_2) \ar[d]_{\overline{\eta}_{123}^{(+)}}
\\
\mathcal{Y}(B'(A_2)) \ar[r]^{\overline{\phi}}& \mathrm{Frac}\,\mathcal{W}(A_2)
}
\end{align}
(ii) For $\delta = +$ and $-$, the isomorphism $\overline{\eta}^{(\delta)}_{123}$ 
is realized as $\overline{\eta}^{(\delta)}_{123} = \mathrm{Ad} \overline{P}_{123}^{(\delta)}$ by
\begin{align}
\label{Pb--++}
\overline{P}_{123}^{(-)}&= 
\e^{\tfrac{1}{\hbar}(u_3-u_2)w_1}
\e^{\tfrac{\overline{\lambda}_0}{\hbar}(w_3-w_2-w_1)}
\e^{\tfrac{1}{\hbar}(\overline{\lambda}_1u_1+\overline{\lambda}_2u_2+\overline{\lambda}_3u_3)}\rho_{23},
\\
\label{Pb-+-+}
\overline{P}_{123}^{(+)} 
&=
\e^{\tfrac{1}{\hbar}(u_1-u_2)w_3}
\e^{\tfrac{\overline{\kappa}_0}{\hbar}(-w_3-w_2+w_1)}
\e^{\tfrac{1}{\hbar}(\overline{\kappa}_1u_1+\overline{\kappa}_2u_2+\overline{\kappa}_3u_3)}\rho_{12}.
\end{align}
\end{proposition}

Consequently, for the two sign sequences $\ve = (-,\delta,-\delta,+)$ with $\delta = \mp$,  
the operators $\overline{R}_{123}^{(\mp)}$ and $\overline{P}_{123}^{(\mp)}$,  
which realize $\widehat{\overline{R}}_{123}$ through  
$\overline{\phi} \circ \widehat{\overline{R}}_{123} = \mathrm{Ad}\,\overline{R}_{123}^{(\mp)} \circ \overline{\phi}'$ and  
$\overline{\eta}_{123}^{(\mp)} = \mathrm{Ad}\,\overline{P}_{123}^{(\mp)}$,  
are obtained as follows.
\begin{align}
\begin{split}
\label{Rb--++}
\overline{R}_{123}^{(-)}
&= \Psi_q(\e^{-d_3-a_2-b_1+u_1+u_3+w_1-w_2+w_3})^{-1}
\Psi_q(\e^{-d_3-e_3-a_2-b_1-u_3+u_1+w_1-w_2+w_3})^{-1}
\\
& \quad \cdot
\Psi_q(\e^{d_3+a_2+b_1+e_1-u_3+u_1-w_1+w_2-w_3})
\Psi_q(\e^{d_3+a_2+e_2+b_1+e_1-u_3+2u_2+u_1-w_1+w_2-w_3})
\\
& \quad \cdot \overline{P}_{123}^{(-)},
\end{split}
\\
\begin{split}
\label{Rb-+-+}
\overline{R}_{123}^{(+)} 
&= 
\Psi_q(\e^{-d_3-a_2-b_1+u_1+u_3+w_1-w_2+w_3})^{-1}
\Psi_q(\e^{d_3+e_3+a_2+b_1+u_3-u_1-w_1+w_2-w_3})
\\
&\quad \cdot
\Psi_q(\e^{-d_3-a_2-b_1-e_1+u_3-u_1+w_1-w_2+w_3})^{-1}
\Psi_q(\e^{d_3+e_3+a_2+e_2+b_1+u_3+2u_2-u_1-w_1+w_2-w_3})
\\
&\quad \cdot \overline{P}_{123}^{(+)},
\end{split}
\end{align}

\begin{remark}\label{re:ac}
Under the  interchange of the parameters $a_i$ and $c_i$,  
the operators $R_{123}^{(-)}$~\eqref{R--++}, 
$P_{123}^{(-)}$~\eqref{P--++}
 and $\overline{R}_{123}^{(+)}$~\eqref{Rb-+-+},
 $\overline{P}_{123}^{(+)}$~\eqref{Pb-+-+}  
are exchanged, and likewise 
$R_{123}^{(+)}$~\eqref{R-+-+},
$P_{123}^{(+)}$~\eqref{P-+-+} and  
$\overline{R}_{123}^{(-)}$~\eqref{Rb--++},
$\overline{P}_{123}^{(-)}$~\eqref{Pb--++} are exchanged.
Namely, the following relations hold:
\begin{align*}
&\overline{R}_{123}^{(+)}  = \left. R_{123}^{(-)}\right |_{(a_1,a_2,a_3) \leftrightarrow (c_1,c_2,c_3)} ,
\qquad
\overline{R}_{123}^{(-)} = 
\left. R_{123}^{(+)}\right |_{(a_1,a_2,a_3) \leftrightarrow (c_1,c_2,c_3)},
\\
&\overline{P}_{123}^{(+)}  = \left. P_{123}^{(-)}\right |_{(a_1,a_2,a_3) \leftrightarrow (c_1,c_2,c_3)} ,
\qquad
\overline{P}_{123}^{(-)} = 
\left. P_{123}^{(+)}\right |_{(a_1,a_2,a_3) \leftrightarrow (c_1,c_2,c_3)}.
\end{align*}
\end{remark}

The well-definedness of the dilogarithm parts of the
$R$-operators~\eqref{R--++}, \eqref{R-+-+}, \eqref{Rb--++}, and~\eqref{Rb-+-+}
is ensured by the following proposition.  
Since an analogous argument will be presented later for the
$K$-operator in~\S\ref{sec:wdK}, we omit the proof here.

\begin{proposition}\label{wdR}
The dilogarithm parts of the $R$-operators ${R}_{123}^{(\pm)}$ 
and $\overline{R}_{123}^{(\pm)}$ belong to the set $\mathcal{L}(A_2):= \mathcal{L}_\gamma$.  
\end{proposition}

Set $\gamma = (1,1,1,1,1,1)$, and let $\mathcal{W}(A_3) := \mathcal{W}_\gamma$
be the corresponding $q$-Weyl algebra defined in~\S\ref{sec:qW}.
The homomorphisms $\eta_{123|\ve}$ and $\overline{\eta}_{123|\ve}$ extend
naturally to $\eta_{ijk|\ve}$ and $\overline{\eta}_{ijk|\ve}$ acting on
$\mathcal{W}(A_3)$ and on its noncommutative field of fractions
$\mathrm{Frac}\,\mathcal{W}(A_3)$.

Let $\prec$ and $\prec'$ be two partial orders on the set
$J := \{1,2,\ldots,6\}$ defined by
\[
  1 \prec 2,4 \prec 3,5,6,  
  \qquad 
  1,2,3 \prec' 4,5 \prec' 6.
\]
In the first case (resp.\ the second case), the symmetric subgroup
$S(A_3) := \mathfrak{S}_2 \times \mathfrak{S}_3 \subset \mathfrak{S}_6$
acts on $J$ so as to preserve the order $\prec$ (resp.\ $\prec'$); namely,
$\mathfrak{S}_2$ acts on $\{2,4\}$ and $\mathfrak{S}_3$ on $\{3,5,6\}$
(resp.\ $\mathfrak{S}_2$ on $\{4,5\}$ and $\mathfrak{S}_3$ on $\{1,2,3\}$).
Define the group $N(A_3)$ to be that generated by
\begin{align}
  \e^{\pm \tfrac{1}{\hbar}u_i w_j}\ (i \succ j),\qquad
  \e^{\tfrac{a}{\hbar}u_i},\ \e^{\tfrac{a}{\hbar}w_i}\ (a \in \C),
  \qquad b \in \C^\times,\ i,j \in J,
  \label{Nd}
\end{align}
and the group $N'(A_3)$ to be that generated by
\begin{align}
  \e^{\pm \tfrac{1}{\hbar}u_i w_j}\ (i \prec' j),\qquad
  \e^{\tfrac{a}{\hbar}u_i},\ \e^{\tfrac{a}{\hbar}w_i}\ (a \in \C),
  \qquad b \in \C^\times,\ i,j \in J.
\end{align}
Multiplication in these groups is defined using the (generalized)
Baker--Campbell--Hausdorff formula together with~\eqref{uwA}, which is
well defined due to the grading by~$\hbar^{-1}$.

The group $S(A_3)$ acts on $N(A_3)$ and $N'(A_3)$ by adjoint action,
permuting the indices of the canonical variables.
Finally, let $\mathcal{L}(A_3) := \mathcal{L}_\gamma$ denote the set of
formal Laurent series~\eqref{fLs}.

The operators $P_{123}^{(\mp)}$, $\overline{P}_{123}^{(\mp)}$, 
$R_{123}^{(\mp)}$, and $\overline{R}_{123}^{(\mp)}$ extend to
$P_{ijk}^{(\mp)}$, $\overline{P}_{ijk}^{(\mp)}$,
$R_{ijk}^{(\mp)}$, and $\overline{R}_{ijk}^{(\mp)}$ acting on
$\mathrm{Frac}\,\mathcal{W}(A_3)$.
The tetrahedron relations for these $R$-operators are summarized as follows.

\begin{proposition}(Cf. \cite[\S 3 and Lemma 4.2]{IKSTY})\label{prop:ns}
Corresponding to the two sign sequences
$\ve = (-,-,+,+)$ and $\ve = (-,+,-,+)$, the followings hold.
\begin{itemize}
\item[(i)]
For $\delta = -, +$, the transformations of canonical variables
$\eta_{ijk}^{(\delta)}$ and $\overline{\eta}_{ijk}^{(\delta)}$ each satisfy the
homogeneous tetrahedron equation
\begin{align}
\label{TE-eta}
&\eta_{456}^{(\delta)}\,\eta_{236}^{(\delta)}\,\eta_{135}^{(\delta)}\,\eta_{124}^{(\delta)}
 = \eta_{124}^{(\delta)}\,\eta_{135}^{(\delta)}\,\eta_{236}^{(\delta)}\,\eta_{456}^{(\delta)},\\
\label{TE-etab}
&\overline{\eta}_{456}^{(\delta)}\,\overline{\eta}_{236}^{(\delta)}\,\overline{\eta}_{135}^{(\delta)}\,\overline{\eta}_{124}^{(\delta)}
 = \overline{\eta}_{124}^{(\delta)}\,\overline{\eta}_{135}^{(\delta)}\,\overline{\eta}_{236}^{(\delta)}\,\overline{\eta}_{456}^{(\delta)}.
\end{align}

\item[(ii)]
The operators $P_{ijk}^{(\delta)}$ and $\overline{P}_{ijk}^{(\delta)}$ satisfy the
tetrahedron equation
\begin{align}
\label{TE-P}
&P_{456}^{(+)}\,P_{236}^{(+)}\,P_{135}^{(+)}\,P_{124}^{(+)}
 = P_{124}^{(+)}\,P_{135}^{(+)}\,P_{236}^{(+)}\,P_{456}^{(+)},\\
\label{TE-Pb}
&\overline{P}_{456}^{(-)}\,\overline{P}_{236}^{(-)}\,\overline{P}_{135}^{(-)}\,\overline{P}_{124}^{(-)}
 = \overline{P}_{124}^{(-)}\,\overline{P}_{135}^{(-)}\,\overline{P}_{236}^{(-)}\,\overline{P}_{456}^{(-)}
\end{align}
in the semidirect product $N(A_3)\rtimes S(A_3)$, and
\begin{align}
\label{TE-P2}
&P_{456}^{(-)}\,P_{236}^{(-)}\,P_{135}^{(-)}\,P_{124}^{(-)}
 = P_{124}^{(-)}\,P_{135}^{(-)}\,P_{236}^{(-)}\,P_{456}^{(-)},\\
\label{TE-Pb2}
&\overline{P}_{456}^{(+)}\,\overline{P}_{236}^{(+)}\,\overline{P}_{135}^{(+)}\,\overline{P}_{124}^{(+)}
 = \overline{P}_{124}^{(+)}\,\overline{P}_{135}^{(+)}\,\overline{P}_{236}^{(+)}\,\overline{P}_{456}^{(+)}
\end{align}
in the semidirect product $N'(A_3)\rtimes S(A_3)$.
\end{itemize}
\end{proposition}

\begin{theorem}\label{thm:TE}
(Cf. \cite[Theorem 4.3]{IKSTY})
The operators $R_{ijk}^{(\delta)}$ and $\overline{R}_{ijk}^{(\delta)}$ satisfy the
tetrahedron equation
\begin{align}
\label{TE-R}
&R_{456}^{(+)}\,R_{236}^{(+)}\,R_{135}^{(+)}\,R_{124}^{(+)}
 = R_{124}^{(+)}\,R_{135}^{(+)}\,R_{236}^{(+)}\,R_{456}^{(+)},\\
\label{TE-Rb}
&\overline{R}_{456}^{(-)}\,\overline{R}_{236}^{(-)}\,\overline{R}_{135}^{(-)}\,\overline{R}_{124}^{(-)}
 = \overline{R}_{124}^{(-)}\,\overline{R}_{135}^{(-)}\,\overline{R}_{236}^{(-)}\,\overline{R}_{456}^{(-)},
\end{align}
in the following sense:  
each side decomposes into a dilogarithmi part in $\mathcal{L}(A_3)$ and a
monomial transformation part in $N(A_3)\rtimes S(A_3)$, and both
components agree.

They also satisfy the tetrahedron equation
\begin{align}
\label{TE-R2}
&R_{456}^{(-)}\,R_{236}^{(-)}\,R_{135}^{(-)}\,R_{124}^{(-)}
 = R_{124}^{(-)}\,R_{135}^{(-)}\,R_{236}^{(-)}\,R_{456}^{(-)},\\
\label{TE-Rb2}
&\overline{R}_{456}^{(+)}\,\overline{R}_{236}^{(+)}\,\overline{R}_{135}^{(+)}\,\overline{R}_{124}^{(+)}
 = \overline{R}_{124}^{(+)}\,\overline{R}_{135}^{(+)}\,\overline{R}_{236}^{(+)}\,\overline{R}_{456}^{(+)},
\end{align}
again in the sense that each side admits the same decomposition into a
dilogarithmi part in $\mathcal{L}(A_3)$ and a monomial part in
$N'(A_3)\rtimes S(A_3)$.
\end{theorem}

See \cite{IKSTY} for the proofs of
\eqref{TE-eta}, \eqref{TE-P}, \eqref{TE-P2}, \eqref{TE-R}, and
\eqref{TE-R2}.  
We remark that \eqref{TE-R} and \eqref{TE-R2} are consequences of
\eqref{TE-P} and \eqref{TE-P2}, together with the well-definedness of the
dilogarithmi part of the tetrahedron relation
\cite[Proposition\,3.9]{IKSTY}.  
The remaining identities are established in an analogous manner.

\section{$K$-operator from $q$-Weyl algebra} \label{sec:k}

\subsection{$K$-operator}\label{ss:Kop} 

For $i=1,2,3,4$, let $(u_i,w_i)$ be canonical variables satisfying
(cf.\,\cite{IKT1})
\begin{align}
\label{uw-def}
[u_i, w_j] &= 
\begin{cases}
\hbar\, \delta_{ij}, & i,j = 1,3,\\[2mm]
2\hbar\, \delta_{ij}, & i,j = 2,4,
\end{cases}
\end{align}
and all other commutators vanish.
Accordingly, set $\gamma = (1,2,1,2)$ and let
$\mathcal{W}(C_2) := \mathcal{W}_\gamma$ be the $q$-Weyl algebra
generated by $\e^{\pm u_i}$ and $\e^{\pm w_i}$ with relations
\eqref{qW-gamma}.
For $i=1,2,3,4$, let  
$\mathscr{P}_i = (a_i,b_i,c_i,d_i,e_i) \in \C^5$ be a tuple of parameters.  
For the sake of uniform notation we keep five symbols for every
$\mathscr{P}_i$, although, in the parametrizations given below,
the entries $a_2$ and $a_4$ actually do not appear.
For the bulk part of the quiver, we assume 
\begin{align}\label{econ}
a_i+b_i+c_i+d_i+e_i=0\;\; (i=1,3).
\end{align}

Following the graphical rule in Figure \ref{fig:para} and \ref{fig:paraC}, 
we define a map $\phi: \mathcal{Y}(B(C_2)) \to$ Frac$\mathcal{W}(C_2)$ by 
\begin{align}
\label{YwC2}
\begin{cases}
 Y_1 \mapsto \exp(-u_2+2w_1-w_2+a_1+d_2),
\\
Y_2 \mapsto \exp(2u_2+e_2),
\\
Y_3 \mapsto \exp(-u_2-u_4-w_2+2w_3-w_4+a_3+b_2+d_4),
\\
Y_4 \mapsto \exp(2u_4+e_4),
\\
Y_5 \mapsto \exp(-u_4-w_4+b_4),
\\
Y_6 \mapsto \exp(-u_1-w_1+d_1),
\\
Y_7 \mapsto \exp(2u_1+e_1)
\\
Y_8\mapsto \exp(-u_1-u_3-w_1+w_2-w_3+b_1+c_2+d_3),
\\
Y_9 \mapsto \exp(2u_3+e_3),
\\
Y_{10}\mapsto \exp(-u_3-w_3+w_4+b_3+c_4),
\\
Y_{11} \mapsto\exp(w_1+w_3+c_1+c_3).
\end{cases}
\end{align}
We also define a map $\phi': \mathcal{Y}(B'(C_2)) \to$ Frac$W(C_2)$ by
\begin{align}
\label{YpwC2}
\begin{cases}
Y'_1 \mapsto \exp(-u_4-w_4+d_4),
\\
Y'_2 \mapsto \exp(2u_4+e_4),
\\
Y'_3 \mapsto \exp(-u_2-u_4-w_2+2w_3-w_4+a_3+d_2+b_4),
\\
Y'_4 \mapsto \exp(2u_2+e_2),
\\
Y'_5 \mapsto \exp(-u_2-w_2+2w_1+a_1+b_2),
\\
Y'_6 \mapsto \exp(-u_3-w_3+w_4+d_3+c_4),
\\
Y'_7 \mapsto \exp(2u_3+e_3)
\\
Y'_8 \mapsto \exp(-u_1-u_3-w_1+w_2-w_3+d_1+c_2+b_3),
\\
Y'_9 \mapsto \exp(2u_1+e_1),
\\
Y'_{10}\mapsto \exp(-u_1-w_1+b_1),
\\
Y'_{11} \mapsto\exp(w_1+w_3+c_1+c_3).
\end{cases}
\end{align}

\begin{figure}[H]
\begin{tikzpicture}
\begin{scope}[>=latex,xshift=0pt]
{\color{red}
\draw [-] (0,2) to [out = 0, in = 135] (1.8,1.2);
\draw [-] (0,0) to [out = 0, in = -135] (1.8,0.8);
\draw [-] (2.2,0.8) to [out = -45, in = 180] (4,0);
\draw [-] (2.2,1.2) to [out = 45, in = 180] (4,2);
}
%
\draw (0.7,1) circle(2pt) coordinate(A) node[left]{\color{red} \scriptsize{$-u_i-w_i+d_i$}};
\draw (3.3,1) circle(2pt) coordinate(B) node[right]{\color{red} \scriptsize{$-u_i-w_i+b_i$}};
\draw (2,2.3) coordinate(C) node[above=1pt]{\color{red} \scriptsize{$2 w_i+a_i$}}; 
\maruni{C};
\draw (2,-0.3) circle(2pt) coordinate(D) node[below]{\color{red} \scriptsize{$w_i+c_i$}};
\draw (2,1) node {\color{red} \scriptsize{$2u_i+e_i$}};
\draw[->,shorten >=4pt] (2,1.3) -- (C) [thick];
\draw[->,shorten >=2pt] (2,0.7) -- (D) [thick];
\draw[->,shorten <=2pt ] (A) -- (1.4,1) [thick];
\draw[->,shorten <=2pt] (B) -- (2.6,1) [thick];
\draw[->,dashed,shorten >=2pt,shorten <=4pt] (C) -- (A) [thick];
\draw[->,dashed,shorten >=2pt,shorten <=4pt] (C) -- (B) [thick];
\qdarrow{D}{A}
\qdarrow{D}{B}
\end{scope}
\begin{scope}[>=latex,xshift=200pt, yshift=30pt]
{\color{red}
\draw [-] (0,0) to [out = 0, in = -135] (1.8,0.8);
\draw [-] (2.2,0.8) to [out = -45, in = 180] (4,0);
}
%
\draw (0.7,1) coordinate(A) node[left=1pt]{\color{red} \scriptsize{$-u_i-w_i+d_i$}};
\maruni{A};
\draw (3.3,1) coordinate(B) node[right=1pt]{\color{red} \scriptsize{$-u_i-w_i+b_i$}};
\maruni{B};
\draw (2,-0.3) circle(2pt) coordinate(D) node[below]{\color{red} \scriptsize{$w_i+c_i$}};
\draw (2,1) node {\color{red} \scriptsize{$2u_i+e_i$}};
\draw[->,shorten >=2pt] (2,0.7) -- (D) [thick];
\draw[->,shorten <=4pt ] (A) -- (1.4,1) [thick];
\draw[->,shorten <=4pt] (B) -- (2.6,1) [thick];
\draw[->,dashed,shorten >=4pt,shorten <=2pt] (D) -- (A) [thick];
\draw[->,dashed,shorten >=4pt,shorten <=2pt] (D) -- (B) [thick];
\end{scope}
\end{tikzpicture}
\caption{Graphical rule for parametrizing the $Y$-variables in terms of
$q$-Weyl algebra generators.  
The left diagram shows the neighborhood of a crossing~$i$ (at the center)
in the wiring diagram (in red), while the right diagram shows the
neighborhood of a reflection point~$i$ (at the top center) on the wall,
together with the adjacent vertices of weight~$2$.}
\label{fig:paraC}
\end{figure}
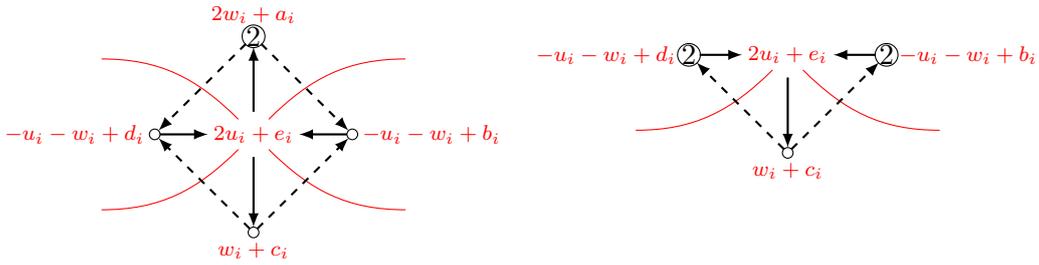

The following lemma is easily checked.

\begin{lemma}
The maps $\phi$ and $\phi'$ are ring homomorphisms of skewfields.
\end{lemma}

Recall the monomial part $\tau^{K}_{1234|\ve}$ \eqref{Kmono} of $\widehat{K}_{1234}$.
Let $\eta^{K}_{1234|\ve}$ be the map on $\mathrm{Frac}\,\mathcal{W}(C_2)$
characterized by the commutative diagram
\begin{align}\label{Kcom}
\xymatrix{
\mathcal{Y}(B'(C_2)) \ar[r]^{\phi'} \ar[d]_{\tau^{K}_{1234|\ve}} &
\mathrm{Frac}\,\mathcal{W}(C_2) \ar[d]_{\eta^{K}_{1234|\ve}}
\\
\mathcal{Y}(B(C_2)) \ar[r]^{\phi} &
\mathrm{Frac}\,\mathcal{W}(C_2)
}
\end{align}

We look for a sign sequence $\ve=(\ve_1,\ldots, \ve_{10})$ for which the following hold:
\begin{itemize}
\item[(i)]
There exists an operator $P^{K}_{1234|\ve}$ realizing $\eta^{K}_{1234|\ve}$
as an adjoint action, that is,
\begin{align}\label{adpe}
\mathrm{Ad}(P^{K}_{1234|\ve})=\eta^{K}_{1234|\ve}.
\end{align}

\item[(ii)]
The monomial transformation $\tau^{K}_{1234|\ve}$
satisfies the 3D reflection equation together with one of the
solutions $\tau_{123|-\mp\pm+}$ or $\overline{\tau}_{123|-\mp\pm+}$
of the tetrahedron equation studied in~\S\ref{sec:Rop}.
\end{itemize}

In this section, we study condition (i).
Our aim is to realize $P^K_{1234|\ve}$ in the form
\begin{align}
&P_{1234|\ve}^K = \exp(\tfrac{1}{\hbar}X) \rho,
\label{PK}
\\
&\quad X = 
\begin{cases}
\displaystyle{\sum_{i=2}^4 A_{i}u_i w_1+\sum_{i=1}^4(B_i u_i+C_iw_i)} \text{ for $\rho=\rho_{24}$},
\\
\displaystyle{\sum_{i=1}^3 A_{i}u_i w_4+\sum_{i=1}^4(B_i u_i+C_iw_i)} \text{ for $\rho=\rho_{13}$},
\end{cases}
\label{PK-X}
\end{align}
where $A_i$, $B_i$, and $C_i$ are coefficients, and
$\rho_{ij}\in\mathfrak{S}_4$ acts by permuting the indices of the canonical variables.
A direct computer calculation yields the following result.

\begin{proposition}
If the parameters $a_i,b_i,c_i,d_i,e_i~(i=1,2,3,4)$ satisfy
\begin{align}\label{ccon}
&b_2+2c_2+d_2+e_2 = c_1 + c_3,\qquad
  b_4+2c_4+d_4+e_4 = -c_1 + c_3,
\end{align}
in addition to~\eqref{econ}, then an operator $P^{K}_{1234|\ve}$ of the
form~\eqref{PK}--\eqref{PK-X} that meets the condition~\eqref{adpe} exists
precisely for the following eight choices of the sign sequence $\ve=(\ve_k)_{k=1,\ldots, 10}$:
\begin{align}
\label{tau24}
&(-1,\varepsilon_2,-1,\varepsilon_4,1,-1,-1,-\varepsilon_4,1,-\varepsilon_2),
\qquad 
\varepsilon_2, \varepsilon_4 \in \{1,-1\} 
\quad\text{(type $\rho_{24}$)},
\\[2mm]
\label{tau13}
&(\varepsilon_1,-1,\varepsilon_3,-1,-1,1,-\varepsilon_3,-1,-\varepsilon_1,1),
\qquad 
\varepsilon_1, \varepsilon_3 \in \{1,-1\}
\quad\text{(type $\rho_{13}$)}.
\end{align}
\end{proposition}

See \S \ref{ad24} and \S \ref{ad13} for the detail of $\eta^K_{1234|\ve}$ and  $P^K_{1234|\ve}$.
For these choices of $\ve$, the monomial part $\tau^K_{1234|\ve}$ \eqref{Kmono} is given as follows. 
For the type $\rho_{24}$ in \eqref{tau24}, we have 
\begin{align}
\label{tauK24}
\tau^K_{1234|\ve} :
\begin{cases}
Y_1' \mapsto Y_1, & Y_6' \mapsto Y_6 Y_7 Y_8, 
\\
Y_2' \mapsto Y_2, & Y_7' \mapsto Y_9,
\\
Y_3' \mapsto Y_3, & Y_8' \mapsto q^2 Y_2^{-1} Y_3^{-1} Y_7^{-1} Y_8^{-1} Y_9^{-1},     
\\
Y_4' \mapsto Y_4, & Y_9' \mapsto Y_2 Y_4^{-1} Y_7, 
\\
Y_5' \mapsto Y_5, & Y_{10}' \mapsto Y_3 Y_4 Y_8 Y_9 Y_{10},
\\
Y_{11}' \mapsto Y_{11},
\end{cases}
\end{align}
which is independent of $\ve_2$ and $\ve_4$.
For the type $\rho_{13}$ in \eqref{tau13} we have  
\begin{align}
\tau^K_{1234|\ve} :
\begin{cases}
Y_1' \mapsto q^4 Y_1 Y_2 Y_3 Y_7^2 Y_8^2, & Y_6' \mapsto Y_6, 
\\
Y_2' \mapsto Y_4 Y_7^{-2} Y_9^2, & Y_7' \mapsto Y_7,
\\
Y_3' \mapsto q^{-4} Y_2^{-1} Y_3^{-1} Y_4^{-1} Y_8^{-2} Y_9^{-2}, & Y_8' \mapsto Y_8,     
\\
Y_4' \mapsto Y_2, & Y_9' \mapsto Y_9, 
\\
Y_5' \mapsto Y_3 Y_4 Y_5, & Y_{10}' \mapsto Y_{10},
\\
Y_{11}' \mapsto Y_{11},
\end{cases}
\end{align}
which is independent of $\ve_1$ and $\ve_3$.

For the choices of $\ve$ in \eqref{tau24} and \eqref{tau13}, we introduce
\begin{equation}\label{ad-K}
\begin{split}
K_{1234|\ve} ={}&\;
\Psi_q\!\bigl(Y^{(1)\varepsilon_1}_8\bigr)^{\varepsilon_1}
\Psi_{q^2}\!\bigl(Y^{(2)\varepsilon_2}_3\bigr)^{\varepsilon_2}
\Psi_q\!\bigl(Y^{(3)\varepsilon_3}_9\bigr)^{\varepsilon_3}
\Psi_{q^2}\!\bigl(Y^{(4)\varepsilon_4}_2\bigr)^{\varepsilon_4}
\\
&\cdot
\Psi_q\!\bigl(Y^{(5)\varepsilon_5}_7\bigr)^{\varepsilon_5}
\Psi_{q^2}\!\bigl(Y^{(6)\varepsilon_6}_4\bigr)^{\varepsilon_6}
\Psi_q\!\bigl(Y^{(7)\varepsilon_7}_9\bigr)^{\varepsilon_7}
\\
&\cdot
\Psi_{q^2}\!\bigl(Y^{(8)\varepsilon_8}_2\bigr)^{\varepsilon_8}
\Psi_q\!\bigl(Y^{(9)\varepsilon_9}_8\bigr)^{\varepsilon_9}
\Psi_{q^2}\!\bigl(Y^{(10)\varepsilon_{10}}_3\bigr)^{\varepsilon_{10}}
P_{1234|\ve}^K,
\end{split}
\end{equation}
where $Y^{(j)}_{k_j}$ here means   
$\phi \circ \tau_{k_1,\ve_1}\cdots\tau_{k_{j-1},\ve_{j-1}}(Y^{(j)}_{k_j})$
with $(k_1,\ldots,k_{10})=(8,3,9,2,7,4,9,2,8,3)$.
The operator \eqref{ad-K} realizes $\widehat{K}_{1234}$ through its adjoint action on 
$\mathrm{Frac}\,\mathcal{W}(C_2)$.
This fact follows from \eqref{Khat}, \eqref{Kmono}, \eqref{Kcom} and \eqref{adpe}.

Note that for each $i=1,\ldots,5$, the quantum dilogarithms in the
$(2i\!-\!1)$th and $2i$th positions commute with one another, reflecting
the commutativity of the corresponding mutations.
In the sequel, we describe the detailed structure of $K_{1234|\ve}$ 
for the two cases \eqref{tau24} and \eqref{tau13}.

\subsection{Type $\rho_{24}$}\label{ad24}

When we set 
$\ve = (-1,\varepsilon_2,-1,\varepsilon_4,1,-1,-1,-\varepsilon_4,1,-\varepsilon_2)$
as in \eqref{tau24}, the map $\eta^K_{1234|\ve}$ making the diagram
\eqref{Kcom} commute is actually independent of $\varepsilon_2$ and
$\varepsilon_4$. Henceforth we denote it simply by $\eta^K_{1234}$. 
It is given by
\begin{align}
\label{tau24uw}
\eta^K_{1234}:
\begin{cases}
u_1\mapsto \frac{1}{2} \left(-b_2+b_4+2 c_1-2 c_2+2 c_4-d_2+d_4 \right)+ u_1+ u_2- u_4,
\\
u_2 \mapsto \frac{1}{2} \left(b_2-b_4-2 c_1+2 c_2-2 c_4+d_2-d_4 \right)+u_4, 
\\
u_3 \mapsto u_3, 
\\
u_4  \mapsto \frac{1}{2} \left(-b_2+b_4+2 c_1-2 c_2+2 c_4-d_2+d_4\right)+ u_2,
\\
w_1 \mapsto \frac{1}{2} \left(-b_2+b_4+d_2-d_4\right)+w_1,
\\
w_2 \mapsto \frac{1}{2} \left(2 a_1-b_2+b_4+2 c_1-2 c_2+2 c_4+d_2-d_4 \right)+2w_1+w_4,
\\
w_3 \mapsto \frac{1}{2} \left(b_2-b_4-d_2+d_4\right)+w_3,
\\
w_4 \mapsto \frac{1}{2} \left(-2 a_1+b_2-b_4-2 c_1+2 c_2-2 c_4-d_2+d_4\right)-2 w_1+ w_2.
\end{cases}
\end{align}
This map defines an
automorphism of~$\mathcal{W}(C_2)$ and thus acts naturally on
$\mathrm{Frac}\,\mathcal{W}(C_2)$.
By a direct computation using the BCH formula, the following lemma is proved.

\begin{lemma}
The operator $P_{24} := P^K_{1234|\ve}$ in~\eqref{PK} that
realizes the automorphism $\eta^K_{1234}$ in~\eqref{tau24uw} is given by    
\begin{equation}\label{p24}
\begin{split}
P_{24}&= \exp\!\left(\frac{1}{2\hbar}X_{24}\right)\rho_{24},
\\
X_{24} &=
(-u_2+u_4)(a_1+c_1-c_2+c_4+2w_1)
-(u_1-u_3)(b_2-b_4-d_2+d_4)
\\
&\quad +(-w_2+w_4)\Bigl(\tfrac12(-b_2+b_4-d_2+d_4)+c_1-c_2+c_4\Bigr).
\end{split}
\end{equation}
\end{lemma}

In \eqref{pf24}, we will see that this rather complicated expression
simplifies substantially once the necessary parameter constraints to fulfill the condition (ii) in \S \ref{ss:Kop} are imposed.

For each choice of $\ve$ in~\eqref{tau24}, the operator $K_{1234|\ve}$ in~\eqref{ad-K} 
takes the following explicit form, 
where every $Y_i$ is understood as its image under~$\phi$ in \eqref{YwC2}.  
For brevity, we denote $K_{1234|\ve}$ simply by $K_{\ve_2,\ve_4}$.
\begin{equation}
\begin{split}
\label{sol241}
\ve=&(-1, 1, -1, 1, 1, -1, -1, -1, 1, -1) ~~(\text{i.e. }\ve_2=\ve_4= 1);
\\
K_{++}&=
\Psi _q\left(Y_8^{-1}\right)^{-1}
\Psi_{q^2}\left(Y_3\right)
   \Psi_q\left(qY_3^{-1}Y_8^{-1}Y_9^{-1}\right)^{-1} \Psi_{q^2}\left(Y_2\right)
   \\
 &\quad \cdot \Psi_q\left(qY_7Y_8\right)
   \Psi_{q^2}\left(Y_3^{-2}Y_4^{-1}Y_8^{-2}Y_9^{-2}\right)^{-1}
   \Psi_q\left(q^{-1}Y_3^{-1}Y_4^{-1}Y_8^{-1}Y_9^{-1}\right)^{-1}   
   \\
&\quad \cdot \Psi _{q^2}\left(Y_2\right)^{-1}
   \Psi_q\left(q^{2}Y_2Y_3Y_7Y_8Y_9\right)
   \Psi_{q^2}\left(Y_3\right)^{-1}P_{24}.
 \end{split}
\end{equation}
\begin{equation}
\begin{split}
\label{sol242}
\ve =&(-1, 1, -1, -1, 1, -1, -1, 1, 1, -1) ~~(\text{i.e. }\ve_2=-\ve_4= 1);
\\
K_{+-}&=\Psi _q\left(Y_8^{-1}\right)^{-1}
\underline{\Psi_{q^2}\left(Y_3\right)
   \Psi_q\left(qY_3^{-1}Y_8^{-1}Y_9^{-1}\right)^{-1}
   \Psi_{q^2}\left(Y_2^{-1}\right)^{-1}}
   \\
& \quad \cdot   \underline{\Psi_q\left(q^{-1}Y_2Y_7Y_8\right)
   \Psi_{q^2}\left(Y_3^{-2}Y_4^{-1}Y_8^{-2}Y_9^{-2}\right)^{-1}
   } \Psi_q\left(q^{-1}Y_3^{-1}Y_4^{-1}Y_8^{-1}Y_9^{-1}\right)^{-1}
   \\
& \quad \cdot \underline{\Psi _{q^2}\left(Y_2^{-1}\right)} 
   \Psi_q\left(q^{2}Y_2Y_3Y_7Y_8Y_9\right)
   \underline{\Psi_{q^2}\left(Y_3\right)^{-1}}P_{24}.
    \end{split}
\end{equation}
\begin{equation}
\label{sol243}
\begin{split}
\ve=&(-1, -1, -1, 1, 1, -1, -1, -1, 1, 1) ~~(\text{i.e. }-\ve_2=\ve_4= 1); 
\\
K_{-+}&=\Psi _q\left(Y_8^{-1}\right)^{-1}
\underline{\Psi_{q^2}\left(Y_3^{-1}\right)^{-1}
   \Psi_q\left(q^{-1}Y_8^{-1}Y_9^{-1}\right)^{-1}
   \Psi_{q^2}\left(q^{2}Y_2Y_3\right)}
   \\
& \quad \cdot  \underline{\Psi_q\left(qY_7Y_8\right)
   \Psi_{q^2}\left(q^{-2}Y_3^{-1}Y_4^{-1}Y_8^{-2}Y_9^{-2}\right)^{-1}}
   \Psi_q\left(q^{-1}Y_3^{-1}Y_4^{-1}Y_8^{-1}Y_9^{-1}\right)^{-1}
   \\
& \quad \cdot \underline{\Psi_{q^2}\left(q^{2}Y_2Y_3\right)^{-1}} 
   \Psi_q\left(q^{2}Y_2Y_3Y_7Y_8Y_9\right)
   \underline{\Psi_{q^2}\left(Y_3^{-1}\right)}P_{24}.
 \end{split}
\end{equation}
\begin{equation}
\label{sol244}
\begin{split}
\ve=&(-1, -1, -1, -1, 1, -1, -1, 1, 1, 1) ~~(\text{i.e. }\ve_2=\ve_4= -1);
\\
K_{--}&=\Psi _q\left(Y_8^{-1}\right)^{-1}
\Psi_{q^2}\left(Y_3^{-1}\right)^{-1}
   \Psi_q\left(q^{-1}Y_8^{-1}Y_9^{-1}\right)^{-1}
   \Psi_{q^2}\left(q^{2}Y_2^{-1}Y_3^{-1}\right)^{-1}
   \\
& \quad \cdot \Psi_q\left(qY_2Y_3Y_7Y_8\right)
   \Psi_{q^2}\left(q^{-2}Y_3^{-1}Y_4^{-1}Y_8^{-2}Y_9^{-2}\right)^{-1}
   \Psi_q\left(q^{-1}Y_3^{-1}Y_4^{-1}Y_8^{-1}Y_9^{-1}\right)^{-1}   
   \\
& \quad \cdot \Psi_{q^2}\left(q^{2}Y_2^{-1}Y_3^{-1}\right)
   \Psi_q\left(q^{2}Y_2Y_3Y_7Y_8Y_9\right)
   \Psi_{q^2}\left(Y_3^{-1}\right)P_{24}.
 \end{split}
\end{equation}
The underlines here are inserted for later use in Appendix~\ref{app:pK}, where we prove the following.

\begin{proposition}\label{prop:K24}
The operator $K_{\ve_2,\ve_4}$ is independent of the choice of 
$\ve_2, \ve_4 \in \{1,-1\}$.
\end{proposition}

In Appendix~\ref{app2}, we present the explicit formulas of 
\eqref{sol241}--\eqref{sol244} in the image of $\phi$, 
where the quantum $Y$-variables are expressed in terms of the canonical variables.

\subsection{Type $\rho_{13}$}\label{ad13}

When we take 
$\ve = (\varepsilon_1,-1,\varepsilon_3,-1,-1,1,-\varepsilon_3,-1,-\varepsilon_1,1)$ 
as in~\eqref{tau13}, the map $\eta^K_{1234}$ making the diagram
\eqref{Kcom} commute is given by
\begin{align}\label{tau13uw}
\eta^K_{1234}:
\begin{cases}
u_1 \mapsto 
  \dfrac{1}{2}\bigl(a_1-a_3+b_1-b_3+c_1-c_3+d_1-d_3\bigr) + u_3,
\\[2pt]
u_2 \mapsto u_2,
\\[2pt]
u_3 \mapsto 
  \dfrac{1}{2}\bigl(-a_1+a_3-b_1+b_3-c_1+c_3-d_1+d_3\bigr) + u_1,
\\[2pt]
u_4 \mapsto 
  a_1-a_3+b_1-b_3+c_1-c_3+d_1-d_3 - 2u_1 + 2u_3 + u_4,
\\[4pt]
w_1 \mapsto 
  \dfrac{1}{2}\bigl(-a_1+a_3+b_1-b_3-c_1+c_3-2c_4-d_1+d_3\bigr)
  + w_3 - w_4,
\\[2pt]
w_2 \mapsto b_1-b_3-d_1+d_3 + w_2,
\\[2pt]
w_3 \mapsto 
  \dfrac{1}{2}\bigl(a_1-a_3-b_1+b_3+c_1-c_3+2c_4+d_1-d_3\bigr)
  + w_1 + w_4,
\\[2pt]
w_4 \mapsto -b_1 + b_3 + d_1 - d_3 + w_4.
\end{cases}
\end{align}
This map is independent of $\varepsilon_1$ and $\varepsilon_3$; as in
the case of type~$\rho_{24}$, it defines an automorphism of
$\mathcal{W}(C_2)$ and hence acts naturally on 
$\mathrm{Frac}\,\mathcal{W}(C_2)$.
As with the type $\rho_{24}$,  we have

\begin{lemma}
The operator $P_{13}:= P^K_{1234|\ve}$ which realizes the automorphism $\eta^K_{1234}$ in \eqref{tau13uw} is given by
\begin{align}
&P_{13} = \exp\!\left(\frac{1}{2\hbar} X_{13}\right) \rho_{13},
\\
&X_{13} =
  (u_1 - u_3)\bigl(a_1-a_3 + c_1-c_3 + 2c_4 + 2w_4\bigr)
  \nonumber\\
&\qquad\quad
  + (w_1 - w_3)\bigl(a_1-a_3 + b_1-b_3 + c_1-c_3 + d_1-d_3\bigr)
  \nonumber\\
&\qquad\quad
  + (u_2 - u_4)\bigl(b_1-b_3 - d_1 + d_3\bigr).
\end{align}
\end{lemma}

For each choice of $\ve$ in~\eqref{tau13}, the operator $K_{1234|\ve}$ in~\eqref{ad-K} 
takes the following explicit form, 
where every $Y_i$ is understood as its image under~$\phi$ in \eqref{YwC2}.  
For brevity, we denote $K_{1234|\ve}$ simply by $K_{\ve_1,\ve_3}$.
\begin{equation}
\label{sol131}
\begin{split}
\ve=&(1, -1, 1, -1, -1, 1, -1, -1, -1, 1) ~~(\text{i.e. }\ve_1=\ve_3= 1);
\\
K_{++}&=\Psi _q\left(Y_8\right)
\Psi_{q^2}\left(Y_3^{-1}\right)^{-1}
   \Psi _q\left(Y_9\right)
   \Psi_{q^2}\left(q^{-2}Y_2^{-1}Y_3^{-1}Y_8^{-2}\right)^{-1}
   \\
&  \quad \cdot \Psi_q\left(Y_2^{-1}Y_3^{-1}Y_7^{-1}Y_8^{-2}\right)^{-1}
   \Psi_{q^2}\left(q^{-2}Y_3Y_4\right)
    \Psi _q\left(Y_9\right)^{-1}
   \\
&  \quad \cdot 
   \Psi_{q^2}\left(q^{2}Y_2^{-1}Y_3^{-1}Y_7^{-2}Y_8^{-2}\right)^{-1}
   \Psi_q\left(Y_8\right)^{-1}
   \Psi_{q^2}\left(q^{-4}Y_2Y_3Y_4Y_8^{2}Y_9^{2}\right) P_{13}.
   \end{split}
\end{equation}
\begin{equation}
\label{sol132}
\begin{split}
\ve=&(1, -1, -1, -1, -1, 1, 1, -1, -1, 1)~~(\text{i.e. }\ve_1=-\ve_3= 1);
\\
K_{+-}&=\Psi _q\left(Y_8\right)
\Psi_{q^2}\left(Y_3^{-1}\right)^{-1}
   \Psi _q\left(Y_9^{-1}\right)^{-1}
   \Psi_{q^2}\left(q^{-2}Y_2^{-1}Y_3^{-1}Y_8^{-2}\right)^{-1}
    \\
&  \quad \cdot \Psi_q\left(Y_2^{-1}Y_3^{-1}Y_7^{-1}Y_8^{-2}\right)^{-1}
   \Psi_{q^2}\left(q^{2}Y_3Y_4Y_9^{2}\right)
    \Psi _q\left(Y_9^{-1}\right)
   \\
&  \quad \cdot 
   \Psi_{q^2}\left(q^{2}Y_2^{-1}Y_3^{-1}Y_7^{-2}Y_8^{-2}\right)^{-1}
   \Psi_q\left(Y_8\right)^{-1}
   \Psi_{q^2}\left(q^{-4}Y_2Y_3Y_4Y_8^{2}Y_9^{2}\right)P_{13}.
   \end{split}
\end{equation}
\begin{equation}
\label{sol133}
\begin{split}
\ve=&(-1, -1, 1, -1, -1, 1, -1, -1, 1, 1) ~~(\text{i.e. }-\ve_1=\ve_3= 1);
\\
K_{-+}&=\Psi _q\left(Y_8^{-1}\right)^{-1}
\Psi_{q^2}\left(Y_3^{-1}\right)^{-1}
   \Psi_q\left(q^{-1}Y_8Y_9\right)
   \Psi_{q^2}\left(q^{2}Y_2^{-1}Y_3^{-1}\right)^{-1}
   \\
&  \quad \cdot \Psi_q\left(qY_2^{-1}Y_3^{-1}Y_7^{-1}Y_8^{-1}\right)^{-1}
   \Psi_{q^2}\left(q^{-2}Y_3Y_4\right)
   \Psi_q\left(q^{-1}Y_8Y_9\right)^{-1}   
   \\
&  \quad \cdot 
   \Psi_{q^2}\left(q^{2}Y_2^{-1}Y_3^{-1}Y_7^{-2}Y_8^{-2}\right)^{-1}
   \Psi_q\left(Y_8^{-1}\right)
   \Psi_{q^2}\left(q^{-4}Y_2Y_3Y_4Y_8^{2}Y_9^{2}\right)P_{13}.  
    \end{split}
\end{equation}
\begin{equation}
\label{sol134}
\begin{split}
\ve=&(-1, -1, -1, -1, -1, 1, 1, -1, 1, 1) ~~(\text{i.e. }\ve_1=\ve_3=-1);
\\
K_{--}&=\Psi _q\left(Y_8^{-1}\right)^{-1}
\Psi_{q^2}\left(Y_3^{-1}\right)^{-1}
   \Psi_q\left(q^{-1}Y_8^{-1}Y_9^{-1}\right)^{-1}
   \Psi_{q^2}\left(q^{2}Y_2^{-1}Y_3^{-1}\right)^{-1}
   \\
&  \quad \cdot \Psi_q\left(qY_2^{-1}Y_3^{-1}Y_7^{-1}Y_8^{-1}\right)^{-1}
   \Psi_{q^2}\left(q^{-2}Y_3Y_4Y_8^{2}Y_9^{2}\right)
   \Psi_q\left(q^{-1}Y_8^{-1}Y_9^{-1}\right) 
   \\
&  \quad \cdot 
   \Psi_{q^2}\left(q^{2}Y_2^{-1}Y_3^{-1}Y_7^{-2}Y_8^{-2}\right)^{-1}
   \Psi_q\left(Y_8^{-1}\right)
   \Psi_{q^2}\left(q^{-4}Y_2Y_3Y_4Y_8^{2}Y_9^{2}\right)P_{13}.    
   \end{split}
\end{equation}
See Appendix \ref{app2} for the explicit formulas of 
\eqref{sol131}--\eqref{sol134} in terms of the canonical variables.

In the same manner as for type~$\rho_{24}$, we can show that the
operator $K_{\ve_1,\ve_3}$ is independent of the choice of
$\ve_1,\ve_3 \in \{1,-1\}$.

\subsection{Well-definedness of the $K$-operator}
\label{sec:wdK}

It is nontrivial that the dilogarithm part of the $K$-operator is
well defined as a formal series in the quantum $Y$-variables, since
the sign sequence \eqref{tau24} for \eqref{ad-K} differs from the
tropical sign sequence for $\qK_{1234}$ (see also
Remark~\ref{rem:trop}).  In this subsection we focus on the type
$\rho_{24}$ case and show that it is indeed well defined both as a
formal Laurent series in the quantum variables $Y_i$ and as a formal
Laurent series in $\mathcal{L}(C_2)$ in the generators of $q$-Weyl algebras
$\e^{u_i}$ and $\e^{w_i}$.  By Proposition~\ref{prop:K24},
it suffices to consider the case $K_{++}$ in \eqref{sol241}.  We write
$K_{++}^{\Psi}$ for the product of dilogarithm functions appearing in
$K_{++}$.

First we consider the expansion of $K_{++}^{\Psi}$ in the quantum $Y$-variables.
Expand the $i$th dilogarithm (from the left) in \eqref{sol241} into a power series in its argument by means of \eqref{dilog-sum}
introducing the summation index $n_i$.  
By further using the $q$-commutativity relation \eqref{q-Y}, 
it can be brought to the form
\begin{equation}\label{KA}
K_{++}^{\Psi} = \sum_{{\bf n}} A({\bf n}) Y_2^{p_1}Y_3^{p_2}Y_4^{p_3}Y_7^{p_4}Y_8^{p_5}Y_9^{p_6},
\end{equation}
where the sum is taken over ${\bf n}=(n_1,\ldots, n_{10}) \in (\Z_{\ge 0})^{10}$
and $A({\bf n})$ is a rational function of $q$.
The powers $p_1,\ldots, p_6$ are given by 
\begin{equation}\label{pnK}
\begin{split}
&p_1= n_4+n_8+n_9,\quad
p_2= n_2-n_3-2n_6-n_7+n_9+n_{10},\quad 
p_3 =-n_6-n_7,
\\
&p_4 = n_5+n_9, \quad 
p_5 = -n_1-n_3+n_5-2 n_6-n_7+n_9,\quad
p_6 = -n_3-2 n_6-n_7+n_9.
\end{split}
\end{equation}
The series \eqref{KA} is well defined since the equations \eqref{pnK} constrain 
${\bf n}$ to finitely many (possibly empty) possibilities for any $(p_1,\ldots, p_6) \in \Z^6$.

Next we consider the expansion of $K_{++}^{\Psi}$ in the variables $\e^{u_i}$ and $\e^{w_i}$.  
Recall that we have set $\gamma =(1,2,1,2)$.
Let $\mathcal{L}(C_2):= \mathcal{L}_\gamma$ be a set of formal Laurent series \eqref{fLs} in $\e^{u_i}$ and $\e^{w_i}$ 
obeying the commutation relation \eqref{uw-def}. 
Substitution of \eqref{YwC2}  into \eqref{KA} leads to 
\begin{align}\label{Kuw}
K_{++}^{\Psi} = \sum_{{\bf n}}\tilde{A}({\bf n}) \,
\e^{\alpha_1u_1+\cdots + \alpha_4 u_4+\alpha_5 w_1+\cdots \alpha_9 w_4},
\end{align}
where the sum is taken over ${\bf n}=(n_1,\ldots, n_{10}) \in (\Z_{\ge 0})^{10}$
and $\tilde{A}({\bf n})$ is a function of $q$ and the parameters $(a_i,b_i,c_i,d_i,e_i)$ $(i=1,\ldots, 4)$.
The powers $\alpha_1,\ldots, \alpha_8$ are given by
\begin{equation}\label{alnK}
\begin{split}
\alpha_1 &= n_1+n_3+n_5+2 n_6+n_7+n_9,
\\
\alpha_2 &=-n_2+n_3+2 n_4+2 n_6+n_7+2 n_8+n_9-n_{10},
\\
\alpha_3 &=n_1-n_3-n_5-2n_6-n_7+n_9,
\\
\alpha_4 &= -n_2+n_3-n_7-n_9-n_{10},
\\
\alpha_5&=n_1+n_3-n_5+2 n_6+n_7-n_9,
\\
\alpha_6&= -n_1-n_2+n_5-n_{10},
\\
\alpha_7 &=n_1+2n_2-n_3-n_5-2 n_6-n_7+n_9+2 n_{10},
\\
\alpha_8&= -n_2+n_3+2 n_6+n_7-n_9-n_{10}.
   \end{split}
   \end{equation}
The series \eqref{Kuw} is well defined since \eqref{alnK} constrain 
${\bf n}$ to finitely many (possibly empty) possibilities. Thus it follows that  $K_{++}^\Psi \in \mathcal{L}(C_2)$.

\section{$3$D Reflection equation from $q$-Weyl algebra}\label{sec:main}

In this section we present our main result: a solution to the 
three-dimensional reflection equation (3DRE, for short) expressed 
in terms of the $R$-operators \eqref{R-+-+}, \eqref{Rb--++} and the 
$K$-operators \eqref{sol241}--\eqref{sol244} of type $\rho_{24}$.
Following the strategy developed for the tetrahedron equation 
in \cite{IKT1,IKT2,IKSTY}, this will be achieved in several steps.

In \S \ref{ss:tau}, we construct solutions 
in terms of the monomial transformations 
$\tau^{(\pm)}_{ijk}, \overline{\tau}^{(\pm)}_{ijk}, \tau^K_{ijkl}$ 
of the quantum $Y$-variables 
(Proposition~\ref{REtau}).  This corresponds to condition (ii) 
described after \eqref{adpe}. 
In \S \ref{ss:eta}, these transformations are translated into solutions 
$\eta^{(\pm)}_{ijk}, \overline{\eta}^{(\pm)}_{ijk}, \eta^K_{ijkl}$ 
in terms of affine transformations of the canonical 
variables consisting of the $q$-Weyl generators (Proposition~\ref{REeta}).
In \S \ref{ss:P}, we obtain solutions in terms of operators 
$P^{(\pm)}_{ijk}, \overline{P}^{(\pm)}_{ijk}, P^K_{ijkl}$ 
whose adjoint action realizes those obtained in \S \ref{ss:eta}
 (Proposition~\ref{prop:RE-P}).  
In \S \ref{ss:mr}, building on 
these steps, we present the main result of the paper, 
which is the full solution of the 3DRE in terms of the 
$R$-operators and $K$-operators (Theorem~\ref{th:main}).

In the intermediate steps,  we require that 
the signs appearing as the superscripts of the operators 
$\tau$, $\eta$ and $P$ are {\em good} so that the 3DRE holds
and a certain homogeneity is satisfied.
It turns out that these conditions gradually constrain the signs and
eventually lead to a single possibility, which is rather remarkable.
The parameters introduced together with the canonical variables are
likewise required to satisfy certain relations.

\subsection{3DRE for $\tau$: Monomial transformations of quantum $Y$-variables}
\label{ss:tau}

Recall the monomial transformations of quantum $Y$-variables, $\tau_{123 | --++}$ \eqref{tau--++}, 
$\tau_{123 | -+-+}$ \eqref{tau-+-+}, $\overline{\tau}_{123 | --++}$ \eqref{taub--++}, 
$\overline{\tau}_{123 | -+-+}$ \eqref{taub-+-+}, and  
$\tau^K_{1234|\ve}$ \eqref{tauK24}. These are naturally extended to those for $\mathcal{Y}(B(C_3))$, such as $\tau_{ijk | --++}$, 
$\tau^K_{ijkl|\ve}$ and so on. 
For simplicity, we write 
\begin{align}
\label{tau-s}
\begin{split}
&\tau_{ijk}^{(-)} := \tau_{ijk|--++}, \quad \tau_{ijk}^{(+)} := \tau_{ijk|-+-+},\\ 
&\overline{\tau}_{ijk}^{(-)}:= \overline{\tau}_{ijk|--++}, \quad \overline{\tau}_{ijk}^{(+)}:= \overline{\tau}_{ijk|-+-+},
\\
&\tau^K_{ijkl}:=\tau^K_{ijkl|\ve}~; ~\text{$\ve$ is of type $\rho_{24}$ \eqref{tau24}}.
\end{split}
\end{align}
Corresponding to \eqref{REiso},
these are expected to satisfy the 3DRE of the form
\begin{align}
\label{RE-tau}
\tau^{(\delta_1)}_{124} \,\tau^K_{1356}\, \overline{\tau}^{(\delta_2)}_{178}\, 
\tau^{(\delta_3)}_{258} \,\tau^K_{2379}\, \tau^K_{4689}\, \overline{\tau}^{(\delta_4)}_{457}
=
\tau^{(\delta'_1)}_{457} \,\tau^K_{4689} \,\tau^K_{2379} \,\overline{\tau}^{(\delta'_2)}_{258} 
\,\tau^{(\delta'_3)}_{178}\, \tau^K_{1356} \,\overline{\tau}^{(\delta'_4)}_{124},
\end{align}
for some sign sequence $\delta= (\delta_1,\delta_2,\delta_3,\delta_4,\delta'_1,\delta'_2,\delta'_3,\delta'_4) \in \{+,-\}^8$.

We say a sign sequence $\delta$ 
is {\it good}, if \eqref{RE-tau} is satisfied with a homogeneity such that $\delta_1 = \delta_3 = \delta_1' = \delta_3'$ 
and $\delta_2 = \delta_4 = \delta_2' = \delta_4'$ holds.
By definition, a good sign sequence has the form
$\delta=(\delta_1,\delta_2,\delta_1,\delta_2,\delta_1,\delta_2,
\delta_1,\delta_2)$.  Henceforth, we will specify such a sequence simply
by $(\delta_1,\delta_2)$.

The following result is obtained by direct computation.

\begin{proposition}
\label{REtau}
The 3DRE \eqref{RE-tau} admits exactly three good sign sequences $\delta$ 
corresponding to $(\delta_1, \delta_2) = (+,+), (+,-)$ and $(-,-)$, 
which give rise to the following:
\begin{align}
\label{REtau-1}
&\tau^{(-)}_{124} \,\tau^K_{1356}\, \overline{\tau}^{(-)}_{178}\, 
\tau^{(-)}_{258} \,\tau^K_{2379}\, \tau^K_{4689}\, \overline{\tau}^{(-)}_{457}
=
\tau^{(-)}_{457} \,\tau^K_{4689} \,\tau^K_{2379} \,\overline{\tau}^{(-)}_{258} \,
\tau^{(-)}_{178}\, \tau^K_{1356} \,\overline{\tau}^{(-)}_{124}, 
\\
\label{REtau-2}
&\tau^{(+)}_{124} \,\tau^K_{1356}\, \overline{\tau}^{(-)}_{178}\, \tau^{(+)}_{258} \,
\tau^K_{2379}\, \tau^K_{4689}\, \overline{\tau}^{(-)}_{457}
=
\tau^{(+)}_{457} \,\tau^K_{4689} \,\tau^K_{2379} \,\overline{\tau}^{(-)}_{258} \,
\tau^{(+)}_{178}\, \tau^K_{1356} \,\overline{\tau}^{(-)}_{124}, 
\\
\label{REtau-3}
&\tau^{(+)}_{124} \,\tau^K_{1356}\, \overline{\tau}^{(+)}_{178}\, \tau^{(+)}_{258} \,
\tau^K_{2379}\, \tau^K_{4689}\, \overline{\tau}^{(+)}_{457}
=
\tau^{(+)}_{457} \,\tau^K_{4689} \,\tau^K_{2379} \,\overline{\tau}^{(+)}_{258} \,
\tau^{(+)}_{178}\, \tau^K_{1356} \,\overline{\tau}^{(+)}_{124}.  
\end{align}
\end{proposition}

\begin{remark}\label{rem:trop}
The tropical sign sequence for the 3DRE leads to the 3DRE for monomial transformations as follows.
\begin{align*}
&\tau_{124|++++} \,\tau^K_{1356|++++++--++} \, \overline{\tau}_{178|++++} \,\tau_{258|+-++} 
\\
& \qquad \circ \tau^K_{2379|++++++--++} \,\tau^K_{4689|--++++--++}\,\overline{\tau}_{457|--++}
\\
&=
\tau_{457|++++} \,\tau^K_{4689|++++++--++} \,\tau^K_{2379|++++++--++} \,\overline{\tau}_{258|--++} 
\\
& \qquad \circ \tau_{178|++-+}\, \tau^K_{1356|+++++---++} \,\overline{\tau}_{124|-+++}.
\end{align*}
One sees that none of $\tau^K_{ijkl|\ve}$, $\tau_{ijk|\ve}$ 
and $\overline{\tau}_{ijk|\ve}$ has homogeneity in sign sequences.  
\end{remark}

\subsection{3DRE for $\eta$: Affine transformations of canonical variables}
\label{ss:eta}

Let $\gamma=(\gamma_i)_{i=1,\ldots,9}$ be given by 
$\gamma=(1,1,2,1,1,2,1,1,2)$.
Let $\mathcal{W}(C_3):=\mathcal{W}_\gamma$ \eqref{qW-gamma} be the $q$-Weyl algebra 
generated by $\e^{\pm u_i}$ and $\e^{\pm w_i}$, with the canonical variables 
$(u_i,w_i)$ $(i=1,\ldots,9)$ satisfying
\begin{align}
[u_i,w_j]=\hbar\,\gamma_i\,\delta_{ij},
\qquad
[u_i,u_j]=[w_i,w_j]=0.
\label{uw-defc3}
\end{align}
Write $\mathrm{Frac}\,\mathcal{W}(C_3)$ for the skewfield of $\mathcal{W}(C_3)$.

Recall the monomial transformations of $q$-Weyl variables
$\eta_{123}^{(\mp)}$ \eqref{uw-1}--\eqref{uw-2},
$\overline{\eta}_{123}^{(\mp)}$ \eqref{uwb-1}--\eqref{uwb-2} on $\mathcal{W}(A_2)$,
and $\eta^K_{1234}$ \eqref{tau24uw} on $\mathcal{W}(C_2)$.
They extend to transformations $\eta_{ijk}^{(\mp)}$, 
$\overline{\eta}_{ijk}^{(\mp)}$ and $\eta^K_{ijkl}$ acting on 
$\mathrm{Frac}\,\mathcal{W}(C_3)$.

For these operators, a good sign sequence 
$\delta=(\delta_1,\delta_2,\delta_3,\delta_4,\delta'_1,\delta'_2,
\delta'_3,\delta'_4)\in\{+,-\}^8$ is
defined in the same way as for the 3DRE \eqref{RE-tau}.  Namely, 
$\delta$ is said to be good if it makes 
\begin{align}
&\eta^{(\delta_1)}_{124} \,\eta^K_{1356}\, \overline{\eta}^{(\delta_2)}_{178}\, 
\eta^{(\delta_3)}_{258} \,\eta^K_{2379}\, \eta^K_{4689}\, \overline{\eta}^{(\delta_4)}_{457}
=
\eta^{(\delta'_1)}_{457} \,\eta^K_{4689} \,\eta^K_{2379} \,
\overline{\eta}^{(\delta'_2)}_{258} \,\eta^{(\delta'_3)}_{178}\, \eta^K_{1356} \,
\overline{\eta}^{(\delta'_4)}_{124}, 
\label{REetag}
 \end{align}
hold, and satisfies the conditions
$\delta_1=\delta_3=\delta'_1=\delta'_3$ and
$\delta_2=\delta_4=\delta'_2=\delta'_4$.
A good sign sequence has the form
$\delta=(\delta_1,\delta_2,\delta_1,\delta_2,\delta_1,\delta_2, \delta_1,\delta_2)$, 
and is specified simply
by $(\delta_1,\delta_2)$.

Using \eqref{econ-a}, \eqref{econ} and \eqref{ccon}, we obtain the conditions
for the parameters $\mathscr{P}_i = (a_i,b_i,c_i,d_i,e_i)$:
\begin{align}
\label{recon1}
&a_i + b_i + c_i + d_i + e_i = 0 
\quad \text{for } i = 1,2,4,5,7,8,
\\[1mm]
\label{recon2}
&\begin{cases}
b_j + 2c_j + d_j + e_j = c_i + c_k,
\\
b_l + 2c_l + d_l + e_l = -c_i + c_k
\end{cases}
\;\text{for } (i,j,k,l) = (4,6,8,9), (2,3,7,9), (1,3,5,6).
\end{align}
We note that \eqref{recon2} implies the necessary conditions
\begin{align}\label{Ksub}
c_1 - c_2 + c_4 = 0, 
\quad
c_2 - c_5 + c_8 = 0, 
\quad
c_1 - c_2 + c_5 - c_7 = 0.
\end{align}

By a direct calculation using \eqref{uw-1}, \eqref{uw-2}, \eqref{uwb-1}, 
\eqref{uwb-2} and \eqref{tau24uw}, we obtain the following.

\begin{proposition}
\label{REeta}
Assume \eqref{recon1}, \eqref{Ksub}  and
\begin{align}
a_1 - a_2 + a_4 = 0, 
\quad 
a_2 - a_5 + a_8 = 0, 
\quad 
a_1 - a_2 + a_5 - a_7 = 0.
\end{align}
Then the 3DRE \eqref{REetag} admits exactly three good sign sequences $\delta$ 
corresponding to $(\delta_1, \delta_2) = (+,+), (+,-)$ and $(-,-)$, 
which give rise to the following:
\begin{align}
\label{REeta-1}
&\eta^{(-)}_{124} \,\eta^K_{1356}\, \overline{\eta}^{(-)}_{178}\, \eta^{(-)}_{258} \,
\eta^K_{2379}\, \eta^K_{4689}\, \overline{\eta}^{(-)}_{457}
=
\eta^{(-)}_{457} \,\eta^K_{4689} \,\eta^K_{2379} \,\overline{\eta}^{(-)}_{258} \,
\eta^{(-)}_{178}\, \eta^K_{1356} \,\overline{\eta}^{(-)}_{124},  
\\
\label{REeta-2}
&\eta^{(+)}_{124} \,\eta^K_{1356}\, \overline{\eta}^{(-)}_{178}\, \eta^{(+)}_{258} \,
\eta^K_{2379}\, \eta^K_{4689}\, \overline{\eta}^{(-)}_{457}
=
\eta^{(+)}_{457} \,\eta^K_{4689} \,\eta^K_{2379} \,\overline{\eta}^{(-)}_{258} \,
\eta^{(+)}_{178}\, \eta^K_{1356} \,\overline{\eta}^{(-)}_{124},  
\\
\label{REeta-3}
&\eta^{(+)}_{124} \,\eta^K_{1356}\, \overline{\eta}^{(+)}_{178}\, \eta^{(+)}_{258} \,
\eta^K_{2379}\, \eta^K_{4689}\, \overline{\eta}^{(+)}_{457}
=
\eta^{(+)}_{457} \,\eta^K_{4689} \,\eta^K_{2379} \,\overline{\eta}^{(+)}_{258} \,
\eta^{(+)}_{178}\, \eta^K_{1356} \,\overline{\eta}^{(+)}_{124}.  
\end{align}
In particular, each of these is compatible with the 3DRE for the monomial
transformations of the $Y$-variables \eqref{REtau-1}--\eqref{REtau-3}, thanks
to the commutative diagrams \eqref{Rcom1}, \eqref{Rcom2} and \eqref{Kcom}. 
\end{proposition}
Note that the claim holds under the condition \eqref{Ksub},  which is weaker than \eqref{recon2}.

\begin{remark}
In addition to the three good sign sequences $\delta$ appearing in
Propositions~\ref{REtau} and \ref{REeta}, there exist ten further sign
sequences that do not satisfy the homogeneity condition but still make
the 3DRE hold.
Among these, two are `symmetric' in the sense that 
$(\delta_1,\delta_2,\delta_3,\delta_4) =(\delta'_1,\delta'_2,\delta'_3,\delta'_4)$:
\[
(-,+,-,-,-,+,-,-), 
\qquad 
(+,+,-,+,+,+,-,+).
\]
The remaining eight are not symmetric:
\begin{align*}
&(-, -, -, -, +, -, -, +), ~(-, -, +, -, +, -, +, +),
~(-, +, -, -, +, +, -, +),
\\ &(-, +, +, -, +, +, +, +), ~(+, -, -, -, -, -, +, -), 
~(+, -, -, +, -, +, +, -),
\\ &(+, -, +, +, +, +, +, -), ~(+, +, -, +, -, +, -, -).
\end{align*}
\end{remark}

\subsection{3DRE for operators $P$}\label{ss:P}

We extend the operators $P_{123}^{(\pm)}$ in \eqref{P--++}, \eqref{P-+-+},
$\overline{P}_{123}^{(\pm)}$ in \eqref{Pb--++}, \eqref{Pb-+-+}
and $P^K_{1234|\ve}$ in \eqref{p24}
to $P_{ijk}^{(\pm)}$,  $\overline{P}_{ijk}^{(\pm)}$ and $P^K_{1234|\ve}$ 
acting on $\mathrm{Frac}\,\mathcal{W}(C_3)$ naturally.
For simplicity we write $P^K_{ijkl}$ for $P^K_{ijkl|\ve}$.

Let $\prec$ be a partial order on a set $J := \{1,2,\ldots,9\}$ given by
$$
  1 \prec 2,4 \prec 5,7,8,3,6,9.
$$
The symmetric subgroup
$S(C_3) = \mathfrak{S}_2 \times \mathfrak{S}_3 \times \mathfrak{S}_3$ of
$\mathfrak{S}_9$ acts on $J$ in such a way that $\mathfrak{S}_2$ acts on
$\{2,4\}$ and $\mathfrak{S}_3 \times \mathfrak{S}_3$ acts on
$\{5,7,8\} \times \{3,6,9\}$.  Note that $S(C_3)$ preserves the partial
order.   
For the canonical variables \eqref{uw-defc3}, we define the group
$N(C_3)$ generated by
\begin{align}\label{NdC3}
\e^{\pm \tfrac{1}{\hbar}u_iw_j} \ (i \succ j),\quad 
\e^{ \tfrac{a}{\hbar}u_i}, \ \e^{ \tfrac{a}{\hbar}w_i} \ (a \in \C),\quad
b \in \C^\times; ~~i, j\in \{1,\ldots, 9\},
\end{align}
in the same way as $N(A_3)$ in \S \ref{sec:Rop}. 
The group $S(C_3)$ acts on $N(C_3)$ via the adjoint action, permuting
the indices of the canonical variables.  Further, let
$\mathcal{L}(C_3):=\mathcal{L}_\gamma$ denote the set of formal Laurent
series \eqref{fLs}.  It is straightforward to see that
$P_{ijk}^{(\pm)}$, $\overline{P}_{ijk}^{(\pm)}$, and $P^K_{ijkl}$ all
belong to $N(C_3)\rtimes S(C_3)$ by inspecting their explicit formulas.

Given a sign sequence $\delta=(\delta_1,\delta_2,\delta_3,\delta_4,\delta'_1,\delta'_2,
\delta'_3,\delta'_4)\in\{+,-\}^8$, we consider the 3DRE:
\begin{align}
&{P}_{124}^{(\delta_1)} P^K_{1356} \overline{P}_{178}^{(\delta_2)} {P}_{258}^{(\delta_3)} 
P^K_{2379} P^K_{4689} \overline{P}_{457}^{(\delta_4)}=
{P}_{457}^{(\delta'_1)} P^K_{4689} P^K_{2379} \overline{P}_{258}^{(\delta'_2)}
 {P}_{178}^{(\delta'_3)} P^K_{1356} \overline{P}_{124}^{(\delta'_4)}.
 \label{REPg}
 \end{align}
 Again, $\delta$ is defined to be good if it has the form
 $\delta=(\delta_1,\delta_2,\delta_1,\delta_2,\delta_1,\delta_2, \delta_1,\delta_2)$,
 and makes \eqref{REPg} hold.

The following proposition is proved by a direct computation using the BCH
formula.  It shows that, among the three good sign sequences
appearing in Propositions~\ref{REtau} and \ref{REeta}, only the case corresponding to 
$(\delta_1,\delta_2)=(+,-)$ survives.  Moreover, it becomes necessary to
impose the condition on the parameters that 
\begin{align}\label{conac}
a_i=c_i \quad \text{for } i=1,2,4,5,7,8.
\end{align}
In this situation, we may use the simplifying feature described in
Remark~\ref{re:ac} and set
\begin{align}\label{srp1}
P_{ijk}=P^{(+)}_{ijk}=\overline{P}^{(-)}_{ijk}.
\end{align}

\begin{proposition}\label{prop:RE-P}
Assume \eqref{recon1}, \eqref{recon2} and \eqref{conac}.
Then the 3DRE \eqref{REPg} admits a unique good sign sequence corresponding to 
$(\delta_1,\delta_2)=(+,-)$.
In terms of the operator $P_{ijk}$ in \eqref{srp1} and $P^K_{ijkl}$, it is expressed as the 
following equality in $N(C_3)\rtimes S(C_3)$:
\begin{align}\label{RE-P}
P_{124}\,P^K_{1356}\,P_{178}\,P_{258}\,P^K_{2379}\,P^K_{4689}\,P_{457}
=
P_{457}\,P^K_{4689}\,P^K_{2379}\,P_{258}\,P_{178}\,P^K_{1356}\,P_{124}.
\end{align}
\end{proposition}

\subsection{Main result: Full solution to 3DRE by $R$ and $K$}\label{ss:mr}

We extend the operators $R_{123}^{(\pm)}$ in \eqref{R--++},
\eqref{R-+-+} to $R_{ijk}^{(\pm)}$, the operators
$\overline{R}_{123}^{(\pm)}$ in \eqref{Rb--++}, \eqref{Rb-+-+} to
$\overline{R}_{ijk}^{(\pm)}$, and the operators $K_{1234|\ve}$ in
\eqref{p24}--\eqref{sol244} to $K_{ijkl|\ve}$, acting naturally on
$\mathrm{Frac}\,\mathcal{W}(C_3)$.  For simplicity, we write $K_{ijkl}$
for $K_{ijkl|\ve}$ (see Proposition~\ref{prop:K24}).

We focus on the unique good sign sequence
$\delta=(+,-,+,-,+,-,+,-)$ in Proposition~\ref{prop:RE-P}, and
examine the validity of the corresponding 3DRE for the $R$- and
$K$-operators:
\begin{align}\label{RErk}
{R}_{124}^{(+)} K_{1356} \overline{R}_{178}^{(-)} {R}_{258}^{(+)}
K_{2379} K_{4689} \overline{R}_{457}^{(-)}
=
{R}_{457}^{(+)} K_{4689} K_{2379} \overline{R}_{258}^{(-)}
R_{178}^{(+)} K_{1356} \overline{R}_{124}^{(-)}.
\end{align}

We continue to assume the condition $a_i=c_i$ for
$i=1,2,4,5,7,8$, as required in Proposition~\ref{prop:RE-P}.
Under this assumption, all $R$-operators in \eqref{RErk} satisfy the
simplifying relation in Remark~\ref{re:ac}. Accordingly, we set
\begin{align}\label{srp2}
R_{ijk}=R^{(+)}_{ijk}=\overline{R}^{(-)}_{ijk}.
\end{align}

From Proposition~\ref{wdR} and \S\ref{sec:wdK}, each of the operators
$R_{ijk}$ and $K_{ijkl}$ decomposes into a dilogarithm part in
$\mathcal{L}(C_3)$ and a monomial part in $N(C_3)\rtimes S(C_3)$.
We write these decompositions as
\begin{align}\label{RK-decomp}
R_{ijk}=R_{ijk}^{\Psi}  P_{ijk},
\qquad
K_{ijkl}=K_{ijkl}^{\Psi}  P_{ijkl}^K.
\end{align}

The main result of this paper is the following theorem, which answers
the quest for \eqref{RErk} affirmatively.

\begin{theorem}\label{th:main}
Assume \eqref{recon1}, \eqref{recon2} and \eqref{conac}.
Then the 3DRE \eqref{RErk} is valid.  In the
notation of \eqref{srp2}, it takes the form
\begin{align}\label{RE-RK}
R_{124} K_{1356} R_{178} R_{258} K_{2379} K_{4689} R_{457}
=
R_{457} K_{4689} K_{2379} R_{258} R_{178} K_{1356} R_{124},
\end{align}
in the sense that each side decomposes into a dilogarithm part in
$\mathcal{L}(C_3)$ and a monomial part in $N(C_3)\rtimes S(C_3)$, and
both components coincide.
\end{theorem}

\begin{proof}
With the decompositions \eqref{RK-decomp} the LHS of the 3DRE has the form
\begin{align*}
&{R}_{124}^{\Psi} P_{124} \cdot K_{1356}^\Psi P^K_{1356} \cdot R_{178}^{\Psi} 
P_{178} \cdot {R}_{258}^{\Psi} {P}_{258}\cdot
K_{2379}^\Psi P_{2379}^K \cdot K_{4689}^\Psi P^K_{4689} \cdot R_{457}^{\Psi} P_{457}.
\end{align*}
Moving the monomial parts to the right, we obtain 
\begin{align}\label{RE-decompR}
{R}_{124}^{\Psi} {K}_{1356}^{\Psi'} R_{178}^{\Psi'}
 {R}_{258}^{\Psi'} {K}_{2379}^{\Psi '} K_{4689}^{\Psi '} R_{457}^{\Psi'} 
\cdot 
{P}_{124} P^K_{1356} P_{178}{P}_{258}P^K_{2379} P^K_{4689} P_{457},
\end{align}
where 
\begin{align*}
&{K}_{1356}^{\Psi'} = \text{Ad}({P}_{124})({K}_{1356}^{\Psi}),
\\
&R_{178}^{\Psi'} = \text{Ad}({P}_{124}P^K_{1356}) (R_{178}^{\Psi}),
\\
&{R}_{258}^{\Psi'} = \text{Ad}({P}_{124} P^K_{1356} P_{178}) ({R}_{258}^{\Psi}),
\\ 
&{K}_{2379}^{\Psi '} = \text{Ad}({P}_{124} P^K_{1356} P_{178} {P}_{258})({K}_{2379}^{\Psi}),
\\
&K_{4689}^{\Psi '} = \text{Ad}({P}_{124} P^K_{1356} P_{178}{P}_{258} P^K_{2379} )(K_{4689}^{\Psi}),
\\
&R_{457}^{\Psi'}= \text{Ad}({P}_{124}P^K_{1356}P_{178}{P}_{258}P^K_{2379} P^K_{4689} )(R_{457}^{\Psi}).  
\end{align*}
Note that all of these terms belong to $\mathcal{L}(C_3)$ by
Lemma~\ref{lem:L}.  We denote by $F^{\Psi}_L$ and $P^{\Psi}_L$ the
dilogarithm and monomial parts of \eqref{RE-decompR}, respectively.
Likewise, we obtain the dilogarithm part $F^{\Psi}_R$ and the
monomial part $P^{\Psi}_R$ of the RHS of \eqref{RE-RK}.  Our goal is
to show that $F^{\Psi}_L=F^{\Psi}_R$ in $\mathcal{L}(C_3)$ and
$P^{\Psi}_L=P^{\Psi}_R$ in $N(C_3)\rtimes S(C_3)$.  The latter already
follows from Proposition~\ref{prop:RE-P}.  In what follows, we prove
the equality $F^{\Psi}_L=F^{\Psi}_R$.

First we consider $F^{\Psi}_L$ and $F^{\Psi}_R$ in terms of quantum
$Y$-variables rather than $q$-Weyl generators. 
Let $\tau_L$ and $\tau_R$ be the LHS and RHS of
\eqref{REtau-2}, respectively.  
The identity \eqref{s-per}  from the general theory, applied to 
the mutation sequence \eqref{RE-SB} for the 3DRE, yields 
\begin{align*}
\mathrm{Ad}(F^\Psi_L) \circ \tau_L
   = \mathrm{Ad}(F^\Psi_R) \circ \tau_R,
\end{align*}
as homomorphisms on $\mathcal{Y}(B(C_3))$.  
From \eqref{REtau-2}, we know $\tau_L = \tau_R$, hence obtain
$\mathrm{Ad}((F^\Psi_R)^{-1}\! \cdot F^\Psi_L)=\mathrm{id}$.

In \S \ref{dilog-p1} and \S \ref{dilog-p3} below we show that all of 
$F^{\Psi}_L$, $F^{\Psi}_R$ and $(F^\Psi_R)^{-1}\! \cdot F^\Psi_L$ are well defined as formal Laurent series in the quantum
$Y$-variables, and that the constant terms of $F^{\Psi}_L$ and 
$F^{\Psi}_R$ are $1$.  This leads to
$F^\Psi_L = F^\Psi_R$ by the same reasoning as in the proof of
\cite[Th.\ 3.5]{KN11} as follows.  We extend the degenerate exchange
matrix $B(C_3)$ (whose rank is $18$) to a nondegenerate one
$\tilde{B}$ of twice the size of $B(C_3)$ (see \cite[Example 2.5]{KN11}).
Then, by the Extension Theorem \cite[Th.\ 4.3]{N11}, the periodicity of
the seed $(B(C_3),Y^{(1)})$ is inherited by the extended seed
$(\tilde{B},\tilde{Y})$.  Thus $(F^\Psi_R)^{-1}\!\cdot F^\Psi_L$
commutes with any element of the quantum torus algebra
$\mathcal{T}(\tilde{B})$.  Since $\tilde{B}$ is nondegenerate, this
implies that $(F^\Psi_R)^{-1}\!\cdot F^\Psi_L$ is a constant depending
only on $q$.  This constant must be $1$, because the constant terms of
$F^\Psi_L$ and $F^\Psi_R$ are both $1$.

Finally, by the argument in \S \ref{dilog-p2} and \S \ref{dilog-p4}  below, we see that 
$F^\Psi_L$, $F^\Psi_R$ and $(F^\Psi_R)^{-1}\! \cdot F^\Psi_L$ are elements of $\mathcal{L}(C_3)$, and hence
we obtain the identity $F^\Psi_L = F^\Psi_R$ in $\mathcal{L}(C_3)$.
\end{proof}

\begin{remark}\label{re:b}
The $K$-operators \eqref{sol131}--\eqref{sol134} of type $\rho_{13}$ are
expected to satisfy the 3DRE associated with the Lie algebra $B_3$, in
conjunction with the $R$-operators for the SB-quiver whose vertices all
have weight two.  In the present paper we do not pursue this direction.
Instead, we restrict ourselves to describing the relation between the
$K$-operators of type $\rho_{13}$ and those for the FG-quiver of
$B_2$-type in \S \ref{sec:lim13}.
\end{remark}

\subsection{Explicit formulas}\label{ss:ef}

The conditions on the parameters \eqref{recon1}, \eqref{recon2} and \eqref{conac}, imposed in Theorem~\ref{th:main},
slightly simplify the $R$- and $K$-operators into more symmetric
forms.  We record here such final expressions together with the
practical computations leading to them.

Consider $R_{124}$, for instance, appearing in the 3DRE
\eqref{RE-RK}.  By \eqref{srp2}, it is obtained from
\eqref{R-+-+}--\eqref{P-+-+} by replacing the indices $1,2,3$ with
$1,2,4$.  Thus we have
\begin{align}
&P_{124} \overset{\eqref{srp1}}{=} P_{124}^{(+)}=
\e^{\tfrac{1}{\hbar}(u_4-u_2)w_1}
\e^{\tfrac{\kappa'_0}{\hbar}(w_4-w_2-w_1)}
\e^{\tfrac{1}{\hbar}(\kappa'_1u_1+\kappa'_2u_2+\kappa'_4u_4)}\rho_{24},
\\
&\kappa'_0=\frac{e_2-e_4}{2},\quad
\kappa'_1=b_4+c_4-b_2-c_2-\kappa'_0,\quad
\kappa'_2=d_4-d_2-a_1-\kappa'_0,\quad
\kappa'_4=c_1-c_2+c_4.
\end{align}
Using the conditions in Theorem~\ref{th:main}, it follows easily that
$\kappa'_1=-\kappa'_2=\frac{-b_2+d_2 +b_4-d_4}{2}$ and $\kappa'_4=0$.  
The same reduction applies to all $P_{ijk}$, and we obtain
\begin{align}\label{pfinal}
P_{ijk}
=
\e^{\tfrac{1}{\hbar}(u_k-u_j)w_i}
\e^{\tfrac{e_j-e_k}{2\hbar}(w_k-w_j-w_i)}
\e^{\tfrac{-b_j+d_j+b_k-d_k}{2\hbar}(u_i-u_j)}\rho_{jk}
\end{align}
for $i,j,k\in\{1,2,4,5,7,8\}$.

Combining this expression with \eqref{R-+-+} or \eqref{rsigma}, the $R$-operator takes the form
\begin{align}
R_{ijk}
&= \Psi_q(\e^{-d_k-c_j-b_i+u_i+u_k+w_i-w_j+w_k})^{-1}
\Psi_q(\e^{d_k+c_j+b_i+e_i-u_k+u_i-w_i+w_j-w_k})
\nonumber \\
& \quad \cdot 
\Psi_q(\e^{-d_k-e_k-c_j-b_i-u_k+u_i+w_i-w_j+w_k})^{-1}
\Psi_q(\e^{d_k+c_j+e_j+b_i+e_i-u_k+2u_j+u_i-w_i+w_j-w_k})P_{ijk}
\label{Rf1}
\\
&= \Psi_q(\e^{-d_k-c_j-b_i+u_i+u_k+w_i-w_j+w_k})^{-1}
\Psi_q(\e^{d_k+c_j+b_i+e_i-u_k+u_i-w_i+w_j-w_k})P_{ijk}
\nonumber \\
&\quad \cdot
\Psi_q(\e^{d_i+e_i+a_j+b_k+u_i-u_k-w_i+w_j-w_k})^{-1}
\Psi_q(\e^{-d_i-a_j-b_k+u_i+u_k-w_i+w_j-w_k}),
\label{Rf2}
\end{align}
where $a_l+b_l+c_l+d_l+e_l=0$ and $a_l=c_l$ for $l=i,j,k$.

Next we consider $K_{1356}$, for instance, appearing in the 3DRE \eqref{RE-RK}.
We have the decomposition into the dilogarithm and monomial parts 
as $K_{1356} = K^\Psi_{1356} P^K_{1356}$, and the latter is given by \eqref{p24} by replacing the 
indices $1,2,3,4$ with $1,3,5,6$.
Thus we have $P^{K}_{1356} = \exp(\frac{X}{2\hbar})\rho_{36}$ with 
\begin{equation}\label{pp24}
\begin{split}
X &=
(-u_3+u_6)(a_1+c_1-c_3+c_6+2w_1)
-(u_1-u_5)(b_3-b_6-d_3+d_6)
\\
&+(-w_3+w_6)\Bigl(\tfrac12(-b_3+b_6-d_3+d_6)+c_1-c_3+c_6\Bigr).
\end{split}
\end{equation}
Using the conditions in Theorem~\ref{th:main},  the coefficient appearing here are 
rewritten as
\begin{align}
a_1+c_1-c_3+c_6 = a_6-a_3, 
\quad
\tfrac12(-b_3+b_6-d_3+d_6)+c_1-c_3+c_6=\tfrac12(e_3-e_6)
\end{align}
by introducing $a_3$ and $a_6$ extending the condition \eqref{recon1} to $i=3,6,9$ as well.
The same rewriting holds for $P^K_{ijkl}$ in general, and we obtain
\begin{equation}\label{pf24}
\begin{split}
P^K_{ijkl} &= \e^{\frac{1}{2\hbar}\{(u_j-u_l)(a_j-a_l-2w_i)
-(b_j-d_j-b_l+d_l)(u_i-u_k)-\tfrac12(e_j-e_l)(w_j-w_l)\}}\rho_{jl}
\end{split}
\end{equation}
for $(i,j,k,l)=(1,3,5,6), (2,3,7,9)$ and $(4,6,8,9)$.

For the dilogarithm part, any of the expressions
\eqref{K241}--\eqref{K244} may be used, since they are all equal by
Proposition~\ref{prop:K24}.  Here we choose \eqref{K243}.  Upon
replacing the indices $1,2,3,4$ with $i,j,k,l$, we obtain
\begin{equation}\label{Kdf}
\begin{split}
K_{ijkl}&=~\Psi _q\left(\e^{-b_i-c_j-d_k+u_i+u_k+w_i-w_j+w_k}\right)^{-1}
\\ & \qquad \cdot 
\Psi_{q^2}\left(\e^{-a_k-b_j-d_l+u_j+u_l+w_j-2 w_k+w_l}\right)^{-1}
\\ & \qquad \cdot \Psi_q\left(\e^{a_k-b_i+b_k-c_j+c_k+u_i-u_k+w_i-w_j+w_k}\right)^{-1}
\\ & \qquad \cdot \Psi_{q^2}\left(\e^{a_k+c_i-2 c_j+c_k-d_j+d_l+u_j-u_l-w_j+2 w_k-w_l}\right)
\\ & \qquad \cdot \Psi_q\left(\e^{-a_i-c_i+c_j-d_i+d_k+u_i-u_k-w_i+w_j-w_k}\right)
\\ & \qquad \cdot \Psi_{q^2}\left(\e^{a_k-2 b_i-b_j+2 b_k+b_l+c_i-2 c_j+c_k+2 c_l+2 u_i+u_j-2 u_k-u_l+2w_i-w_j+w_l}\right)^{-1}
\\ & \qquad \cdot \Psi_q\left(\e^{-b_i-b_j+b_k+b_l+c_i-c_j+2c_l+u_i+u_j-u_k-u_l+w_i-w_k+w_l}\right)^{-1}
\\ & \qquad \cdot \Psi _{q^2}\left(\e^{a_k+c_i-2c_j+c_k-d_j+d_l+u_j-u_l-w_j+2 w_k-w_l}\right)^{-1}
\\ & \qquad \cdot \Psi_q\left(\e^{-a_i-b_k-c_j-d_i-d_j+d_l+u_i+u_j+u_k-u_l-w_i+w_k-w_l}\right)
\\ & \qquad \cdot \Psi_{q^2}\left(\e^{-a_k-b_j-d_l+u_j+u_l+w_j-2 w_k+w_l}\right)P^K_{ijkl}.
    \end{split}
\end{equation}
Further setting $c_i=a_i$ and $c_k=a_k$ leads to the formula \eqref{kf}.

\subsection{Well-definedness of the dilogarithm part of the 3DRE: quantum $Y$-variables}
\label{dilog-p1}

In the remainder of this section, all the formulas are based on the 
expression 
\eqref{sol243} of the $K$-operator, which corresponds to the choice $(\ve_2, \ve_4)=(-,+)$. 

First we prove the well-definedness of the dilogarithm part of the 3DRE
as formal Laurent series in quantum $Y$-variables. 
The dilogarithm part of the LHS reads 
\begin{equation}\label{L46}
\begin{split}
&\Psi_q\!\left(Y_{17}^{-1}\right)^{-1}
 \Psi_q\!\left(q Y_{16} Y_{17}\right)
 \Psi_q\!\left(q^{-1} Y_{17}^{-1} Y_{18}^{-1}\right)^{-1}
 \Psi_q\!\left(Y_{9} Y_{16} Y_{17}\right)
\\
& \cdot
 \Psi_q\!\left(Y_{10}^{-1} Y_{17}^{-1} Y_{18}^{-1}\right)^{-1}
 \Psi_{q^2}\!\left(Y_{3}^{-1}\right)^{-1}
 \Psi_q\!\left(q^{-1} Y_{10}^{-1} Y_{11}^{-1} Y_{17}^{-1} Y_{18}^{-1}\right)^{-1}
 \Psi_{q^2}\!\left(q^{2} Y_{2} Y_{3}\right)
\\
&\cdot 
 \Psi_q\!\left(q Y_{9} Y_{10} Y_{16} Y_{17}\right)
 \Psi_{q^2}\!\left(q^{-2} Y_{3}^{-1} Y_{4}^{-1} Y_{10}^{-2} Y_{11}^{-2} Y_{17}^{-2} Y_{18}^{-2}\right)^{-1}
\\
&\cdot
 \Psi_q\!\left(q^{-1} Y_{3}^{-1} Y_{4}^{-1} Y_{10}^{-1} Y_{11}^{-1} Y_{17}^{-1} Y_{18}^{-1}\right)^{-1}
 \Psi_{q^2}\!\left(q^{2} Y_{2} Y_{3}\right)^{-1}
 \Psi_q\!\left(q^{2} Y_{2} Y_{3} Y_{9} Y_{10} Y_{11} Y_{16} Y_{17}\right)
\\
&\cdot
 \Psi_{q^2}\!\left(Y_{3}^{-1}\right)
 \Psi_q\!\left(q^{-2} Y_{3}^{-1} Y_{4}^{-1} Y_{10}^{-1} Y_{11}^{-1} Y_{12}^{-1} Y_{17}^{-1} Y_{18}^{-1}\right)^{-1}
\\
&\cdot
 \Psi_q\!\left(q^{-3} Y_{3}^{-1} Y_{4}^{-1} Y_{10}^{-1} Y_{11}^{-1} Y_{12}^{-1} Y_{13}^{-1} Y_{17}^{-1} Y_{18}^{-1}\right)^{-1}
 \Psi_q\!\left(q^{3} Y_{2} Y_{3} Y_{9} Y_{10} Y_{11} Y_{12} Y_{16} Y_{17}\right)
\\
&\cdot
 \Psi_q\!\left(q^{4} Y_{2} Y_{3} Y_{9} Y_{10} Y_{11} Y_{12} Y_{16} Y_{17} Y_{20}\right)
 \Psi_q\!\left(Y_{19}^{-1}\right)^{-1}
\\
&\cdot
 \Psi_q\!\left(q Y_{18} Y_{19}\right)
 \Psi_q\!\left(q^{-1} Y_{19}^{-1} Y_{20}^{-1}\right)^{-1}
 \Psi_q\!\left(Y_{11} Y_{18} Y_{19}\right)
 \Psi_q\!\left(Y_{12}^{-1} Y_{19}^{-1} Y_{20}^{-1}\right)^{-1}
\\
&\cdot
 \Psi_{q^2}\!\left(Y_{5}^{-1}\right)^{-1}
 \Psi_q\!\left(q^{-1} Y_{12}^{-1} Y_{13}^{-1} Y_{19}^{-1} Y_{20}^{-1}\right)^{-1}
 \Psi_{q^2}\!\left(q^{2} Y_{4} Y_{5}\right)
\\
&\cdot
 \Psi_q\!\left(q Y_{11} Y_{12} Y_{18} Y_{19}\right)
 \Psi_{q^2}\!\left(q^{-2} Y_{5}^{-1} Y_{6}^{-1} Y_{12}^{-2} Y_{13}^{-2} Y_{19}^{-2} Y_{20}^{-2}\right)^{-1}
\\
&\cdot
 \Psi_q\!\left(q^{-1} Y_{5}^{-1} Y_{6}^{-1} Y_{12}^{-1} Y_{13}^{-1} Y_{19}^{-1} Y_{20}^{-1}\right)^{-1}
 \Psi_{q^2}\!\left(q^{2} Y_{4} Y_{5}\right)^{-1}
\\
&\cdot
 \Psi_q\!\left(q^{2} Y_{4} Y_{5} Y_{11} Y_{12} Y_{13} Y_{18} Y_{19}\right)
 \Psi_{q^2}\!\left(Y_{5}^{-1}\right)
 \Psi_q\!\left(Y_{10}^{-1}\right)^{-1}
 \Psi_{q^2}\!\left(Y_{3}^{-1}\right)^{-1}
\\
&\cdot
 \Psi_q\!\left(q^{-1} Y_{10}^{-1} Y_{11}^{-1}\right)^{-1}
 \Psi_{q^2}\!\left(q^{2} Y_{2} Y_{3}\right)
 \Psi_q\!\left(q Y_{9} Y_{10}\right)
 \Psi_{q^2}\!\left(q^{-2} Y_{3}^{-1} Y_{4}^{-1} Y_{10}^{-2} Y_{11}^{-2}\right)^{-1}
\\
&\cdot
 \Psi_q\!\left(q^{-1} Y_{3}^{-1} Y_{4}^{-1} Y_{10}^{-1} Y_{11}^{-1}\right)^{-1}
 \Psi_{q^2}\!\left(q^{2} Y_{2} Y_{3}\right)^{-1}
 \Psi_q\!\left(q^{2} Y_{2} Y_{3} Y_{9} Y_{10} Y_{11}\right)
 \Psi_{q^2}\!\left(Y_{3}^{-1}\right)
\\
&\cdot
 \Psi_q\!\left(q^{-2} Y_{3}^{-1} Y_{4}^{-1} Y_{10}^{-1} Y_{11}^{-1} Y_{12}^{-1}\right)^{-1}
 \Psi_q\!\left(q^{-3} Y_{3}^{-1} Y_{4}^{-1} Y_{10}^{-1} Y_{11}^{-1} Y_{12}^{-1} Y_{13}^{-1}\right)^{-1}
\\
&\cdot
 \Psi_q\!\left(q^{3} Y_{2} Y_{3} Y_{9} Y_{10} Y_{11} Y_{12}\right)
 \Psi_q\!\left(q^{4} Y_{2} Y_{3} Y_{9} Y_{10} Y_{11} Y_{12} Y_{20}\right).
\end{split}
\end{equation}
It consists of $46 = 3 \times 10 + 4 \times 4$ dilogarithms.  
Expand the $i$th one (from the left) into a power series in its argument by means of \eqref{dilog-sum}, 
introducing the summation index $n_i$.  
By further using the $q$-commutativity relation \eqref{q-Y}, the result can be brought to the form
\begin{align}\label{CYp}
\sum_{\mathbf{n}} C(\mathbf{n}) \,
Y_2^{p_1} Y_3^{p_2} Y_4^{p_3} Y_5^{p_4} Y_6^{p_5} Y_9^{p_6}
Y_{10}^{p_7} Y_{11}^{p_8} Y_{12}^{p_9} Y_{13}^{p_{10}}
Y_{16}^{p_{11}} Y_{17}^{p_{12}} Y_{18}^{p_{13}} Y_{19}^{p_{14}} Y_{20}^{p_{15}},
\end{align}
where the sum runs over $\mathbf{n} = (n_1,\ldots,n_{46}) \in (\mathbb{Z}_{\ge0})^{46}$, 
and $C(\mathbf{n})$ is a rational function of $q$.  
The series \eqref{CYp} involves 15 $Y$-variables attached to the unfrozen vertices of the quiver $B(C_3)$ in Figure \ref{fig:RE}.  
Their powers $p_1,\ldots,p_{15}$ are linear forms in $\mathbf{n}$ given as follows:  
\begin{align}\label{pnL}
p_{1} &= n_2+n_4+n_9+n_{13}+n_{17}+n_{18},\nonumber \displaybreak[0] \\ 
p_{2} &= -n_1+n_2-n_3+n_4-n_5-n_7+n_9-2
   n_{10}-n_{11}+n_{13}-n_{15}-n_{16}+n_{17}+n_{18},\nonumber \displaybreak[0] \\ 
p_{3} &= -n_{28}-n_{29},\nonumber \displaybreak[0] \\ 
p_{4} &= -n_3-n_5-n_7-2
   n_{10}-n_{11}-n_{15}-n_{16}+n_{20}+n_{22}+n_{27}+n_{31},\nonumber \displaybreak[0] \\ 
p_{5} &= -n_{19}+n_{20}-n_{21}+n_{22}-n_{23}-n_{25}+n_{27}-2
   n_{28}-n_{29}+n_{31},\nonumber \displaybreak[0]
   \\ p_{6} &= -n_{24}+n_{26}-n_{28}-n_{29}+n_{30}+n_{31}-n_{32},\nonumber \displaybreak[0] \\ 
p_{7} &= -n_{16}-n_{25}-2n_{28}-n_{29}+n_{31}-n_{44},\nonumber \displaybreak[0] \\ 
p_{8} &= -n_{10}-n_{11}-n_{15}-n_{16}+n_{26}+n_{30}+n_{31}-n_{38}-n_{39}-n_{43}-n_{44},\nonumber \displaybreak[0] \\ 
p_{9} &= 
   n_{18}-n_{21}-n_{23}-n_{25}-2
   n_{28}-n_{29}+n_{46},\nonumber \displaybreak[0] \\ 
p_{10} &= n_4+n_9+n_{13}+n_{17}+n_{18}+n_{37}+n_{41}+n_{45}+n_{46},\nonumber \displaybreak[0] \\ 
p_{11} &= n_8+n_{12}+n_{13}+n_{17}+n_{18}+n_{36}+n_{40}+n_{41}+n_{45}+n_{46},\nonumber \displaybreak[0] \\ 
p_{12} &= -n_{15}-n_{16}+n_{17}+n_{18}-n_{23}-n_{25}+n_{27}-2n_{28}
-n_{29}+n_{31}-n_{43}-n_{44}+n_{45}+n_{46},\nonumber \displaybreak[0] \\ 
p_{13} &= -n_7-2n_{10}-n_{11}+n_{13}-n_{15}-n_{16}+n_{17}+n_{18}+n_{22}+n_{27}+n_{31}-n_{35}-2n_{38}
   \nonumber \displaybreak[0] \\
  & \quad -n_{39}+n_{41}-n_{43}-n_{44}+n_{45}+n_{46},\nonumber \displaybreak[0] \\ 
p_{14} &= -n_5-n_7+n_9-2
   n_{10}-n_{11}+n_{13}-n_{15}-n_{16}+n_{17}+n_{18}-n_{33}-n_{35}+n_{37}-2
   n_{38}
   \nonumber \\
   & \quad -n_{39}+n_{41}-n_{43}-n_{44}+n_{45}+n_{46},\nonumber \displaybreak[0] 
   \\ p_{15} &= -n_6+n_8-n_{10}-n_{11}+n_{12}+n_{13}-n_{14}-n_{15}-n_{16}+n_{17}+n_{18}-n_{34}+n_{36}-n_{38}
   \nonumber \\
   & \quad -n_{39}+n_{40}+n_{41}-n_{42}-n_{43}-n_{44}+n_{45}+n_{46}
   \end{align}
To prove that the product \eqref{L46} is well defined, 
it suffices to show that the system \eqref{pnL} admits only finitely many (possibly empty) 
solutions $\mathbf{n}$ for each fixed $(p_1,\ldots,p_{15}) \in \mathbb{Z}^{15}$.  
This is readily verified.  
From the equations involving $p_1,p_3,p_{10}, p_{11}$ and the condition $\forall n_i \in \mathbb{Z}_{\ge0}$, 
there are finitely many possibilities for
$n_2,n_4,n_8,n_9,n_{12},n_{13},n_{17},n_{18},n_{28}$, $n_{29},n_{36},n_{37},n_{40},n_{41},n_{45},n_{46}$.  
For each such choice, the equations involving $p_2,p_9, p_{14}, p_{15}$ take the form
\begin{equation*}
\begin{split}
p'_{2} &= -n_1-n_3-n_5-n_7-2 n_{10}-n_{11}-n_{15}-n_{16},\nonumber \\ 
p'_9 &= -n_{21}-n_{23}-n_{25},\nonumber \\ 
p'_{14} &= -n_5-n_7-2 n_{10}-n_{11}-n_{15}-n_{16}-n_{33}-n_{35}-2
   n_{38}-n_{39}-n_{43}-n_{44},\nonumber \\ 
p'_{15} &= -n_6-n_{10}-n_{11}-n_{14}-n_{15}-n_{16}-n_{34}-n_{38}-n_{39}-n_{42}-n_{43}-n_{44}
\end{split}
\end{equation*}
for certain $p'_2,p'_9,p'_{14}, p'_{15}$.  
These constrain 
$n_1,n_3,n_5,n_6,n_7,n_{10},n_{11},n_{14},n_{15},n_{16},n_{21},n_{23},n_{25}$, 
$n_{33},n_{34},n_{35},n_{38},n_{39},
n_{42},n_{43},n_{44}$ 
to finitely many values.  
Assuming such values fixed, the remaining equations become
\begin{equation*}
\begin{split}
p''_{4} &= n_{20}+n_{22}+n_{27}+n_{31},\quad
p''_5 = -n_{19}+n_{20}+n_{22}+n_{27}+n_{31}, 
\nonumber \\
p''_6 &= -n_{24}+n_{26}+n_{30}+n_{31}-n_{32},\quad 
p''_{7} = n_{31},\quad 
p''_{8} = n_{26}+n_{30}+n_{31},\nonumber \\ 
p''_{12} &= n_{27}+n_{31},\quad 
p''_{13} = n_{22}+n_{27} +n_{31}.
\end{split}
\end{equation*}
The equations for $p''_4,p''_7, p''_8, p''_{12}, p''_{13}$ constrain
$n_{20},n_{22},n_{26},n_{27},n_{30},n_{31}$
to finitely many possibilities. 
Once these are fixed, the remaining two equations reduce to 
$p''_{5} = -n_{19}$ and $p''_{6} = -n_{24}-n_{32}$ for some $p''_{5}$ and $p''_{6}$, 
thereby forcing $n_{19}, n_{24}$ and $n_{32}$ also to take only finitely many values.
Hence all variables $n_1,\ldots,n_{46}$ are confined to finitely many choices, 
which proves the well-definedness of \eqref{L46}.  
It also follows that the constant term of the series \eqref{L46}, 
corresponding to $\mathbf{n} = \mathbf{0}$, equals $1$.
  
For simplicity, we write the above procedure as
\begin{equation}\label{ppnL}
\begin{split}
p&: 1,3,10,11,\\
n&: 2,4,8,9,12,13,17,18,28,29,36,37,40,41,45,46,
\\
p&: 2,9,14,15,\\
n&: 1,3,5,6,7,10,11,14,15,16,21,23,25,33,34,35,38,39,42,43,44,
\\
p&: 4,7,8,12,13,\\
n&: {20},{22},{26},{27},{30},{31},
\\
p&: 5,6,\\
n&: 19,24,32.
\end{split}
\end{equation}

The same argument applies to the right-hand side of the 3DRE, 
showing its well-definedness and the unity of its constant term as well.  
We therefore omit the details here and simply present formulas 
analogous to \eqref{L46} and \eqref{pnL}.
The dilogarithm part of the RHS reads
\begin{equation}\label{R46}
\begin{split}
&\Psi_q\!\left(Y_{19}^{-1}\right)^{-1}
   \Psi_q\!\left(q Y_{18} Y_{19}\right)
   \Psi_q\!\left(q^{-1} Y_{19}^{-1} Y_{20}^{-1}\right)^{-1}
   \Psi_q\!\left(Y_{11} Y_{18} Y_{19}\right)                                   \\
&\cdot\Psi_q\!\left(Y_{12}^{-1} Y_{19}^{-1} Y_{20}^{-1}\right)^{-1}
   \Psi_{q^2}\!\left(Y_{5}^{-1}\right)^{-1}
   \Psi_q\!\left(q^{-1} Y_{12}^{-1} Y_{13}^{-1} Y_{19}^{-1} Y_{20}^{-1}\right)^{-1}
   \Psi_{q^2}\!\left(q^{2} Y_{4} Y_{5}\right)                                    \\
&\cdot\Psi_q\!\left(q Y_{11} Y_{12} Y_{18} Y_{19}\right)
   \Psi_{q^2}\!\left(q^{-2} Y_{5}^{-1} Y_{6}^{-1} Y_{12}^{-2} Y_{13}^{-2}
                     Y_{19}^{-2} Y_{20}^{-2}\right)^{-1}                         \\
&\cdot\Psi_q\!\left(q^{-1} Y_{5}^{-1} Y_{6}^{-1} Y_{12}^{-1} Y_{13}^{-1}
                     Y_{19}^{-1} Y_{20}^{-1}\right)^{-1}
   \Psi_{q^2}\!\left(q^{2} Y_{4} Y_{5}\right)^{-1}                               \\
&\cdot\Psi_q\!\left(q^{2} Y_{4} Y_{5} Y_{11} Y_{12} Y_{13} Y_{18} Y_{19}\right)
   \Psi_{q^2}\!\left(Y_{5}^{-1}\right)
   \Psi_q\!\left(Y_{10}^{-1}\right)^{-1}
   \Psi_{q^2}\!\left(Y_{3}^{-1}\right)^{-1}                                      \\
&\cdot\Psi_q\!\left(q^{-1} Y_{10}^{-1} Y_{11}^{-1}\right)^{-1}
   \Psi_{q^2}\!\left(q^{2} Y_{2} Y_{3}\right)
   \Psi_q\!\left(q Y_{9} Y_{10}\right)
   \Psi_{q^2}\!\left(q^{-2} Y_{3}^{-1} Y_{4}^{-1} Y_{10}^{-2} Y_{11}^{-2}\right)^{-1} \\
&\cdot\Psi_q\!\left(q^{-1} Y_{3}^{-1} Y_{4}^{-1} Y_{10}^{-1} Y_{11}^{-1}\right)^{-1}
   \Psi_{q^2}\!\left(q^{2} Y_{2} Y_{3}\right)^{-1}
   \Psi_q\!\left(q^{2} Y_{2} Y_{3} Y_{9} Y_{10} Y_{11}\right)
   \Psi_{q^2}\!\left(Y_{3}^{-1}\right)                                           \\
&\cdot\Psi_q\!\left(q^{-2} Y_{3}^{-1} Y_{4}^{-1} Y_{10}^{-1} Y_{11}^{-1}
                     Y_{12}^{-1}\right)^{-1}
   \Psi_q\!\left(q^{-3} Y_{3}^{-1} Y_{4}^{-1} Y_{10}^{-1} Y_{11}^{-1}
                 Y_{12}^{-1} Y_{13}^{-1}\right)^{-1}                             \\
&\cdot\Psi_q\!\left(q^{3} Y_{2} Y_{3} Y_{9} Y_{10} Y_{11} Y_{12}\right)
   \Psi_q\!\left(q^{4} Y_{2} Y_{3} Y_{9} Y_{10} Y_{11} Y_{12} Y_{20}\right)       \\
&\cdot\Psi_q\!\left(Y_{17}^{-1} Y_{18}^{-1} Y_{19}^{-1}\right)^{-1}
   \Psi_q\!\left(q Y_{16} Y_{17} Y_{18} Y_{19}\right)
   \Psi_q\!\left(q^{-1} Y_{17}^{-1} Y_{18}^{-1} Y_{19}^{-1} Y_{20}^{-1}\right)^{-1} \\
&\cdot\Psi_q\!\left(Y_{11} Y_{16} Y_{17} Y_{18} Y_{19}\right)
   \Psi_q\!\left(Y_{12}^{-1} Y_{17}^{-1} Y_{18}^{-1} Y_{19}^{-1} Y_{20}^{-1}\right)^{-1}
   \Psi_{q^2}\!\left(Y_{5}^{-1}\right)^{-1}                                      \\
&\cdot\Psi_q\!\left(q^{-1} Y_{12}^{-1} Y_{13}^{-1} Y_{17}^{-1} Y_{18}^{-1}
                     Y_{19}^{-1} Y_{20}^{-1}\right)^{-1}
   \Psi_{q^2}\!\left(q^{2} Y_{4} Y_{5}\right)
   \Psi_q\!\left(q Y_{11} Y_{12} Y_{16} Y_{17} Y_{18} Y_{19}\right)              \\
&\cdot\Psi_{q^2}\!\left(q^{-2} Y_{5}^{-1} Y_{6}^{-1} Y_{12}^{-2} Y_{13}^{-2}
                         Y_{17}^{-2} Y_{18}^{-2} Y_{19}^{-2} Y_{20}^{-2}\right)^{-1} \\
&\cdot\Psi_q\!\left(q^{-1} Y_{5}^{-1} Y_{6}^{-1} Y_{12}^{-1} Y_{13}^{-1}
                     Y_{17}^{-1} Y_{18}^{-1} Y_{19}^{-1} Y_{20}^{-1}\right)^{-1}
   \Psi_{q^2}\!\left(q^{2} Y_{4} Y_{5}\right)^{-1}                               \\
&\cdot\Psi_q\!\left(q^{2} Y_{4} Y_{5} Y_{11} Y_{12} Y_{13} Y_{16}
                     Y_{17} Y_{18} Y_{19}\right)
   \Psi_{q^2}\!\left(Y_{5}^{-1}\right)                                          \\
&\cdot\Psi_q\!\left(Y_{4} Y_{6}^{-1} Y_{11} Y_{17}^{-1} Y_{20}^{-1}\right)^{-1}
   \Psi_q\!\left(q^{-1} Y_{17}^{-1} Y_{18}^{-1}\right)^{-1}
   \Psi_q\!\left(q Y_{16} Y_{17}\right)
   \Psi_q\!\left(Y_{2} Y_{4}^{-1} Y_{9} Y_{13}^{-1} Y_{16} Y_{17} Y_{20}\right).
\end{split}
\end{equation}
Unlike the LHS \eqref{L46}, there are two sign-incoherent arguments in the bottom line.

The equations corresponding to \eqref{pnL} read
\begin{align}\label{pnR}
p_{1} &= -n_{15}-n_{17}+n_{19}-2
   n_{20}-n_{21}+n_{23}-n_{25}-n_{26}+n_{27}+n_{28},\nonumber \displaybreak[0] \\ 
   p_{2} &= -n_{16}+n_{18}-n_{20}-n_{21}+n_{22}+n_{23}-n_{24}-n_{25}-n_{26
   }+n_{27}+n_{28},\nonumber \displaybreak[0] \\ 
   p_{3} &= -n_5-n_7+n_9-2 n_{10}-n_{11}+n_{13}-n_{25}-n_{26}+n_{27}+n_{28}-n_{33}-n_{35}+n_{37}
\nonumber  \\
 & \quad -2n_{38}-n_{39}+n_{41},\nonumber \displaybreak[0] \\ 
   p_{4} &= -n_1+n_2-n_3+n_4-n_5-n_7+n_9-2
   n_{10}-n_{11}+n_{13}-n_{29}+n_{30}-n_{31}+n_{32}
   \nonumber \\
   & \quad -n_{33}-n_{35}+n_{37}-2n_{38}-n_{39}+n_{41},\nonumber \displaybreak[0] \\ 
   p_{5} &= -n_6+n_8-n_{10}-n_{11}+n_{12}+n_{13}-n_{14}-n_{34}+n_{36}-n_{38}-n_{39}+n_{40}+n_{41}-n_{4
   2},\nonumber \displaybreak[0] \\ 
   p_{6} &= -n_{10}-n_{11}-n_{38}-n_{39}-n_{43},\nonumber \\ 
   p_{7} &= n_4+n_9+n_{13}-n_{17}-2
   n_{20}-n_{21}+n_{23}-n_{25}-n_{26}+n_{27}+n_{28}+n_{32}+n_{37}+n_{41}+n_{43},\nonumber \displaybreak[0] \\ 
   p_{8} &= n_2+n_4+n_9+n_{13}-n_{29}+n_{30}-n
   _{31}+n_{32}-n_{33}-n_{35}+n_{37}-2 n_{38}-n_{39}+n_{41}-n_{44},\nonumber \displaybreak[0] \\ 
   p_{9} &= -n_7-2 n_{10}-n_{11}+n_{13}-n_{26}-n_{35}-2
   n_{38}-n_{39}+n_{41}-n_{46},\nonumber \displaybreak[0] \\ 
   p_{10} &= n_8+n_{12}+n_{13}-n_{20}-n_{21}-n_{25}-n_{26}+n_{36}+n_{40}+n_{41}+n_{43}-n_{46},
 \nonumber \displaybreak[0] \\
  p_{11} &= n_{19}+n_{23}+n_{27}+n_{28}+n_{46},\nonumber \displaybreak[0] \\ 
  p_{12} &= n_{18}+n_{22}+n_{23}+n_{27}+n_{28}+n_{46},\nonumber \displaybreak[0] \\ 
  p_{13} &= -n_3-n_5-n_7-2
   n_{10}-n_{11}+n_{28}-n_{31}-n_{33}-n_{35}-2
   n_{38}-n_{39}-n_{43}+n_{46},\nonumber \displaybreak[0] \\ 
   p_{14} &= n_{30}+n_{32}+n_{37}+n_{41}+n_{45}+n_{46},\nonumber \displaybreak[0] \\ 
   p_{15} &= -n_{29}+n_{30}-n_{31}+n_{32}-n_{33}-n_{35}+n_{37}-2 n_{38}-n_{39}+n_{41}-n_{43}-n_{44}+n_{45}+n_{46}.
   \end{align}
 As in \eqref{ppnL}, one can confirm the well-definedness along the following procedure:
\begin{equation}\label{ppnR}
\begin{split}
p&: 6,11,12,14,\\
n&:10,11,18,19,22,23,27,28,30,32,37,38,39,41,43,45,46,
\\
p&: 1,2,13,15,\\
n&: 3,5,7,15,16,17,20,21,24,25,26,29,31,33,35,44,
\\
p&: 3,7,8,9,10,\\
n&: 2,4,8,9,12,13,36,40,
\\
p&: 4,5,\\
n&: 1,6,14,34,42.
\end{split}
\end{equation}

\subsection{Well-definedness of the dilogarithm part of the 3DRE: $q$-Weyl variables}\label{dilog-p2}

According to \eqref{Ypw}, \eqref{Ypwb} and \eqref{YwC2}, 
the ring homomorphism from $\mathcal{Y}(B(C_3))$ to $\mathrm{Frac}\,\mathcal{W}(C_3)$ is given by
 \begin{equation}\label{yeuw}
 \begin{cases}
 Y_1 \mapsto  \e^{a_2+d_3-u_3+2 w_2-w_3}, & Y_{12} \mapsto  \e^{a_7+b_5+c_6+d_8-u_5-u_8-w_5+w_6+w_7-w_8}, \\
 Y_2 \mapsto  \e^{e_3+2 u_3}, & Y_{13} \mapsto  \e^{e_8+2 u_8}, \\
 Y_3 \mapsto  \e^{a_5+b_3+d_6-u_3-u_6-w_3+2 w_5-w_6}, & Y_{14} \mapsto  \e^{b_8+c_9-u_8-w_8+w_9}, \\
 Y_4 \mapsto  \e^{e_6+2 u_6}, & Y_{15} \mapsto  \e^{d_1-u_1-w_1}, \\
 Y_5 \mapsto \e^{ a_8+b_6+d_9-u_6-u_9-w_6+2 w_8-w_9}, &  Y_{16} \mapsto  \e^{e_1+2 u_1}, \\ 
 Y_6 \mapsto  \e^{e_9+2 u_9}, &  Y_{17} \mapsto  \e^{b_1+c_2+d_4-u_1-u_4-w_1+w_2-w_4}, \\
 Y_7 \mapsto  \e^{b_9-u_9-w_9}, &  Y_{18} \mapsto  \e^{e_4+2 u_4}, \\
 Y_8 \mapsto  \e^{a_1+d_2-u_2+w_1-w_2}, & Y_{19} \mapsto  \e^{b_4+c_5+d_7-u_4-u_7-w_4+w_5-w_7}, \\
 Y_9 \mapsto  \e^{e_2+2 u_2}, &  Y_{20} \mapsto  \e^{e_7+2 u_7}, \\
 Y_{10} \mapsto  \e^{a_4+b_2+c_3+d_5-u_2-u_5-w_2+w_3+w_4-w_5}, &  Y_{21} \mapsto  \e^{b_7+c_8-u_7-w_7+w_8}, \\
 Y_{11} \mapsto  \e^{e_5+2 u_5}, &  Y_{22} \mapsto  \e^{c_1+c_4+c_7+w_1+w_4+w_7}.
 \end{cases}
 \end{equation}
 
First we consider the LHS of the 3D  reflection equation.
Substituting \eqref{yeuw} into the series \eqref{CYp} and applying the $q$-commutativity of 
the generators on the $q$-Weyl algebra, one can express it as 
\begin{align}\label{anL}
\sum_{{\bf n}} \tilde{C}({\bf n}) \e^{\alpha_1u_1+\cdots +\alpha_9u_9+ 
\alpha_{10}w_1+\cdots +\alpha_{18}w_9},
\end{align}
 where the sum extends over ${\bf n} = (n_1,\ldots, n_{46}) \in (\Z_{\ge 0})^{46}$, and 
$\tilde{C}({\bf n})$ is  function of $q$ and the parameters in \eqref{yeuw}.
 The coefficients $\alpha_1,\ldots, \alpha_{18}$ are given by 
\begin{align}
 \alpha_{1} &= n_1+n_2+n_3+n_4+n_5+n_7+n_9+2 n_{10}+n_{11}+n_{13}
 +n_{15}+n_{16}+n_{17}+n_{18},\nonumber \displaybreak[0] \\ 
 \alpha_{2} &= 2 n_4+n_5+n_7+n_9+2
   n_{10}+n_{11}+n_{13}+n_{15}+n_{16}+n_{17}+n_{18}+n_{33}+n_{35}+n_{37}
   \nonumber \\
   &\quad +2n_{38}+n_{39}+n_{41}+n_{43}+n_{44}+n_{45}+n_{46},\nonumber \displaybreak[0] \\ 
   \alpha_{3} &= n_6+n_8+n_{10}+n_{11}+n_{12}+n_{13}+n_{14}+n_{15}+n_{16}+n_{17
   }+n_{18}+n_{34}+n_{36}+n_{38}
   \nonumber \\
   &\quad +n_{39}+n_{40}+n_{41}+n_{42}+n_{43}+n_{44}+n_{45}+n_{46},\nonumber \displaybreak[0] \\ 
   \alpha_{4} &= n_1-n_2-n_3-n_4-n_5-n_7-n
   _9-2 n_{10}-n_{11}-n_{13}-n_{15}-n_{16}-n_{17}-n_{18}
   \nonumber \\
   &\quad +n_{19}+n_{20}+n_{21}+n_{22}+n_{23}+n_{25}+n_{27}+2
   n_{28}+n_{29}+n_{31},\nonumber \displaybreak[0] \\
   \alpha_{5} &= n_5-n_7-n_9-2 n_{10}-n_{11}+n_{13}+2 n_{22}+n_{23}+n_{25}+n_{27}+2
   n_{28}+n_{29}+n_{31}
   \nonumber \\
   &\quad +n_{33}-n_{35}-n_{37}-2
   n_{38}-n_{39}+n_{41},\nonumber \displaybreak[0] \\ 
   \alpha_{6} &= n_6-n_8-n_{10}-n_{11}-n_{12}-n_{13}+n_{14}-n_{15}-n_{16}-n_{17}-n_{18}+n_{24}+n_{26}
   +n_{28} \nonumber \\
   &\quad +n_{29}+n_{30}+n_{31}+n_{32}+n_{34}-n_{36}-n_{38}-n_{39}-n_{40}
   -n_{41}+n_{42}-n_{43}-n_{44}-n_{45}-n_{46},\nonumber \displaybreak[0] \\ 
\alpha_{7} &= 2
   n_{18}+n_{19}-n_{20}-n_{21}-n_{22}-n_{23}-n_{25}-n_{27}-2 n_{28}-n_{29}-n_{31}+2
   n_{46},\nonumber \displaybreak[0] \\ 
   \alpha_{8} &= n_{15}-n_{16}-n_{17}-n_{18}+n_{23}-n_{25}-n_{27}-2
   n_{28}-n_{29}+n_{31}+n_{43}-n_{44}-n_{45}-n_{46},\nonumber \displaybreak[0] \\ 
   \alpha_{9} &= n_{24}-n_{26}-n_{28}-n_{29}-n_{30}-n_{31}+n_{32},\nonumber \displaybreak[0] \\ 
   \alpha_{10} &= n_1-n_2
   +n_3-n_4+n_5+n_7-n_9+2
   n_{10}+n_{11}-n_{13}+n_{15}+n_{16}-n_{17}-n_{18},\nonumber \displaybreak[0] \\ 
   \alpha_{11} &= -n_1+n_2-n_3+n_4+n_{33}+n_{35}-n_{37}+2
   n_{38}+n_{39}-n_{41}+n_{43}+n_{44}-n_{45}-n_{46},\nonumber \displaybreak[0] \\ 
   \alpha_{12} &= -n_5+n_6-n_7-n_8+n_9-n_{10}-n_{12}+n_{14}-n_{33}+n_{34}-n_{35
   }-n_{36}+n_{37}-n_{38}-n_{40}+n_{42},\nonumber \displaybreak[0] \\ 
   \alpha_{13} &= n_1-n_2+n_3-n_4+n_{19}-n_{20}+n_{21}-n_{22}+n_{23}+n_{25}-n_{27}+2
   n_{28}+n_{29}-n_{31}
   \nonumber \\
   &\quad -n_{33}-n_{35}+n_{37}-2 n_{38}-n_{39}+n_{41}-n_{43}-n_{44}+n_{45}+n_{46},\nonumber \displaybreak[0] \\ 
   \alpha_{14} &= n_5-2 n_6+n_7+2
   n_8-n_9-n_{11}+2 n_{12}+n_{13}-2 n_{14}-n_{19}+n_{20}-n_{21}+n_{22}
   \nonumber \\
   &\quad +n_{33}-2 n_{34}+n_{35}+2 n_{36}-n_{37}-n_{39}+2
   n_{40}+n_{41}-2n_{42},\nonumber \displaybreak[0]\\ 
   \alpha_{15} &= n_6-n_8+n_{10}+n_{11}-n_{12}-n_{13}+n_{14}-n_{23}+n_{24}-n_{25}-n_{26}+n_{27}-n_{28}-n_{30}
   \nonumber \\
   &\quad +n_{32}+n_{34}-n_{36}+n_{38}+n_{39}-n_{40}-n_{41}+n_{42},\nonumber \displaybreak[0] \\ 
   \alpha_{16} &= -n_{15}-n_{16}+n_{17}+n_{18}+n_{19}-n_{20}+n_{21}-n_{22}-n_{43}-n
   _{44}+n_{45}+n_{46},\nonumber \displaybreak[0] \\ 
   \alpha_{17} &= n_{15}+n_{16}-n_{17}-n_{18}+n_{23}-2 n_{24}+n_{25}+2 n_{26}-n_{27}-n_{29}+2
   n_{30}+n_{31}-2 n_{32}
   \nonumber \\
   &\quad +n_{43}+n_{44}-n_{45}-n_{46},\nonumber \displaybreak[0] \\ 
   \alpha_{18} &= n_{24}-n_{26}+n_{28}+n_{29}-n_{30}-n_{31}+n_{32}.
   \label{alL}
\end{align}
In the notation analogous to \eqref{ppnL}, 
one can show the well-definedness of \eqref{anL} only in two steps as
\begin{equation}\label{aalL}
\begin{split}
\alpha&: 1,2,3,\\
n&:1,2,3,4,5,6,7,8,9,10,11,12,13,14,15,16,17,18,33,34,35,36,37,38,39,40,
\\ &\;\; \;41,42,43,44,45,46,
\\
\alpha&: 4,5,6,\\
n&: 19,20,21,22,23,24,25,26,27,28,29,30,31,32.
\end{split}
\end{equation}

As for the RHS of the 3DRE, the equations corresponding to \eqref{alL} read  
\begin{align}
\nonumber \\ \alpha_{1} &= n_{29}+n_{30}+n_{31}+n_{32}+n_{33}+n_{35}+n_{37}+2
   n_{38}+n_{39}+n_{41}+n_{43}+n_{44}+n_{45}+n_{46},\nonumber \displaybreak[0] \\ 
   \alpha_{2} &= n_{15}+n_{17}+n_{19}+2
   n_{20}+n_{21}+n_{23}+n_{25}+n_{26}+n_{27}+n_{28}+2
   n_{46},\nonumber \displaybreak[0] \\ 
   \alpha_{3} &= n_{16}+n_{18}+n_{20}+n_{21}+n_{22}+n_{23}+n_{24}+n_{25}+n_{26}+n_{27}+n_{28}+2
   n_{46},\nonumber \displaybreak[0] \\ 
   \alpha_{4} &= n_1+n_2+n_3+n_4+n_5+n_7+n_9+2 n_{10}+n_{11}+n_{13}+n_{43}-n_{44}-n_{45}-n_{46},
   \nonumber \displaybreak[0] \\ 
   \alpha_{5} &= 2 n_4+n_5+n_7+n_9+2
   n_{10}+n_{11}+n_{13}+n_{15}-n_{17}-n_{19}-2 n_{20}-n_{21}+n_{23}+2 n_{32}
   \nonumber \\
   &\quad +n_{33}+n_{35}+n_{37}+2
   n_{38}+n_{39}+n_{41}+2
   n_{43},\nonumber \displaybreak[0] \\ 
   \alpha_{6} &= n_6+n_8+n_{10}+n_{11}+n_{12}+n_{13}+n_{14}+n_{16}-n_{18}-n_{20}-n_{21}-n_{22}-n_{23}+n_{24}
   \nonumber \\
   &\quad -n_{25}-n_{26
   }-n_{27}-n_{28}+n_{34}+n_{36}+n_{38}+n_{39}+n_{40}+n_{41}+n_{42}+2 n_{43}-2
   n_{46},\nonumber \displaybreak[0] \\ 
   \alpha_{7} &= n_1-n_2-n_3-n_4-n_5-n_7-n_9-2 n_{10}-n_{11}-n_{13}+2
   n_{28}+n_{29}-n_{30}-n_{31}
   \nonumber \\
   &\quad -n_{32}-n_{33}-n_{35}-n_{37}-2 n_{38}-n_{39}-n_{41}-2 n_{43}+2 n_{46},\nonumber \\ 
   \alpha_{8} &= n_5-n_7-n_9-2
   n_{10}-n_{11}+n_{13}+n_{25}-n_{26}-n_{27}-n_{28}+n_{33}-n_{35}-n_{37}-2 n_{38}
   \nonumber \\
   &\quad -n_{39}+n_{41}-2
   n_{46},\nonumber \displaybreak[0] \\ 
   \alpha_{9} &= n_6-n_8-n_{10}-n_{11}-n_{12}-n_{13}+n_{14}+n_{34}-n_{36}-n_{38}-n_{39}-n_{40}-n_{41}+n_{42}-2
   n_{43},\nonumber \displaybreak[0] \\
    \alpha_{10} &= n_{29}-n_{30}+n_{31}-n_{32}+n_{33}+n_{35}-n_{37}+2
   n_{38}+n_{39}-n_{41}+n_{43}+n_{44}-n_{45}-n_{46},\nonumber \displaybreak[0] \\ 
   \alpha_{11} &= n_{15}+n_{17}-n_{19}+2
   n_{20}+n_{21}-n_{23}+n_{25}+n_{26}-n_{27}-n_{28}-n_{29}+n_{30}-n_{31}
   \nonumber \\
   &\quad +n_{32}-n_{33}-n_{35}+n_{37}-2
   n_{38}-n_{39}+n_{41}-n_{43}-n_{44}+n_{45}+n_{46},\nonumber \displaybreak[0] \\ 
   \alpha_{12} &= -n_{15}+n_{16}-n_{17}-n_{18}+n_{19}-n_{20}-n_{22}+n_{24},\nonumber \displaybreak[0] \\ 
   \alpha_{13} &= n_1-n_2+n_3-n_4+n_5+n_7-n_9+2 n_{10}+n_{11}-n_{13}-n_{15}-n_{17}+n_{19}-2n_{20}
   \nonumber \\
   &\quad -n_{21}+n_{23}-n_{25}-n_{26}+n_{27}+n_{28}+2 n_{29}-2 n_{30}+2 n_{31}-2 n_{32}+2 n_{33}+2 n_{35}
   \nonumber \\
   &\quad -2 n_{37}+4
   n_{38}+2 n_{39}-2 n_{41}+n_{43}+n_{44}-n_{45}-n_{46},\nonumber \displaybreak[0] \\ 
   \alpha_{14} &= -n_1+n_2-n_3+n_4+n_{15}-2 n_{16}+n_{17}+2
   n_{18}-n_{19}-n_{21}+2 n_{22}+n_{23}-2
   n_{24}
   \nonumber \\
   &\quad -n_{29}+n_{30}-n_{31}+n_{32},\nonumber \displaybreak[0] \\ 
   \alpha_{15} &= -n_5+n_6-n_7-n_8+n_9-n_{10}-n_{12}+n_{14}+n_{16}-n_{18}+n_{20}+n_{21}-n_{22}-n_{23}
   \nonumber \\
   &\quad +n_{24}-n_{33}+n_{34}-n_{35}-n_{36}+n_{37}-n_{38}-n_{40}+n_{42},\nonumber \displaybreak[0] \\ 
   \alpha_{16} &= n_1-n_2+n_3-n_4-n_{25}-n_{26}+n_{27}+n
   _{28}+n_{29}-n_{30}+n_{31}-n_{32},\nonumber \displaybreak[0] \\ 
   \alpha_{17} &= n_5-2 n_6+n_7+2 n_8-n_9-n_{11}+2 n_{12}+n_{13}-2
   n_{14}+n_{25}+n_{26}-n_{27}-n_{28}+n_{33}
   \nonumber \\
   &\quad -2 n_{34}+n_{35}+2 n_{36}-n_{37}-n_{39}+2 n_{40}+n_{41}-2
   n_{42},\nonumber \displaybreak[0] \\ 
\alpha_{18} &= n_6-n_8+n_{10}+n_{11}-n_{12}-n_{13}+n_{14}+n_{34}-n_{36}+n_{38}+n_{39}-n_{40}-n_{41}+n_{42}.
   \label{alR}
\end{align}
This time, the  well-defineness is shown along the following:
\begin{equation}\label{aalR}
\begin{split}
\alpha&: 1,2,3,\\
n&:15,16,17,18,19,20,21,22,23,24,25,26,27,28,29,30,31,32,33,35,37,\\
&\;\;\; 38,39, 41,43,44,45,46,
\\
\alpha&: 4,5,6,\\
n&: 1,2,3,4,5,6,7,8,9,10,11,12,13,14,34,36,40,42.
\end{split}
\end{equation}
 
\section{Reduction to the $K$-operators for FG quiver}\label{sec:fg}

\subsection{$K$-operators for the FG quiver}
Let us recall the $K$-operator constructed for the FG quiver in \cite{IKT1}.
Corresponding to the wiring diagrams \eqref{SB-KC}, we have 
the transformation of FG quivers as 
\begin{align}
\label{quiver-d:C2}
\begin{split}
\begin{tikzpicture}
\begin{scope}[>=latex,xshift=0pt]
{\color{red}
\fill (2,0) circle(2pt) coordinate(A) node[below]{$1$};
\fill (4,0) circle(2pt) coordinate(C) node[below]{$3$};
\fill (3,1) circle(2pt) coordinate(B) node[above]{$2$};
\fill (5,1) circle(2pt) coordinate(D) node[above]{$4$};
\draw [->] (1,0.5) to [out = 0, in = 135] (A) to [out = -45, in = -135] (C)--(D) to [out = -45, in = 180] (6,0.5);
\draw [->] (1,-0.5) to [out = 0, in = -135] (A) -- (B) -- (C) to [out = -45, in = 180] (6,-0.5);
}
\draw[->] (6.5,0.3) -- (7.5,0.3);
\draw (7.1,0.3) circle(0pt) node[below] {$\qK_{1234}$};%
%
%
\path (2,1) node[circle]{2} coordinate(E) node[above=0.2em]{$1$};
\draw (2,1) circle[radius=0.15];
\path (4,1) node[circle]{2} coordinate(F) node[above=0.2em]{$2$};
\draw (4,1) circle[radius=0.15];
\path (6,1) node[circle]{2} coordinate(G) node[above=0.2em]{$3$};
\draw (6,1) circle[radius=0.15];
\draw (1,0) circle(2pt) coordinate(B) node[below]{$4$};
\draw (3,0) circle(2pt) coordinate(C) node[below]{$5$};
\draw (5,0) circle(2pt) coordinate(D) node[below]{$6$};
\qarrow{B}{C}
\qarrow{C}{D}
\draw[->,shorten >=4pt,shorten <=4pt] (E) -- (F) [thick];
\draw[->,shorten >=4pt,shorten <=4pt] (F) -- (G) [thick];
\draw[->,shorten >=2pt,shorten <=4pt] (F) -- (C) [thick];
\draw[->,shorten >=4pt,shorten <=2pt] (C) -- (E) [thick];
\draw[->,shorten >=4pt,shorten <=2pt] (D) -- (F) [thick];
\qdarrow{E}{B}
\qdarrow{G}{D} 
\path (3,-0.9) node {$B_{\mathrm{FG}}(C_2)$};
\end{scope}
\begin{scope}[>=latex,xshift=205pt]
{\color{red}
\fill (5,0) circle(2pt) coordinate(A) node[below]{$1$};
\fill (3,0) circle(2pt) coordinate(C) node[below]{$3$};
\fill (4,1) circle(2pt) coordinate(B) node[above]{$2$};
\fill (2,1) circle(2pt) coordinate(D) node[above]{$4$};
\draw [->] (1,0.5) to [out = 0, in = -135] (D)-- (C) to [out = -45, in = -135] (A) to [out = 45, in = 180] (6,0.5);
\draw [->] (1,-0.5) to [out = 0, in = -135] (C) -- (B) -- (A) to [out = -45, in = 180] (6,-0.5);
}
%
\path (1,1) node[circle]{2} coordinate(E) node[above=0.2em]{$1$};
\draw (1,1) circle[radius=0.15];
\path (3,1) node[circle]{2} coordinate(F) node[above=0.2em]{$2$};
\draw (3,1) circle[radius=0.15];
\path (5,1) node[circle]{2} coordinate(G) node[above=0.2em]{$3$};
\draw (5,1) circle[radius=0.15];
\draw (2,0) circle(2pt) coordinate(B) node[below]{$4$};
\draw (4,0) circle(2pt) coordinate(C) node[below]{$5$};
\draw (6,0) circle(2pt) coordinate(D) node[below]{$6$};
\qarrow{B}{C}
\qarrow{C}{D}
\draw[->,shorten >=4pt,shorten <=4pt] (E) -- (F) [thick];
\draw[->,shorten >=4pt,shorten <=4pt] (F) -- (G) [thick];
\draw[->,shorten >=4pt,shorten <=2pt] (C) -- (F) [thick];
\draw[->,shorten >=2pt,shorten <=4pt] (F) -- (B) [thick];
\draw[->,shorten >=2pt,shorten <=4pt] (G) -- (C) [thick];
\draw[->,dashed,shorten >=4pt,shorten <=2pt] (B) -- (E) [thick];
\draw[->,dashed,shorten >=4pt,shorten <=2pt] (D) -- (G) [thick];
\path (3,-0.9) node {$B'_{\mathrm{FG}}(C_2)$};
\end{scope}
\end{tikzpicture}
\end{split}
\end{align}
This consists of the mutation sequence $\mu_2 \mu_5 \mu_2$. 
For a sigh sequence $\ve=(\ve_1,\ve_2,\ve_3) \in \{1,-1\}^3$, 
the corresponding mutation sequence of quantum $Y$-seeds
\begin{align}\label{ms:C2}
(B_{\mathrm FG}(C_2),\scY) = (B^{(1)},\scY^{(1)}) \stackrel{\mu_2}{\longrightarrow} 
(B^{(2)},\scY^{(2)}) \stackrel{\mu_5}{\longrightarrow}
(B^{(3)},\scY^{(3)}) \stackrel{\mu_2}{\longrightarrow}
(B^{(4)},\scY^{(4)}) = (B'_{\mathrm FG}(C_2),Y')
\end{align}
induces the isomorphism 
$\widehat{K}_{1234}: \mathcal{Y}(B'_{\mathrm FG}(C_2)) \to \mathcal{Y}(B_{\mathrm FG}(C_2))$ 
expressed as
\begin{align}\label{Kve:C2}
\widehat{K}_{1234} = 
\mathrm{Ad}\bigl(\Psi_q((Y^{(1)}_2)^{\varepsilon_1})^{\varepsilon_1}\bigr)\tau_{2,\varepsilon_1}
\mathrm{Ad}\bigl(\Psi_q((Y^{(2)}_5)^{\varepsilon_2})^{\varepsilon_2}\bigr)\tau_{5,\varepsilon_2}
\mathrm{Ad}\bigl(\Psi_q((Y^{(3)}_2)^{\varepsilon_3})^{\varepsilon_3}\bigr)\tau_{2,\varepsilon_3}.
\end{align} 

Let $\scY_i$ and $\scY'_i$ denote the generators of 
$\mathcal{Y}(B_{\mathrm{FG}}(C_2))$ and $\mathcal{Y}(B'_{\mathrm{FG}}(C_2))$, respectively.  
Using the canonical variables \eqref{uw-def}, 
we define the embeddings
$\phi_{\mathrm{FG}}: \mathcal{Y}(B_{\mathrm{FG}}(C_2)) \hookrightarrow \mathrm{Frac}\,\mathcal{W}(C_2)$ and 
$\phi'_{\mathrm{FG}}: \mathcal{Y}(B'_{\mathrm{FG}}(C_2)) \hookrightarrow \mathrm{Frac}\,\mathcal{W}(C_2)$,  
involving the complex parameters $\theta_i ~(i=1,2,3,4)$  as follows:
\begin{align}\label{phi-K}
\phi_{\mathrm{FG}}: 
\begin{cases}
\scY_1 \mapsto \e^{-w_2-u_2+2 w_1-\theta_2},
\\
\scY_2 \mapsto \e^{-w_2+u_2-w_4-u_4+2 w_3+\theta_2-\theta_4},
\\
\scY_3 \mapsto \e^{-w_4+u_4+\theta_4},
\\
\scY_4 \mapsto \e^{-w_1-u_1-\theta_1},
\\
\scY_5 \mapsto \e^{-w_1+u_1-w_3-u_3+w_2+\theta_1-\theta_3},
\\
\scY_6 \mapsto \e^{-w_3+u_3+w_4+\theta_3},
\end{cases}
\qquad
\phi'_{\mathrm{FG}}:
\begin{cases}
\scY'_1 \mapsto \e^{-w_4-u_4-\theta_4},
\\
\scY'_2 \mapsto \e^{-w_4+u_4-w_2-u_2+2 w_3-\theta_2+\theta_4},
\\
\scY'_3 \mapsto \e^{-w_2+u_2+2 w_1+\theta_2},
\\
\scY'_4 \mapsto \e^{-w_3-u_3+w_4-\theta_3},
\\
\scY'_5 \mapsto \e^{-w_3+u_3-w_1-u_1+w_2-\theta_1+\theta_3},
\\
\scY'_6 \mapsto \e^{-w_1+u_1+\theta_1}. 
\end{cases}
\end{align} 
Here we adjust the definition of canonical pairs in \cite{IKT1} to those in \eqref{uw-def} in this article. 
Further, we define the isomorphism $\pi^K_{C_2}$ of $\mathcal{W}(C_2)$ 
given by the following affine transformation of canonical variables:
\begin{align}\label{K-pi}
\pi^K_{C_2}:
\begin{cases}
\;w_1 \mapsto w_1- \theta_{24}, 
\quad 
w_2 \mapsto w_4+ 2w_1- \theta_{24},
\\
\;w_3  \mapsto w_3+\theta_{24},
\quad
w_4 \mapsto w_2-2 w_1+ \theta_{24},
\\
\;u_1 \mapsto u_1+u_2-u_4,
\quad 
u_2  \mapsto u_4,
\\
\; u_3  \mapsto u_3,
\quad 
u_4 \mapsto u_2.
\end{cases}
\end{align}
Here we set $\theta_{24} = \theta_2-\theta_4$. 

Set
\begin{align}
\label{K++-C2}
\begin{split}
&{K}_{C_2:++-} 
= 
\Psi_{q^2}(\e^{\theta_{24}-w_2+u_2-w_4-u_4+2w_3}) \Psi_q(\e^{\theta_{13}-w_1+u_1-w_3-u_3+w_2}) 
\\
& \qquad \qquad \qquad \cdot\Psi_{q^2}(\e^{\theta_{24}-w_2+u_2-w_4-u_4+2w_3})^{-1} {P}_{C_2},  
\end{split}
\\
\label{FG-P-C2}
&{P}_{C_2}= \e^{\frac{1}{\hbar}\left(w_1(u_4-u_2)+\theta_{24}(u_3-u_1)\right)} \rho_{24}.
\end{align} 

\begin{proposition}\cite[\S 3]{IKT1}
\label{Kop-FG}
The isomorphism $\widehat{K}_{1234}$ in \eqref{Kve:C2}, with the sign
sequence $\ve=(\ve_1,\ve_2,\ve_3)=(1,1,-1)$, is realized, in its image
in $\mathrm{Frac}\,\mathcal{W}(C_2)$, by the adjoint action of
${K}_{C_2:++-}$ in \eqref{K++-C2}.  Namely, the following commutative
diagram holds:
\begin{align}\label{Kcom-FG}
\xymatrix{
\mathcal{Y}(B'_{\mathrm FG}(C_2)) \ar[r]^{\phi'_{\mathrm FG}} \ar[d]_{\widehat{K}_{1234}} 
& \mathrm{Frac}\,\mathcal{W}(C_2) \ar[d]_{\mathrm{Ad} \,{K}_{C_2:++-}}
\\
\mathcal{Y}(B_{\mathrm FG}(C_2)) \ar[r]^{\phi_{\mathrm FG}}& \mathrm{Frac}\,\mathcal{W}(C_2)
}
\end{align}
In particular, it holds that $\pi^K_{C_2}=\mathrm{Ad}\,\mathcal{P}_{C_2}$,
and the diagram below is also commutative:
\begin{align}\label{Kcom-FG2}
\xymatrix{
\mathcal{Y}(B'_{\mathrm FG}(C_2)) \ar[r]^{\phi'_{\mathrm FG}} \ar[d]_{\tau_{2,+}\tau_{5,+}\tau_{2,-}} 
& \mathrm{Frac}\,\mathcal{W}(C_2) \ar[d]_{\mathrm{Ad} \,\mathcal{P}_{C_2}}
\\
\mathcal{Y}(B_{\mathrm FG}(C_2)) \ar[r]^{\phi_{\mathrm FG}}& \mathrm{Frac}\,\mathcal{W}(C_2)
}
\end{align}
\end{proposition}

We introduce an alternative realization of $\widehat{K}_{1234}$ associated with a sign sequence
different from that appearing in Proposition \ref{Kop-FG}.

\begin{proposition}\label{Kop-FGb}
The isomorphism $\widehat{K}_{1234}$ with the sign sequence
$\ve=(-1,1,1)$ is realized, in the same sense as in Proposition
\ref{Kop-FG}, by the operator $\mathcal{K}_{C_2:-++}$ defined by
\begin{align}
\begin{split}
\label{K-++C2}
&{K}_{C_2:-++} 
= 
\Psi_{q^2}(\e^{-\theta_{24}+w_2-u_2+w_4+u_4-2w_3})^{-1} 
\Psi_q(\e^{\theta_{13}+\theta_{24}+u_1-u_3+ u_2- u_4 -w_1+w_3-w_4}) 
\\
& \qquad \qquad \qquad \cdot \Psi_{q^2}(\e^{-\theta_{24}+w_2-u_2+w_4+u_4-2w_3}) P_{C_2},
\end{split}
\end{align}
where the monomial part ${P}_{C_2}$ is again given by \eqref{FG-P-C2}. 
Moreover, ${K}_{C_2:-++}$ coincides with ${K}_{C_2:++-}$.
\end{proposition}
\begin{proof}
The operator ${K}_{C_2:-++}$ is constructed in the same manner as ${K}_{C_2:++-}$. 
One verifies that the dilogarithm parts of the two
operators agree by applying Lemma \ref{dilog-2}, in the same way as in
the proof of Proposition \ref{prop:K24}.
\end{proof}

Now we introduce the dual of \eqref{quiver-d:C2} associated with the Weyl group $W(B_2)$. 
We consider the following transformation of FG quivers:
\begin{align}
\label{quiver-d:B2}
\begin{split}
\begin{tikzpicture}
\begin{scope}[>=latex,xshift=0pt]
{\color{red}
\fill (2,0) circle(2pt) coordinate(A) node[below]{$1$};
\fill (4,0) circle(2pt) coordinate(C) node[below]{$3$};
\fill (3,1) circle(2pt) coordinate(B) node[above]{$2$};
\fill (5,1) circle(2pt) coordinate(D) node[above]{$4$};
\draw [-] (1,0.5) to [out = 0, in = 135] (A);%
\draw [-] (A) to [out = -45, in = -135] (C); 
\draw [-] (C) -- (D);
\draw [->] (D) to [out = -45, in = 180] (6,0.5);
\draw [-] (1,-0.5) to [out = 0, in = -135] (A); 
\draw [-] (A) -- (B);
\draw [-] (B) -- (C);
\draw [->] (C) to [out = -45, in = 180] (6,-0.5);
}
\draw[->] (6.5,0.3) -- (7.5,0.3);
\draw (7.1,0.3) circle(0pt) node[below] {$\qK_{1234}$};
%
%
\draw (2,1) circle(2pt) coordinate(E) node[above]{$1$};
\draw (4,1) circle(2pt) coordinate(F) node[above]{$2$};
\draw (6,1) circle(2pt) coordinate(G) node[above]{$3$};
\draw (1,0) node[circle]{2} coordinate(B) node[below=0.2em]{$4$};
\draw (1,0) circle[radius=0.15];
\path (3,0) node[circle]{2} coordinate(C) node[below=0.2em]{$5$};
\draw (3,0) circle[radius=0.15];
\path (5,0) node[circle]{2} coordinate(D) node[below=0.2em]{$6$};
\draw (5,0) circle[radius=0.15];
\qarrow{E}{F}
\qarrow{F}{G}
\draw[->,shorten >=4pt,shorten <=4pt] (B) -- (C) [thick];
\draw[->,shorten >=4pt,shorten <=4pt] (C) -- (D) [thick];
\draw[->,shorten >=4pt,shorten <=2pt] (F) -- (C) [thick];
\draw[->,shorten >=2pt,shorten <=4pt] (C) -- (E) [thick];
\draw[->,shorten >=2pt,shorten <=4pt] (D) -- (F) [thick];
\draw[->,dashed,shorten >=4pt,shorten <=2pt] (E) -- (B) [thick];
\draw[->,dashed,shorten >=4pt,shorten <=2pt] (G) -- (D) [thick];
\path (3,-0.9) node {$B_{\mathrm{FG}}(B_2)$};
\end{scope}
\begin{scope}[>=latex,xshift=205pt]
{\color{red}
\fill (5,0) circle(2pt) coordinate(A) node[below]{$1$};
\fill (3,0) circle(2pt) coordinate(C) node[below]{$3$};
\fill (4,1) circle(2pt) coordinate(B) node[above]{$2$};
\fill (2,1) circle(2pt) coordinate(D) node[above]{$4$};
\draw [-] (1,0.5) to [out = 0, in = -135] (D);
\draw [-] (D) -- (C); 
\draw [-] (C) to [out = -45, in = -135] (A);
\draw [->] (A) to [out = 45, in = 180] (6,0.5);
\draw [-] (1,-0.5) to [out = 0, in = -135] (C); 
\draw [-] (C) -- (B);
\draw [-] (B) -- (A);
\draw [->] (A) to [out = -45, in = 180] (6,-0.5);
}
%
\draw (1,1) circle(2pt) coordinate(E) node[above]{$1$};
\draw (3,1) circle(2pt) coordinate(F) node[above]{$2$};
\draw (5,1) circle(2pt) coordinate(G) node[above]{$3$};
\path (2,0) node[circle]{2} coordinate(B) node[below=0.2em]{$4$};
\draw (2,0) circle[radius=0.15];
\path (4,0) node[circle]{2} coordinate(C) node[below=0.2em]{$5$};
\draw (4,0) circle[radius=0.15];
\path (6,0) node[circle]{2} coordinate(D) node[below=0.2em]{$6$};
\draw (6,0) circle[radius=0.15];
\qarrow{E}{F}
\qarrow{F}{G}
\draw[->,shorten >=4pt,shorten <=4pt] (B) -- (C) [thick];
\draw[->,shorten >=4pt,shorten <=4pt] (C) -- (D) [thick];
\draw[->,shorten >=2pt,shorten <=4pt] (C) -- (F) [thick];
\draw[->,shorten >=4pt,shorten <=2pt] (F) -- (B) [thick];
\draw[->,shorten >=4pt,shorten <=2pt] (G) -- (C) [thick];
\draw[->,dashed,shorten >=2pt,shorten <=4pt] (B) -- (E) [thick];
\draw[->,dashed,shorten >=2pt,shorten <=4pt] (D) -- (G) [thick];
\path (3,-0.9) node {$B'_{\mathrm{FG}}(B_2)$};
\end{scope}
\end{tikzpicture}
\end{split}
\end{align}
It corresponds to the mutation sequence $\mu_2 \mu_5 \mu_2$.
The difference between \eqref{quiver-d:C2} and \eqref{quiver-d:B2} 
is that in \eqref{quiver-d:C2} the quiver vertices of weight $2$ are on the wall, 
whereas in \eqref{quiver-d:C2} the quiver vertices of weight $1$ are on the wall.   
Note that the quiver $B_{\mathrm{FG}}(B_2)$ (resp. $B'_{\mathrm{FG}}(B_2)$) 
coincides with $B'_{\mathrm{FG}}(C_2)$ (resp. $B_{\mathrm{FG}}(C_2)$) 
by identifying the vertices as $(1,2,3,4,5,6) \mapsto (4,5,6,1,2,3)$ and the crossings as $(1,2,3,4) \mapsto (4,3,2,1)$.
For $\qK_{1234}=\mu_2 \mu_5 \mu_2$, define a sequence of quantum seeds and the isomorphism 
$\widehat{K}_{1234} : \mathcal{Y}(B'_{\mathrm FG}(B_2)) \to \mathcal{Y}(B_{\mathrm FG}(B_2))$ 
in the same manner as \eqref{ms:C2} and \eqref{Kve:C2}.  

Let $\scY_i$ and $\scY'_i$ denote the generators of 
$\mathcal{Y}(B_{\mathrm{FG}}(B_2))$ and $\mathcal{Y}(B'_{\mathrm{FG}}(B_2))$, respectively.  
Set $\gamma=(2,1,2,1)$, and let $\mathcal{W}(B_2):= \mathcal{W}_\gamma$ 
be the $q$-Weyl algebra generated by $\e^{\pm u_i}, \e^{\pm w_i}$ with relations \eqref{qW-gamma}.
Define the embeddings $\psi_{\mathrm{FG}}: \mathcal{Y}(B_{\mathrm{FG}}(B_2)) \hookrightarrow \mathrm{Frac}\,\mathcal{W}(B_2)$ 
and $\psi'_{\mathrm{FG}}: \mathcal{Y}(B'_{\mathrm{FG}}(B_2)) \hookrightarrow \mathrm{Frac}\,\mathcal{W}(B_2)$ 
involving complex parameters $\theta_i ~(i=1,2,3,4)$ by
\begin{align}\label{phi-K-B}
\psi_{\mathrm{FG}}: 
\begin{cases}
\scY_1 \mapsto \e^{-w_2-u_2+ w_1-\theta_2},
\\
\scY_2 \mapsto \e^{-w_2+u_2-w_4-u_4+ w_3+\theta_2-\theta_4},
\\
\scY_3 \mapsto \e^{-w_4+u_4+\theta_4},
\\
\scY_4 \mapsto \e^{-w_1-u_1-\theta_1},
\\
\scY_5 \mapsto \e^{-w_1+u_1-w_3-u_3+2w_2+\theta_1-\theta_3},
\\
\scY_6 \mapsto \e^{-w_3+u_3+2w_4+\theta_3},
\end{cases}
\qquad
\psi'_{\mathrm{FG}}:
\begin{cases}
\scY'_1 \mapsto \e^{-w_4-u_4-\theta_4},
\\
\scY'_2 \mapsto \e^{-w_4+u_4-w_2-u_2+ w_3-\theta_2+\theta_4},
\\
\scY'_3 \mapsto \e^{-w_2+u_2+ w_1+\theta_2},
\\
\scY'_4 \mapsto \e^{-w_3-u_3+2w_4-\theta_3},
\\
\scY'_5 \mapsto \e^{-w_3+u_3-w_1-u_1+2w_2-\theta_1+\theta_3},
\\
\scY'_6 \mapsto \e^{-w_1+u_1+\theta_1}. 
\end{cases}
\end{align} 

We define the isomorphism $\pi^K_{B_2}$ of $\mathcal{W}(B_2)$ by
\begin{align}\label{K-pi-B}
\pi^K_{B_2}:
\begin{cases}
\;w_1 \mapsto w_1- 2 \theta_{24}, 
\quad 
w_2 \mapsto w_4+ w_1- \theta_{24},
\\
\;w_3  \mapsto w_3+2\theta_{24},
\quad
w_4 \mapsto w_2- w_1+ \theta_{24},
\\
\;u_1 \mapsto u_1+2u_2-2u_4,
\quad 
u_2  \mapsto u_4,
\\
\; u_3  \mapsto u_3,
\quad 
u_4 \mapsto u_2,
\end{cases}
\end{align}
in the sense of exponentials.
We obtain the following result, in parallel with Proposition \ref{Kop-FG} and \ref{Kop-FGb}. 

\begin{proposition}
For sign sequences $\ve=(1,1,-1)$ and $(-1,1,1)$, 
the transformation $\widehat{K}_{1234}$ has two expressions 
$\psi_{\mathrm{FG}} \circ \widehat{K}_{1234} = \mathrm{Ad}({K}_{B_2:++-})\circ 
\psi'_{\mathrm{FG}}= \mathrm{Ad}({K}_{B_2:-++}) \circ \psi'_{\mathrm{FG}}$ with
\begin{align}
\label{K++-B2}
\begin{split}
&{K}_{B_2:++-} 
= 
\Psi_{q}(\e^{\theta_{24}-w_2+u_2-w_4-u_4+w_3}) \Psi_{q^2}(\e^{\theta_{13}-w_1+u_1-w_3-u_3+2w_2}) 
\\
&\qquad \qquad \qquad \cdot
\Psi_{q}(\e^{\theta_{24}-w_2+u_2-w_4-u_4+w_3})^{-1} {P}_{B_2},
\end{split}
\\
\begin{split}
\label{K-++B2}
&{K}_{B_2:-++} 
= 
\Psi_{q}(\e^{-\theta_{24}+w_2-u_2+w_4+u_4-w_3})^{-1} \Psi_{q^2}(\e^{\theta_{13}+2\theta_{24}+u_1-u_3+2 u_2-2 u_4 -w_1+w_3-2w_4}), 
\\
& \qquad \qquad \qquad \cdot \Psi_{q}(\e^{-\theta_{24}+w_2-u_2+w_4+u_4-w_3}){P}_{B_2},
\end{split}
\\
\label{FG-P-B2}
&{P}_{B_2}= \e^{\frac{1}{\hbar}\left(w_1(u_4-u_2) + \theta_{24}(u_3-u_1)\right)} \rho_{24}.
\end{align} 
In particular, both monomial transformations $\tau_{2,+}\tau_{5,+} \tau_{2,-}$ 
and $\tau_{2,-}\tau_{5,+} \tau_{2,+}$  for $\widehat{K}_{1234}$ 
are realized as the adjoint action of ${P}_{B_2}$. Moreover, ${K}_{B_2:++-}$ coincides with ${K}_{B_2:-++}$.

\end{proposition}

\subsection{Limit of $K$-operators of type $\rho_{24}$}

In a manner parallel to \cite[\S 8.2]{IKSTY}, we define homomorphisms of
skewfields
$\alpha:\mathcal{Y}(B_{\mathrm{FG}}(C_2))\to\mathcal{Y}(B(C_2))$ and
$\alpha':\mathcal{Y}(B'_{\mathrm{FG}}(C_2))\to\mathcal{Y}(B'(C_2))$ given by the same formulata as
follows:
\begin{align}
\alpha, \alpha': \begin{cases}
\mathscr{Y}_1 \mapsto Y_1, & \mathscr{Y}_4 \mapsto Y_6,
\\
\mathscr{Y}_2 \mapsto q^2 Y_2Y_3, & \mathscr{Y}_5 \mapsto qY_7Y_8,
\\
\mathscr{Y}_3 \mapsto q^2 Y_4 Y_5, & \mathscr{Y}_6 \mapsto qY_9Y_{10}.
\end{cases}
\end{align}
We consider the following diagram   
\begin{align}\label{cd-C2}
\begin{CD}
\mathcal{Y}(B_\mathrm{FG}(C_2)) @> {\alpha}>> \mathcal{Y}(B(C_2)) 
\\
@V{\phi_\mathrm{FG}}VV @VV{\phi}V \\
\mathrm{Frac}\,\mathcal{W}(C_2) @>{\mathrm{id}}>> \mathrm{Frac}\,\mathcal{W}(C_2)\\
@V{\mathrm{Ad} {P}_{C_2}}VV @VV{\mathrm{Ad}P_{24}}V \\
\mathrm{Frac}\,\mathcal{W}(C_2) @>{\mathrm{id}}>> \mathrm{Frac}\,\mathcal{W}(C_2)\\
@A{\phi'_\mathrm{FG}}AA @AA{\phi'}A \\
\mathcal{Y}(B'_\mathrm{FG}(C_2)) @> {\alpha'}>> 
\mathcal{Y}(B'(C_2)) 
\end{CD}
\end{align}
where $\phi$, $\phi'$ and $\mathrm{Ad}(P_{24})$ are defined 
by \eqref{YwC2}, \eqref{YpwC2} and \eqref{tau24uw}, respectively.

When we impose the commutativity of the diagram \eqref{cd-C2}, 
the parameters are required to satisfy the following relations:
\begin{align}
&\theta_2= -a_1-d_2, 
\label{condi}\\
&\theta_2-\theta_4 = e_2+a_3+b_2+d_4,
\\
&\theta_4 = e_4+b_4,
\\
&\theta_1 = -d_1,
\\
&\theta_1-\theta_3 = e_1+b_1+c_2+d_3,
\\
&\theta_3 = e_3+b_3+c_4,
\label{condi2}
\\
&a_1+c_1-c_2+c_4=0,
\label{condro2}
\\
&b_2+d_2+2a_1=b_4+d_4,
\\
&2(\theta_2-\theta_4) = b_2-b_4-d_2+d_4.
\label{condro}
\end{align}
Here, the commutativity of the upper square (resp. the middle square) in
\eqref{cd-C2} corresponds to the relations
(\ref{condi})--(\ref{condi2}) (resp.\ (\ref{condro2})--(\ref{condro})),
and the commutativity of the lower square follows from these.
By taking into account the conditions \eqref{econ} and \eqref{ccon}, we get  
\begin{equation}\label{kai24}
\begin{split}
&a_1 = a_3 = -c_1 = -b_2+b_4+\theta_2-\theta_4,  
\\ 
&c_2 = c_4,
\\ 
& c_3 = b_2-b_4+ 2 c_4-\theta_2+\theta_4,
\\ 
& d_1 =  -\theta _1, 
\\ 
& d_2 =  b_2-b_4-2 \theta_2 + \theta_4, 
\\ 
& d_3 =  -c_4 - \theta_3,
\\ 
& d_4 =  -\theta_4, 
\\ 
& e_1 =  -b_1+\theta_1,
\\ 
& e_2 = e_4 =  -b_4 +\theta _4,
\\ 
& e_3 = -b_3 - c_4 + \theta _3.
\end{split}
\end{equation}

\begin{theorem}
\label{thm:type24}
In the limit 
\begin{align}\label{lim24}
b_i \rightarrow +\infty, \quad b_i+e_i = \text{fixed}\;\;(i=1,2,3,4),
\quad 
b_1-b_3 \rightarrow +\infty,\quad b_2-b_4 = \text{fixed},
\end{align} 
the operators $K_{-+}$ \eqref{sol243} and $K_{--}$ \eqref{sol244} (explicitly \eqref{K243} and \eqref{K244}) 
are reduced to
 ${K}_{C_2:++-}$ \eqref{K++-C2} and 
${K}_{C_2:-++}$  \eqref{K-++C2}, respectively.
\end{theorem}

\begin{proof}
Substitution of (\ref{kai24}) into (\ref{K243})  leads to
\begin{equation}\label{K243r}
\begin{split}
K_{-+}=&~\Psi_q\left(\e^{-b_1+\theta_3+u_1+u_3+w_1-w_2+w_3}\right)^{-1}  
\Psi _{q^2}\left(\e^{c_1-b_2 +\theta_4+u_2+u_4+w_2-2 w_3+w_4}\right)^{-1} 
\\ & \cdot
\Psi _q\left(\e^{-b_1+b_3+c_4+u_1-u_3+w_1-w_2+w_3}\right)^{-1}   
\underline{\Psi _{q^2}\left(\e^{\theta _2-\theta_4+u_2-u_4-w_2+2 w_3-w_4}\right)} 
\\ & \cdot \underline{\Psi _q\left(\e^{\theta _1-\theta
   _3+u_1-u_3-w_1+w_2-w_3}\right)} 
\\ & \cdot \Psi _{q^2}\left(\e^{-2 b_1+2 b_3+2 c_4-\theta_2+ \theta _4+2 u_1+u_2-2 u_3-u_4+2 w_1-w_2+w_4}\right)^{-1} 
\\ & \cdot \Psi_q\left(\e^{-b_1+b_3+c_4-\theta_2+ \theta_4+u_1+u_2-u_3-u_4+w_1-w_3+w_4}\right)^{-1} 
\\ & \cdot \underline{\Psi_{q^2}\left(\e^{\theta _2-\theta _4+u_2-u_4-w_2+2 w_3-w_4}\right)^{-1}} 
\\ & \cdot \Psi _q\left(\e^{-b_3-c_4+\theta _1+\theta_2-\theta _4+u_1+u_2+u_3-u_4-w_1+w_3-w_4}\right) 
\\ & \cdot \Psi_{q^2}\left(\e^{-b_4-\theta _2+u_2+u_4+w_2-2 w_3+w_4}\right) {P}_{C_2},
   \end{split}
\end{equation}
where $P_{24}$ reduces to ${P}_{C_2}$ \eqref{FG-P-C2} at this stage.
In the limit \eqref{lim24}, only the underlined quantum dilogarithms in \eqref{K243r} survive,
while all the remaining factors converge to~$1$.
Consequently, \eqref{K243r} reduces precisely to ${K}_{C_2:++-}$.

Substitution of (\ref{kai24}) into (\ref{K244}) leads to
\begin{equation}\label{K244r}
\begin{split}
K_{--}=&~\Psi _q\left(\e^{-b_1+\theta_3+u_1+u_3+w_1-w_2+w_3}\right)^{-1}
 \Psi _{q^2}\left(\e^{-b_4-d_4-\theta _2+\theta_4+u_2+u_4+w_2-2 w_3+w_4}\right)^{-1}
\\ & \cdot \Psi _q\left(\e^{-b_1+b_3+u_1-u_3+w_1-w_2+w_3}\right)^{-1}
\underline{\Psi _{q^2}\left(\e^{-\theta _2+\theta_4-u_2+u_4+w_2-2 w_3+w_4}\right)^{-1}}
\\ & \cdot \underline{\Psi _q\left(\e^{\theta_1+\theta _2-\theta _3-\theta _4+u_1+u_2-u_3-u_4-w_1+w_3-w_4}\right)}
\\ & \cdot \Psi _{q^2}\left(\e^{-2b_1+2 b_3-\theta _2+\theta _4+2 u_1+u_2-2 u_3-u_4+2w_1-w_2+w_4}\right)^{-1}
\\ & \cdot \Psi _q\left(\e^{-b_1+b_3+c_4-\theta_2+\theta_4+u_1+u_2-u_3-u_4+w_1-w_3+w_4}\right)^{-1}
\underline{\Psi _{q^2}\left(\e^{-\theta _2+\theta _4-u_2+u_4+w_2-2
   w_3+w_4}\right)}
\\ & \cdot \Psi _q\left(\e^{-b_3-c_4+\theta _1+\theta _2-\theta_4+u_1+u_2+u_3-u_4-w_1+w_3-w_4}\right)
\\ & \cdot \Psi_{q^2}\left(\e^{-b_4-\theta _2+2 \theta _4+u_2+u_4+w_2-2 w_3+w_4}\right)
{P}_{C_2}.
   \end{split}
\end{equation}
Similarly, in the limit \eqref{lim24}, \eqref{K244r} is reduced to ${K}_{C_2:-++}$ \eqref{K-++C2}. 
\end{proof}

\subsection{The limit of $K$-operators of type $\rho_{13}$}\label{sec:lim13}

In the same spirit as in the case of type $\rho_{24}$,  
we define homomorphisms of skewfields $\beta: \mathcal{Y}(B_\mathrm{FG}(B_2)) \to \mathcal{Y}(B(C_2))$
and $\beta': \mathcal{Y}(B'_\mathrm{FG}(B_2)) \to \mathcal{Y}(B'(C_2))$ in exactly the same way as follows:
\begin{align}
\beta, \beta': \begin{cases}
\mathscr{Y}_1 \mapsto Y_{10}, & \mathscr{Y}_4 \mapsto Y_5,
\\
\mathscr{Y}_2 \mapsto q Y_9Y_8, & \mathscr{Y}_5 \mapsto q^2Y_4Y_3,
\\
\mathscr{Y}_3 \mapsto q Y_7 Y_6, & \mathscr{Y}_6 \mapsto q^2Y_2Y_{1}.
\end{cases}
\end{align}
Define a ring homomorphism of skewfields $\iota: \mathrm{Frac}\,\mathcal{W}(C_2) \to \mathrm{Frac}\,\mathcal{W}(B_2)$ 
given by $u_i \mapsto u_{5-i}$ and $w_i \mapsto w_{5-i}$.   
We also use $\iota$ to interchange the label of parameters as
$\theta_i \leftrightarrow \theta_{5-i}$. 
When consider the commutativity of the diagram
\begin{align}\label{cd-B2}
\begin{CD}
\mathcal{Y}(B_\mathrm{FG}(B_2)) @> {\beta}>> \mathcal{Y}(B(C_2)) 
\\
@V{\psi_\mathrm{FG}}VV @VV{\phi}V \\
\mathrm{Frac}\,\mathcal{W}(B_2) @>{\iota}>> \mathrm{Frac}\,\mathcal{W}(C_2)\\
@V{\mathrm{Ad}{P}_{B_2}}VV @VV{\mathrm{Ad}P_{13}}V \\
\mathrm{Frac}\,\mathcal{W}(B_2) @>{\iota}>> \mathrm{Frac}\,\mathcal{W}(C_2)\\
@A{\psi'_\mathrm{FG}}AA @AA{\phi'}A \\
\mathcal{Y}(B'_\mathrm{FG}(B_2)) @> {\beta'}>> 
\mathcal{Y}(B'(C_2)) 
\end{CD}
\end{align}
\begin{align}
&\theta_3 = -b_3-c_4,  
\label{con1}\\ 
&\theta_1-\theta_3 =  -b_1-c_2-d_3-e_3,  \\ 
& \theta_1 = d_1+e_1, \\ 
& \theta_4 = b_4,  \\ 
   & \theta_2-\theta_4=-a_3-b_2-d_4-e_4,  \\ 
   & \theta_2 = a_1+d_2+e_2,  
   \label{con2}\\ 
   & a_1-a_3+b_1-b_3+c_1-c_3+d_1-d_3 = 0,  
   \label{con3}\\ 
   & -a_1+a_3+b_1-b_3-c_1+c_3-2 c_4-d_1+d_3+2 \theta_1-2\theta_3 = 0,  \\ 
   & b_1-b_3-d_1+d_3+2 \theta_1-2 \theta_3 =0.
   \label{con4}
   \end{align}
Here, the commutativity of the upper square (resp.\ the middle square) of
\eqref{cd-B2} corresponds precisely to the relations
(\ref{con1})--(\ref{con2}) (resp.\ (\ref{con3})--(\ref{con4})), and the
commutativity of the lower square follows from these.

By solving 13 conditions 
(\ref{econ}), (\ref{ccon}), (\ref{con1})--(\ref{con4}) in total, we get
\begin{equation}\label{kai13}
\begin{split}
&a_3=  a_1 = -c_1, \\ &c_2=  d_1-d_3-\theta_1+\theta_3, \\ &c_3=  -a_1+2 d_1-2 d_3-2
   \theta_1+2 \theta_3, \\ &c_4=  d_1-d_3-\theta_1+\theta_3, \\ &b_1=  -\theta_1, \\ &b_2= 
   -a_1-\theta_2, \\ &b_3=  -d_1+d_3+\theta_1-2 \theta_3, \\ &b_4=  -\theta_4, \\ &e_1= e_3 = \theta_1-d_1, 
\\ &e_2=  -a_1-d_2+\theta_2, \\ &e_4=  \theta_4-d_4.
\end{split}
\end{equation}
  
\begin{theorem}
In the limit 
\begin{align}\label{lim13}
d_i \rightarrow + \infty, \quad d_i+e_i = \text{fixed}\; (i=1,2,3,4),  \quad
d_2- d_4 \rightarrow -\infty,\quad d_1-d_3 = \text{fixed}, 
\end{align}
the operators $K_{-+}$ \eqref{sol133} and $K_{--}$ \eqref{sol134}
(explicitly \eqref{K133} and \eqref{K134}) are reduced to
the images under $\iota$ of
${K}^{B_2:++-}_{\mathrm{FG}}(B_2)$ \eqref{K++-B2} and
${K}^{B_2:-++}_{\mathrm{FG}}(B_2)$ \eqref{K-++B2}, respectively.
\end{theorem}

\begin{proof}
Substitution of (\ref{kai13}) into (\ref{K133}) leads to
\begin{equation}\label{K133r}
\begin{split}
K_{-+}=&~\Psi _q\left(\e^{-d_1+2 \theta_1-\theta
   _3+u_1+u_3+w_1-w_2+w_3}\right)^{-1}
\Psi _{q^2}\left(\e^{-d_4+\theta
   _2+u_2+u_4+w_2-2 w_3+w_4}\right)^{-1}
\\ & \cdot \underline{\Psi_q\left(\e^{-\theta_1+\theta_3-u_1+u_3-w_1+w_2-w_3}\right)}
\Psi_{q^2}\left(\e^{a_1+d_2-d_4-u_2+u_4+w_2-2 w_3+w_4}\right)^{-1}
\\ & \cdot \Psi_q\left(\e^{a_1+d_2-d_4+\theta_1-\theta
   _3-u_1-u_2+u_3+u_4+w_1-w_3+w_4}\right)^{-1}
\\ & \cdot \underline{\Psi _{q^2}\left(\e^{-\theta
   _2+\theta_4-u_2+u_4-w_2+2 w_3-w_4}\right)}
\\ & \cdot \underline{\Psi_q\left(\e^{-\theta_1+\theta_3-u_1+u_3-w_1+w_2-w_3}\right)^{-1}}
\\ & \cdot \Psi _{q^2}\left(\e^{a_1+d_2-d_4+2 \theta_1-2
   \theta_3-2 u_1-u_2+2 u_3+u_4+2 w_1-w_2+w_4}\right)^{-1}
\\ & \cdot \Psi_q\left(\e^{-d_1+2 \theta_1-\theta_3+u_1+u_3+w_1-w_2+w_3}\right)
\\ & \cdot \Psi_{q^2}\left(\e^{-a_1-d_2-2 \theta_1+2 \theta_3+\theta_4-2 u_1+u_2+2 u_3+u_4-2w_1+w_2-w_4}\right)\iota({P}_{B_2})
   \end{split}
\end{equation}
In the limit \eqref{lim13}, only the underlined quantum dilogarithms in \eqref{K133r} survive. 
Hence \eqref{K133r} is reduced to the image of $\iota$ of ${K}_{B_2:++-}$ \eqref{K++-B2}.

Substitution of (\ref{kai13}) into (\ref{K134}) leads to
\begin{equation}\label{K134r}
\begin{split}
K_{--} =&~\Psi _q\left(\e^{-d_1+2 \theta_1-\theta_3+u_1+u_3+w_1-w_2+w_3}\right)^{-1}
\Psi _{q^2}\left(\e^{-d_4+\theta_2+u_2+u_4+w_2-2 w_3+w_4}\right)^{-1}
\\ & \cdot \underline{\Psi_q\left(\e^{\theta_1-\theta_3+u_1-u_3+w_1-w_2+w_3}\right)^{-1}}
\Psi_{q^2}\left(\e^{a_1+d_2-d_4-u_2+u_4+w_2-2 w_3+w_4}\right)^{-1}
\\ & \cdot \Psi_q\left(\e^{a_1+d_2-d_4+\theta_1-\theta_3-u_1-u_2+u_3+u_4+w_1-w_3+w_4}\right)^{-1}
\\ & \cdot \underline{\Psi _{q^2}\left(\e^{-2 \theta_1-\theta_2+2 \theta_3+\theta_4-2 u_1-u_2+2 u_3+u_4-2 w_1+w_2-w_4}\right)}
\\ & \cdot \underline{\Psi _q\left(\e^{\theta_1-\theta_3+u_1-u_3+w_1-w_2+w_3}\right)}
\\ & \cdot \Psi_{q^2}\left(\e^{a_1+d_2-d_4+2 \theta_1-2 \theta_3-2 u_1-u_2+2 u_3+u_4+2w_1-w_2+w_4}\right)^{-1}
\\ & \cdot \Psi _q\left(\e^{-d_1+2 \theta_1-\theta_3+u_1+u_3+w_1-w_2+w_3}\right)
\\ & \cdot \Psi _{q^2}\left(\e^{-a_1-d_2-2 \theta_1+2 \theta_3+\theta
   _4-2 u_1+u_2+2 u_3+u_4-2 w_1+w_2-w_4}\right) \iota({P}_{B_2}).
\end{split}
\end{equation}
In the same manner, \eqref{K134r} is reduced to the image of $\iota$ of ${K}_{B_2:-++}$ \eqref{K-++B2} in the limit \eqref{lim13}.

\end{proof}

\begin{remark} 
In \cite{IKSTY}, the $R$-operators $R^{(\pm)}_{123}$ \eqref{R--++}, \eqref{R-+-+} are shown to be reduced to that for the FG-quiver.
We reformulate this reduction more clearly in Appendix \ref{app:R}.
\end{remark}

\appendix

\section{Proof of Proposition \ref{prop:K24}}\label{app:pK}

We use the following lemma,
which can be proved easily.
\begin{lemma}\label{dilog-2}
For the $q$-commuting variables $X$ and $Y$ as $XY = q^2 YX$, it holds that
\begin{align}
\mathrm{Ad}(\Psi_q(X) \Psi_q(X^{-1}))(Y) = q^{-1} XY = q YX.
\end{align}
\end{lemma}

\begin{proof}[Proof of Proposition \ref{prop:K24}]
First we check that $K_{++}$ coincides with $K_{+-}$.
Due to the $q$-commutativity of $Y$-variables \eqref{q-Y} with \eqref{Bhat-K},
the dilogarithm parts of $K_{++}$ \eqref{sol241} and $K_{+-}$ \eqref{sol242} are respectively rewritten as 
\begin{align}
&\Psi _q\left(Y_8^{-1}\right)^{-1}
\Psi_{q^2}\left(Y_3\right)
\Psi_q\left(qY_3^{-1}Y_8^{-1}Y_9^{-1}\right)^{-1}
 \cdot \underline{\Psi_{q^2}\left(Y_2\right) \Psi_q\left(qY_7Y_8\right) \Psi _{q^2}\left(Y_2\right)^{-1}}   
   \nonumber \\
& \cdot \Psi_{q^2}\left(Y_3^{-2}Y_4^{-1}Y_8^{-2}Y_9^{-2}\right)^{-1}
\Psi_q\left(q^{-1}Y_3^{-1}Y_4^{-1}Y_8^{-1}Y_9^{-1}\right)^{-1}
   \Psi_q\left(q^{2}Y_2Y_3Y_7Y_8Y_9\right)
   \Psi_{q^2}\left(Y_3\right)^{-1},
\label{K++d}
\\
&\Psi _q\left(Y_8^{-1}\right)^{-1}
\Psi_{q^2}\left(Y_3\right)
\Psi_q\left(qY_3^{-1}Y_8^{-1}Y_9^{-1}\right)^{-1}
 \cdot \underline{\Psi_{q^2}\left(Y_2^{-1}\right)^{-1} \Psi_q\left(q^{-1}Y_2Y_7Y_8\right) \Psi _{q^2}\left(Y_2^{-1}\right)}   
  \nonumber \\
& \cdot \Psi_{q^2}\left(Y_3^{-2}Y_4^{-1}Y_8^{-2}Y_9^{-2}\right)^{-1}
\Psi_q\left(q^{-1}Y_3^{-1}Y_4^{-1}Y_8^{-1}Y_9^{-1}\right)^{-1}
   \Psi_q\left(q^{2}Y_2Y_3Y_7Y_8Y_9\right)
   \Psi_{q^2}\left(Y_3\right)^{-1}.
\label{K+-d}
\end{align}
These are identical except for the underlined parts.  By applying
Lemma~\ref{dilog-2}, the underlined part of \eqref{K+-d} can be rewritten
as follows:
\begin{align*}
\Psi_{q^2}\left(Y_2^{-1}\right)^{-1} &\Psi_q\left(q^{-1} Y_2Y_7Y_8\right) \Psi _{q^2}\left(Y_2^{-1}\right)
\\
&= \Psi_{q^2}\left(Y_2^{-1}\right)^{-1} \cdot \mathrm{Ad}(\Psi_{q^2}\left(Y_2 \right) \Psi_{q^2}\left(Y_2^{-1}\right))(\Psi_q\left(qY_7Y_8\right)) \cdot \Psi _{q^2}\left(Y_2^{-1}\right)
\\
&=
\Psi_{q^2}\left(Y_2\right) \Psi_q\left(qY_7Y_8\right) \Psi _{q^2}\left(Y_2\right)^{-1}.
\end{align*} 
The last line coincides with the underlined part of \eqref{K++d}.
In a similar manner, one proves that $K_{-+}$ \eqref{sol243} coincides with $K_{--}$ \eqref{sol244}.

As the last step we show that $K_{+-}$ \eqref{sol242} coincides with $K_{-+}$ \eqref{sol243}.
Using \eqref{Bhat-K}, we rewrite the underlined part in \eqref{sol243} as follows:
\begin{align*}
&\Psi_{q^2}\left(Y_3^{-1}\right)^{-1}
   \Psi_q\left(q^{-1}Y_8^{-1}Y_9^{-1}\right)^{-1}
   \Psi_{q^2}\left(q^{2}Y_2Y_3\right)
   \\
& \qquad \cdot  \Psi_q\left(qY_7Y_8\right)
   \Psi_{q^2}\left(q^{-2}Y_3^{-1}Y_4^{-1}Y_8^{-2}Y_9^{-2}\right)^{-1}
   \Psi_{q^2}\left(q^{2}Y_2Y_3\right)^{-1} \Psi_{q^2}\left(Y_3^{-1}\right)
\\
&=\Psi_{q^2}\left(Y_3^{-1}\right)^{-1}
\cdot 
\mathrm{Ad}(\Psi_{q^2}\left(Y_3\right) \Psi_{q^2}\left(Y_3^{-1}\right)) 
(\Psi_q\left(q Y_3^{-1} Y_8^{-1}Y_9^{-1}\right)^{-1}
 \Psi_{q^2}\left(Y_2\right) 
\\
&\hspace{4cm}\Psi_q\left(qY_7Y_8\right)
   \Psi_{q^2}\left(Y_3^{-2}Y_4^{-1}Y_8^{-2}Y_9^{-2}\right)^{-1}
   \Psi_{q^2}\left(Y_2\right)^{-1})
\cdot \Psi_{q^2}\left(Y_3^{-1}\right)
\\
&= \Psi_{q^2}\left(Y_3\right) \Psi_{q^2}\left(Y_2^{-1}\right)^{-1}
\\
& \qquad 
\cdot \mathrm{Ad}(\Psi_{q^2}\left(Y_2\right) \Psi_{q^2}\left(Y_2^{-1}\right))
 (  \Psi_q\left(q Y_3^{-1} Y_8^{-1}Y_9^{-1}\right)^{-1}
\Psi_q\left(qY_7Y_8\right)
   \Psi_{q^2}\left(Y_3^{-2}Y_4^{-1}Y_8^{-2}Y_9^{-2}\right)^{-1})
\\
& \qquad 
   \cdot \Psi_{q^2}\left(Y_2^{-1}\right) \Psi_{q^2}\left(Y_3\right)^{-1}
\\
&= \Psi_{q^2}\left(Y_3\right) \Psi_{q^2}\left(Y_2^{-1}\right)^{-1}
 \Psi_q\left(q Y_3^{-1} Y_8^{-1}Y_9^{-1}\right)^{-1}
\Psi_q\left(q^{-1}Y_2Y_7Y_8\right)
   \Psi_{q^2}\left(Y_3^{-2}Y_4^{-1}Y_8^{-2}Y_9^{-2}\right)^{-1}
\\
& \qquad 
   \cdot \Psi_{q^2}\left(Y_2^{-1}\right) \Psi_{q^2}\left(Y_3\right)^{-1}.
\end{align*}
This coincides with the underlined part in \eqref{sol242} using \eqref{Bhat-K}. 
\end{proof}

\section{Explicit formulas for $K_{1234|\ve}$}
\label{app2}

\subsection{Type $\rho_{24}$}

We present the explicit formulas for $K_{\ve_2,\ve_4}$ \eqref{sol241}--\eqref{sol244} in terms of canonical variables.
We eliminate  $e_1,e_2, e_3, e_4$ from (\ref{econ}) and  (\ref{ccon}).

\begin{equation}\label{K241}
\begin{split}
\ve=(&-1, 1, -1, 1, 1, -1, -1, -1, 1, -1);
\\
K_{++}=& ~\Psi _q\left(\e^{-b_1-c_2-d_3+u_1+u_3+w_1-w_2+w_3}\right)^{-1}
\\ & \cdot
  \Psi_{q^2}\left(\e^{a_3+b_2+d_4-u_2-u_4-w_2+2 w_3-w_4}\right)
  \\ & \cdot 
\Psi_q\left(\e^{-b_1-b_2+b_3-c_2+c_3-d_4+u_1+u_2-u_3+u_4+w_1-w_3+w_4}\right)^{-1}\\ & \cdot 
  \Psi_{q^2}\left(\e^{-b_2+c_1-2 c_2+c_3-d_2+2 u_2}\right) 
\\ & \cdot \Psi _q\left(\e^{-a_1-c_1+c_2-d_1+d_3+u_1-u_3-w_1+w_2-w_3}\right)
\\ & \cdot \Psi_{q^2}\left(\e^{-2 b_1-2 b_2+2 b_3+b_4+c_1-2 c_2+c_3+2 c_4-d_4+2 u_1+2 u_2-2 u_3+2w_1-2 w_3+2 w_4}\right)^{-1}
\\ & \cdot \Psi_q\left(\e^{-b_1-b_2+b_3+b_4+c_1-c_2+2c_4+u_1+u_2-u_3-u_4+w_1-w_3+w_4}\right)^{-1}
\\ & \cdot \Psi _{q^2}\left(\e^{-b_2+c_1-2c_2+c_3-d_2+2 u_2}\right)^{-1}
\\ & \cdot \Psi_q\left(\e^{-a_1-b_3-c_2-d_1-d_2+d_4+u_1+u_2+u_3-u_4-w_1+w_3-w_4}\right)
\\ & \cdot \Psi_{q^2}\left(\e^{a_3+b_2+d_4-u_2-u_4-w_2+2 w_3-w_4}\right)^{-1}P_{24}.
    \end{split}
\end{equation}

\begin{equation}
\begin{split}
\ve =(& -1, 1, -1, -1, 1, -1, -1, 1, 1, -1);
\\
K_{+-}= & ~\Psi _q\left(\e^{-b_1-c_2-d_3+u_1+u_3+w_1-w_2+w_3}\right)^{-1}
\\ & \cdot 
\Psi_{q^2}\left(\e^{a_3+b_2+d_4-u_2-u_4-w_2+2 w_3-w_4}\right)
\\ & \cdot 
\Psi_q\left(\e^{-b_1-b_2+b_3-c_2+c_3-d_4+u_1+u_2-u_3+u_4+w_1-w_3+w_4}\right)^{-1}
\\ & \cdot 
\Psi_{q^2}\left(\e^{b_2-c_1+2 c_2-c_3+d_2-2 u_2}\right)^{-1}
\\ & \cdot 
\Psi _q\left(\e^{-a_1-b_2-c_2+c_3-d_1-d_2+d_3+u_1+2u_2-u_3-w_1+w_2-w_3}\right)
\\ & \cdot 
\Psi _{q^2}\left(\e^{-2 b_1-2 b_2+2 b_3+b_4+c_1-2
   c_2+c_3+2 c_4-d_4+2 u_1+2 u_2-2 u_3+2 w_1-2 w_3+2 w_4}\right)^{-1}
\\ & \cdot 
\Psi_q\left(\e^{-b_1-b_2+b_3+b_4+c_1-c_2+2c_4+u_1+u_2-u_3-u_4+w_1-w_3+w_4}\right)^{-1}
\\ & \cdot 
\Psi_{q^2}\left(\e^{b_2-c_1+2 c_2-c_3+d_2-2 u_2}\right)
\\ & \cdot \Psi_q\left(\e^{-a_1-b_3-c_2-d_1-d_2+d_4+u_1+u_2+u_3-u_4-w_1+w_3-w_4}\right)
\\ & \cdot
\Psi_{q^2}\left(\e^{a_3+b_2+d_4-u_2-u_4-w_2+2 w_3-w_4}\right)^{-1}P_{24}.
    \end{split}
\end{equation}

\begin{equation}\label{K243}
\begin{split}
\ve=(&-1, -1, -1, 1, 1, -1, -1, -1, 1, 1);
\\
K_{-+}&=~\Psi _q\left(\e^{-b_1-c_2-d_3+u_1+u_3+w_1-w_2+w_3}\right)^{-1}
\\ & \cdot 
\Psi_{q^2}\left(\e^{-a_3-b_2-d_4+u_2+u_4+w_2-2 w_3+w_4}\right)^{-1}
\\ & \cdot \Psi_q\left(\e^{a_3-b_1+b_3-c_2+c_3+u_1-u_3+w_1-w_2+w_3}\right)^{-1}
\\ & \cdot \Psi_{q^2}\left(\e^{a_3+c_1-2 c_2+c_3-d_2+d_4+u_2-u_4-w_2+2 w_3-w_4}\right)
\\ & \cdot \Psi_q\left(\e^{-a_1-c_1+c_2-d_1+d_3+u_1-u_3-w_1+w_2-w_3}\right)
\\ & \cdot \Psi_{q^2}\left(\e^{a_3-2 b_1-b_2+2 b_3+b_4+c_1-2 c_2+c_3+2 c_4+2 u_1+u_2-2 u_3-u_4+2w_1-w_2+w_4}\right)^{-1}
\\ & \cdot \Psi_q\left(\e^{-b_1-b_2+b_3+b_4+c_1-c_2+2c_4+u_1+u_2-u_3-u_4+w_1-w_3+w_4}\right)^{-1}
\\ & \cdot \Psi _{q^2}\left(\e^{a_3+c_1-2c_2+c_3-d_2+d_4+u_2-u_4-w_2+2 w_3-w_4}\right)^{-1}
\\ & \cdot \Psi_q\left(\e^{-a_1-b_3-c_2-d_1-d_2+d_4+u_1+u_2+u_3-u_4-w_1+w_3-w_4}\right)
\\ & \cdot \Psi_{q^2}\left(\e^{-a_3-b_2-d_4+u_2+u_4+w_2-2 w_3+w_4}\right)P_{24}.
    \end{split}
\end{equation}

\begin{equation}\label{K244}
\begin{split}
\ve=(&-1, -1, -1, -1, 1, -1, -1, 1, 1, 1);
\\
K_{--}=& ~\Psi _q\left(\e^{-b_1-c_2-d_3+u_1+u_3+w_1-w_2+w_3}\right)^{-1}
\\ & \cdot \Psi_{q^2}\left(\e^{-a_3-b_2-d_4+u_2+u_4+w_2-2 w_3+w_4}\right)^{-1}
\\ & \cdot \Psi_q\left(\e^{a_3-b_1+b_3-c_2+c_3+u_1-u_3+w_1-w_2+w_3}\right)^{-1}
\\ & \cdot \Psi_{q^2}\left(\e^{-a_3-c_1+2 c_2-c_3+d_2-d_4-u_2+u_4+w_2-2
   w_3+w_4}\right)^{-1}
\\ & \cdot \Psi_q\left(\e^{-a_1+a_3-c_2+c_3-d_1-d_2+d_3+d_4+u_1+u_2-u_3-u_4-w_1+w_3-w_4}\right)
\\ & \cdot \Psi_{q^2}\left(\e^{a_3-2 b_1-b_2+2 b_3+b_4+c_1-2 c_2+c_3+2 c_4+2 u_1+u_2-2 u_3-u_4+2w_1-w_2+w_4}\right)^{-1}
\\ & \cdot \Psi_q\left(\e^{-b_1-b_2+b_3+b_4+c_1-c_2+2c_4+u_1+u_2-u_3-u_4+w_1-w_3+w_4}\right)^{-1}
\\ & \cdot \Psi _{q^2}\left(\e^{-a_3-c_1+2c_2-c_3+d_2-d_4-u_2+u_4+w_2-2 w_3+w_4}\right)
\\ & \cdot \Psi_q\left(\e^{-a_1-b_3-c_2-d_1-d_2+d_4+u_1+u_2+u_3-u_4-w_1+w_3-w_4}\right)
\\ & \cdot 
\Psi_{q^2}\left(\e^{-a_3-b_2-d_4+u_2+u_4+w_2-2 w_3+w_4}\right)P_{24}.
    \end{split}
\end{equation}

\subsection{Type $\rho_{13}$}
We present the explicit formulae for $K_{\ve_1,\ve_3}$ \eqref{sol131}--\eqref{sol134} in terms of canonical variables.
We eliminate $e_1,e_2, e_3,e_4$ from (\ref{econ}) and (\ref{ccon}).

\begin{equation}\label{K131}
\begin{split}
\ve=(&1, -1, 1, -1, -1, 1, -1, -1, -1, 1);
\\
K_{++}=&~\Psi _q\left(\e^{b_1+c_2+d_3-u_1-u_3-w_1+w_2-w_3}\right)
\\ & \cdot \Psi_{q^2}\left(\e^{-a_3-b_2-d_4+u_2+u_4+w_2-2 w_3+w_4}\right)^{-1}
\\ & \cdot \Psi _q\left(\e^{-a_3-b_3-c_3-d_3+2u_3}\right)
\\ & \cdot \Psi_{q^2}\left(\e^{-a_3-2 b_1-c_1-c_3+d_2-2 d_3-d_4+2 u_1-u_2+2 u_3+u_4+2w_1-w_2+w_4}\right)^{-1}
\\ & \cdot \Psi _q\left(\e^{a_1-a_3-b_1-c_3+d_1+d_2-2 d_3-d_4-u_2+2 u_3+u_4+2w_1-w_2+w_4}\right)^{-1}
\\ & \cdot \Psi _{q^2}\left(\e^{a_3+b_2-b_4-c_1+c_3-2c_4-u_2+u_4-w_2+2 w_3-w_4}\right)
\\ & \cdot \Psi_q\left(\e^{-a_3-b_3-c_3-d_3+2 u_3}\right)^{-1}
\\ & \cdot \Psi _{q^2}\left(\e^{2 a_1-a_3+c_1-c_3+2 d_1+d_2-2d_3-d_4-2 u_1-u_2+2 u_3+u_4+2 w_1-w_2+w_4}\right)^{-1}
\\ & \cdot \Psi_q\left(\e^{b_1+c_2+d_3-u_1-u_3-w_1+w_2-w_3}\right)^{-1}
\\ & \cdot \Psi_{q^2}\left(\e^{-a_3+2 b_1-2 b_3-b_4-2 c_4-d_2-2 u_1+u_2+2 u_3+u_4-2w_1+w_2-w_4}\right)P_{13}.
\end{split}
\end{equation}

\begin{equation}
\begin{split}
\ve=(&1, -1, -1, -1, -1, 1, 1, -1, -1, 1);
\\
K_{+-}=&~\Psi _q\left(\e^{b_1+c_2+d_3-u_1-u_3-w_1+w_2-w_3}\right)
\\ & \cdot \Psi_{q^2}\left(\e^{-a_3-b_2-d_4+u_2+u_4+w_2-2 w_3+w_4}\right)^{-1}
\\ & \cdot \Psi _q\left(\e^{a_3+b_3+c_3+d_3-2u_3}\right)^{-1}
\\ & \cdot \Psi_{q^2}\left(\e^{-a_3-2 b_1-c_1-c_3+d_2-2 d_3-d_4+2 u_1-u_2+2 u_3+u_4+2w_1-w_2+w_4}\right)^{-1}
\\ & \cdot \Psi _q\left(\e^{a_1-a_3-b_1-c_3+d_1+d_2-2d_3-d_4-u_2+2 u_3+u_4+2 w_1-w_2+w_4}\right)^{-1}
\\ & \cdot \Psi_{q^2}\left(\e^{-a_3+b_2-2 b_3-b_4-c_1-c_3-2 c_4-2 d_3-u_2+4 u_3+u_4-w_2+2w_3-w_4}\right)
\\ & \cdot \Psi_q\left(\e^{a_3+b_3+c_3+d_3-2 u_3}\right)
\\ & \cdot \Psi _{q^2}\left(\e^{2 a_1-a_3+c_1-c_3+2 d_1+d_2-2 d_3-d_4-2u_1-u_2+2 u_3+u_4+2 w_1-w_2+w_4}\right)^{-1}
\\ & \cdot \Psi_q\left(\e^{b_1+c_2+d_3-u_1-u_3-w_1+w_2-w_3}\right)^{-1}
\\ & \cdot \Psi_{q^2}\left(\e^{-a_3+2 b_1-2 b_3-b_4-2 c_4-d_2-2 u_1+u_2+2 u_3+u_4-2w_1+w_2-w_4}\right)P_{13}.
    \end{split}
\end{equation}

\begin{equation}\label{K133}
\begin{split}
\ve=(&-1, -1, 1, -1, -1, 1, -1, -1, 1, 1);
\\
K_{-+}=&~\Psi _q\left(\e^{-b_1-c_2-d_3+u_1+u_3+w_1-w_2+w_3}\right)^{-1}
\\ & \cdot \Psi_{q^2}\left(\e^{-a_3-b_2-d_4+u_2+u_4+w_2-2 w_3+w_4}\right)^{-1}
\\ & \cdot \Psi_q\left(\e^{-a_3+b_1-b_3+c_2-c_3-u_1+u_3-w_1+w_2-w_3}\right)
\\ & \cdot \Psi_{q^2}\left(\e^{-a_3-c_1+2 c_2-c_3+d_2-d_4-u_2+u_4+w_2-2
   w_3+w_4}\right)^{-1}
\\ & \cdot \Psi_q\left(\e^{a_1-a_3+c_2-c_3+d_1+d_2-d_3-d_4-u_1-u_2+u_3+u_4+w_1-w_3+w_4}\right)^{-1}
\\ & \cdot \Psi _{q^2}\left(\e^{a_3+b_2-b_4-c_1+c_3-2 c_4-u_2+u_4-w_2+2w_3-w_4}\right)
\\ & \cdot \Psi_q\left(\e^{-a_3+b_1-b_3+c_2-c_3-u_1+u_3-w_1+w_2-w_3}\right)^{-1}
\\ & \cdot \Psi _{q^2}\left(\e^{2 a_1-a_3+c_1-c_3+2 d_1+d_2-2 d_3-d_4-2u_1-u_2+2 u_3+u_4+2 w_1-w_2+w_4}\right)^{-1}
\\ & \cdot \Psi_q\left(\e^{-b_1-c_2-d_3+u_1+u_3+w_1-w_2+w_3}\right)
\\ & \cdot \Psi _{q^2}\left(\e^{-a_3+2 b_1-2b_3-b_4-2 c_4-d_2-2 u_1+u_2+2 u_3+u_4-2 w_1+w_2-w_4}\right)P_{13}.
   \end{split}
\end{equation}
  
  \begin{equation}\label{K134}
\begin{split}
\ve=(& -1, -1, -1, -1, -1, 1, 1, -1, 1, 1);
   \\
K_{--}=&~\Psi _q\left(\e^{-b_1-c_2-d_3+u_1+u_3+w_1-w_2+w_3}\right)^{-1}
\\ & \cdot \Psi_{q^2}\left(\e^{-a_3-b_2-d_4+u_2+u_4+w_2-2 w_3+w_4}\right)^{-1}
\\ & \cdot \Psi_q\left(\e^{a_3-b_1+b_3-c_2+c_3+u_1-u_3+w_1-w_2+w_3}\right)^{-1}
\\ & \cdot \Psi_{q^2}\left(\e^{-a_3-c_1+2 c_2-c_3+d_2-d_4-u_2+u_4+w_2-2w_3+w_4}\right)^{-1}
\\ & \cdot \Psi_q\left(\e^{a_1-a_3+c_2-c_3+d_1+d_2-d_3-d_4-u_1-u_2+u_3+u_4+w_1-w_3+w_4}\right)^{-1}
\\ & \cdot \Psi _{q^2}\left(\e^{-a_3+2 b_1+b_2-2 b_3-b_4-c_1+2 c_2-c_3-2 c_4-2u_1-u_2+2 u_3+u_4-2 w_1+w_2-w_4}\right)
\\ & \cdot \Psi_q\left(\e^{a_3-b_1+b_3-c_2+c_3+u_1-u_3+w_1-w_2+w_3}\right)
\\ & \cdot \Psi _{q^2}\left(\e^{2 a_1-a_3+c_1-c_3+2d_1+d_2-2 d_3-d_4-2 u_1-u_2+2 u_3+u_4+2 w_1-w_2+w_4}\right)^{-1}
\\ & \cdot \Psi_q\left(\e^{-b_1-c_2-d_3+u_1+u_3+w_1-w_2+w_3}\right)
\\ & \cdot \Psi _{q^2}\left(\e^{-a_3+2 b_1-2b_3-b_4-2 c_4-d_2-2 u_1+u_2+2 u_3+u_4-2 w_1+w_2-w_4}\right)P_{13}.
   \end{split}
\end{equation}

\section{Well-definedness of $(F^\Psi_R)^{-1} F^\Psi_L$}\label{app:ff}

\subsection{$(F^\Psi_R)^{-1} F^\Psi_L$ in $Y$-variables}\label{dilog-p3}
 
Using the summation indices $n_i$ $(i=1,\ldots,92)$ for the
$i$th quantum dilogarithm appearing from the left in
$(F^\Psi_R)^{-1}F^\Psi_L$, the expression can be expanded in the
form~\eqref{CYp}, with
\begin{align}
p_{1} &= n_1+n_{19}+n_{20}+n_{24}+n_{25}+n_{29}+n_{54}+n_{58}+n_{59}+n_{63}+n_{64}+n_{82}+n_{86}
\nonumber \\
& \quad +n_{87}+n_{91}+n_{92},\nonumber  \displaybreak[0] \\ 
p_{2} &= n_{19}+n_{20}-n_{21}-n_{22}-n_{23}+n_{24}+n_{25}-n_{26}-n_{27}+n_{29}-n_{31}-n_{52}+n_{54}
\nonumber \\
&  \quad -n_{56}-n_{57}+n_{5
   8}+n_{59}-n_{60}-n_{61}-n_{62}+n_{63}+n_{64}-n_{80}+n_{82}-n_{84}-n_{85}
   \nonumber \\
   & \quad +n_{86}+n_{87}-n_{88}-n_{89}-n_{90}+n_{91}+n
   _{92},\nonumber  \displaybreak[0]  \\ 
   p_{3} &= -n_1+n_4+n_6+n_7+n_{11}-n_{21}-n_{22}-n_{26}-n_{27}+n_{34}+n_{35}+n_{39}-n_{56}-n_{57}
   \nonumber \\
   & \quad -n_{61}-n_{62}+n_{72}+n_{76}+n_{77}-n_{84}-n_{85}-n_{89}-n_{90},\nonumber  \displaybreak[0] \\ 
   p_{4} &= -n_5+n_6+n_7-n_8-n_9+n_{11}-n_{13}-n_{33}+n_{34}+n_{35}-n_{36}-n_{37}+n_{39}-n_{41}
   \nonumber \\
   & \quad -n_{70}+n_{72}-n_{74}-n_{75}+n_{76}+n_{77}-n_{78},\nonumber \\ 
   p_{5} &= -n_4-n_8-n_9-n_{36}-n_{37}-n_{74}-n_{75},\nonumber  \displaybreak[0] \\ 
   p_{6} &= n_1+n_{19}+n_{20}+n_{24}+n_{28}+n_{50}+n_{55}+n_{59}+n_{63}+n_{64}+n_{83}+n_{87}+n_{91}+n_{92},\nonumber  \displaybreak[0] \\ 
   p_{7} &= n_{19}+n_{20}-
   n_{21}-n_{22}+n_{24}-n_{26}-2 n_{27}+n_{28}-n_{30}-n_{32}-n_{51}-n_{53}+n_{55}
   \nonumber  \displaybreak[0] \\
   & \quad -2n_{56}-n_{57}+n_{59}-n_{61}-n_{62}+n_{63}+n_{64}-n_{79}-n_{81}+n_{83}-2n_{84}-n_{85}+n_{87}
   \nonumber \\
   & \quad -n_{89}-n_{90}+n_{91}+n_{92},\nonumber  \displaybreak[0] \\ 
   p_{8} &= n_4+n_6+n_{10}+n_{15}+n_{19}+n_{20}-n_{21}-n_{22}+n_{24}-n_{26}-2 n_{27}-n_{30}+n_{34}+n_{38}
   \nonumber \\
   & \quad +n_{43}-n_{53}-2n_{56}-n_{57}+n_{59}-n_{61}-n_{62}+n_{63}+n_{64}+n_{68}+n_{73}+n_{77}-n_{81}
   \nonumber \\
   & \quad -2n_{84}-n_{85}+n_{87}-n_{89}-n_{90}+n_{91}+n_{92},\nonumber  \displaybreak[0] \\ 
   p_{9} &= n_6-n_8-2
   n_9+n_{10}-n_{12}-n_{14}+n_{19}+n_{20}-n_{21}-n_{22}+n_{34}-n_{36}-2
   n_{37}+n_{38}
   \nonumber \\
   & \quad -n_{40}-n_{42}-n_{61}-n_{62}+n_{63}+n_{64}-n_{69}-n_{71}+n_{73}-2n_{74}-n_{75}+n_{77}-n_{89}-n_{90}
   \nonumber \\
   & \quad +n_{91}+n_{92},\nonumber  \displaybreak[0] \\ 
   p_{10} &= -n_1+n_6-n_8-2 n_9-n_{12}-n_{21}+n_{34}-n_{36}-2
   n_{37}-n_{40}-n_{62}-n_{71}-2n_{74}
   \nonumber \\
   & \quad -n_{75}+n_{77}-n_{90},\nonumber \\
    p_{11} &= n_1+n_2+n_6+n_{10}+n_{15}+n_{17}+n_{48}+n_{50}+n_{55}+n_{59}+n_{63}+n_{64},\nonumber \\ 
    p_{12} &= n_1+n_2-n_3-n_4+n_6-n_8-2
   n_9+n_{10}-n_{12}-n_{14}+n_{15}-n_{16}+n_{17}-n_{18}
   \nonumber  \displaybreak[0] \\
   & \quad -n_{47}+n_{48}-n_{49}+n_{50}-n_{51}-n_{53}+n_{55}-2
   n_{56}-n_{57}+n_{59}-n_{61}-n_{62}+n_{63}+n_{64},\nonumber \\ 
   p_{13} &= -n_3+n_6-n_8-2
   n_9+n_{10}-n_{12}-n_{14}+n_{15}-n_{16}+n_{17}-n_{18}+n_{34}+n_{38}+n_{43}
   \nonumber \\
   & \quad +n_{45}-n_{49}-n_{51}-n_{53}-2
   n_{56}-n_{57}-n_{61}-n_{62}+n_{66}+n_{68}+n_{73}+n_{77},\nonumber  \displaybreak[0] \\ 
   p_{14} &= n_6-n_8-2
   n_9+n_{10}-n_{12}-n_{14}+n_{15}-n_{16}+n_{17}-n_{18}+n_{34}-n_{36}-2n_{37}+n_{38}
   \nonumber \\
   & \quad -n_{40}-n_{42}+n_{43}-n_{44}+n_{45}-n_{46}-n_{65}+n_{66}-n_{67}+n_{68}-n_{69}-n_{71}+n_{73}
   \nonumber \\
   & \quad -2n_{74}-n_{75}+n_{77},\nonumber  \displaybreak[0] \\ 
   p_{15} &= n_1-n_4-n_8-2 n_9-n_{12}-n_{14}-n_{16}+n_{19}-n_{36}-2
   n_{37}-n_{40}-n_{42}-n_{44}+n_{64}
   \nonumber \\
   & \quad -n_{67}-n_{69}-n_{71}-2 n_{74}-n_{75}+n_{92}.
   \label{FFY}
\end{align}
In the notation \eqref{ppnL}, the well-definedness of 
$(F^\Psi_R)^{-1}F^\Psi_L$ as a formal Laurent series in the
$Y$-variables is verified along the following procedure:
\begin{equation}\label{ppnRL}
\begin{split}
p&: 1,5,6,11,\\
n&: 1,2,4,6,8,9,10,15,17,19,20,24,25,28,29,36,37,48,50,54,55,58,59,63,64,
\\ &\quad 74,75,82,83,86,87,91,92,
\\
p&: 2,7,12,15,\\
n&:3,12,14,16,18,21,22,23,26,27,30,31,32,40,42,44,47,49,51,52,53,56,57,\\
& \quad 60,61,62,67,69,71,79,80,81,84,85,88,89,90,
\\
p&: 3,8,9,10,13,\\
n&: 7,11,34,35,38,39,43,45,66,68,72,73,76,77,
\\
p&: 4,14,\\
n&: 5,13,33,41,46,65,70,78.
\end{split}
\end{equation}

\subsection{$(F^\Psi_R)^{-1} F^\Psi_L$ in $q$-Weyl variables}\label{dilog-p4}
 
 Using the summation indices $n_i$ $(i=1,\ldots,92)$ for the
$i$th quantum dilogarithm appearing from the left in
$(F^\Psi_R)^{-1}F^\Psi_L$, and making the substitution \eqref{yeuw}, the expression can be expanded in the
form~\eqref{anL}, with
\begin{align}
\nonumber \\ \alpha_{1} &= n_1+n_2+n_3+n_4+n_6+n_8+2
   n_9+n_{10}+n_{12}+n_{14}+n_{15}+n_{16}+n_{17}+n_{18}
   \nonumber \\
   & \quad +n_{47}+n_{48}+n_{49}+n_{50}+n_{51}+n_{53}+n_{55}+2
   n_{56}+n_{57}+n_{59}+n_{61}+n_{62}+n_{63}+n_{64},\nonumber  \displaybreak[0]  \\ 
   \alpha_{2} &= 2 n_1+n_{19}+n_{20}+n_{21}+n_{22}+n_{24}+n_{26}+2
   n_{27}+n_{28}+n_{30}+n_{32}+2 n_{50}+n_{51}
   \nonumber \\
   & \quad +n_{53}+n_{55}+2
   n_{56}+n_{57}+n_{59}+n_{61}+n_{62}+n_{63}+n_{64}+n_{79}+n_{81}+n_{83}+
   2n_{84}
   \nonumber \\
   & \quad +n_{85}+n_{87}+n_{89}+n_{90}+n_{91}+n_{92},\nonumber  \displaybreak[0] \\ 
   \alpha_{3} &= 2
   n_1+n_{19}+n_{20}+n_{21}+n_{22}+n_{23}+n_{24}+n_{25}+n_{26}+n_{27}+n_{29}+n_{31}+n_{52}
   \nonumber \\
   & \quad +n_{54}+n_{56}+n_{57}+n_{58}+
   n_{59}+n_{60}+n_{61}+n_{62}+n_{63}+n_{64}+n_{80}+n_{82}+n_{84}
   \nonumber \\
   & \quad +n_{85}+n_{86}+n_{87}+n_{88}+n_{89}+n_{90}+n_{91}+n_{9
   2},\nonumber  \displaybreak[0] \\ 
   \alpha_{4} &= -n_1-n_2-n_3+n_4+n_{34}+n_{36}+2
   n_{37}+n_{38}+n_{40}+n_{42}+n_{43}+n_{44}+n_{45}+n_{46}
   \nonumber \\
   & \quad +n_{47}-n_{48}-n_{49}-n_{50}-n_{51}-n_{53}-n_{55}-2
   n_{56}-n_{57}-n_{59}-n_{61}-n_{62}-n_{63}-n_{64}
   \nonumber \\
   & \quad +n_{65}+n_{66}+n_{67}+n_{68}+n_{69}+n_{71}+n_{73}+2
   n_{74}+n_{75}+n_{77},\nonumber  \displaybreak[0] \\ 
   \alpha_{5} &= 2 n_4+n_6+n_8+2 n_9+n_{10}+n_{12}+n_{14}+2 n_{15}+n_{24}-n_{26}-2
   n_{27}-n_{28}-n_{30}+n_{32}
   \nonumber \\
   & \quad +n_{34}+n_{36}+2 n_{37}+n_{38}+n_{40}+n_{42}+2 n_{43}+n_{51}-n_{53}-n_{55}-2
   n_{56}-n_{57}+n_{59}
   \nonumber \\
   & \quad +2 n_{68}+n_{69}+n_{71}+n_{73}+2 n_{74}+n_{75}+n_{77}+n_{79}-n_{81}-n_{83}-2
   n_{84}-n_{85}+n_{87},\nonumber \\ 
   \alpha_{6} &= -2 n_1+2
   n_4+n_5+n_6+n_7+n_8+n_9+n_{11}+n_{13}-n_{19}-n_{20}-n_{21}-n_{22}+n_{23}
   \nonumber \\
   & \quad -n_{24}-n_{25}-n_{26}-n_{27}-n_{29}+n_{31}+n_{33}+n_{34}+n_{35}+n_{36}+n_{37}+n_{39}+n_{41}+n_{52}
   \nonumber \\
   & \quad -n_{54}-n_{56}-n_{57}-n_{58}-n_{59}+n_{60}-n_{61}-n_{62}-n_{63}-n_{64}+n_{70}+n_{72}+n_{74}+n_{75}
   \nonumber \\
   & \quad +n_{76}+n_{77}+n_{78}+n_{80}-n_{82}-n_{84}-n_{85}-n_{86}-n_{87}+n_{88}-n_{89}-n_
   {90}-n_{91}-n_{92},\nonumber  \displaybreak[0] \\ 
   \alpha_{7} &= 2 n_1-2 n_4-n_6-n_8-2 n_9-n_{10}-n_{12}-n_{14}-n_{15}-n_{16}-n_{17}+n_{18}+2n_{19}
   \nonumber \\
   & \quad -n_{34}-n_{36}-2 n_{37}-n_{38}-n_{40}-n_{42}-n_{43}-n_{44}-n_{45}+n_{46}+2n_{64}+n_{65}-n_{66}
   \nonumber \\
   & \quad -n_{67}-n_{68}-n_{69}-n_{71}-n_{73}-2 n_{74}-n_{75}-n_{77}+2 n_{92},\nonumber  \displaybreak[0] \\ 
   \alpha_{8} &= -2 n_1+n_6-n_8-2
   n_9-n_{10}-n_{12}+n_{14}-n_{19}-n_{20}-n_{21}+n_{22}+n_{34}-n_{36}
   \nonumber \\
   & \quad -2n_{37}-n_{38}-n_{40}+n_{42}+n_{61}-n_{62}-n_{63}-n_{64}+n_{69}-n_{71}-n_{73}-2n_{74}-n_{75}
   \nonumber \\
   & \quad +n_{77}+n_{89}-n_{90}-n_{91}-n_{92},\nonumber  \displaybreak[0] \\ 
   \alpha_{9} &= -2
   n_4+n_5-n_6-n_7-n_8-n_9-n_{11}+n_{13}+n_{33}-n_{34}-n_{35}-n_{36}-n_{37}-n_{39}
   \nonumber \\
   & \quad +n_{41}+n_{70}-n_{72}-n_{74}-n_{75}-n
   _{76}-n_{77}+n_{78},\nonumber  \displaybreak[0] \\ 
   \alpha_{10} &= -n_1-n_2+n_3+n_4-n_6+n_8+2
   n_9-n_{10}+n_{12}+n_{14}-n_{15}+n_{16}-n_{17}+n_{18}
   \nonumber \\
   & \quad +n_{47}-n_{48}+n_{49}-n_{50}+n_{51}+n_{53}-n_{55}+2
   n_{56}+n_{57}-n_{59}+n_{61}+n_{62}-n_{63}-n_{64},\nonumber  \displaybreak[0] \\ 
   \alpha_{11} &= n_1+n_2-n_3-n_4+n_6-n_8-2
   n_9+n_{10}-n_{12}-n_{14}+n_{15}-n_{16}+n_{17}-n_{18}-n_{19}
   \nonumber \\
   & \quad -n_{20}+n_{21}+n_{22}-n_{24}+n_{26}+2n_{27}-n_{28}+n_{30}+n_{32}-n_{47}+n_{48}-n_{49}+n_{50}+n_{79}
   \nonumber \\
   & \quad +n_{81}-n_{83}+2
   n_{84}+n_{85}-n_{87}+n_{89}+n_{90}-n_{91}-n_{92},\nonumber  \displaybreak[0] \\ 
   \alpha_{12} &= n_{23}-n_{25}-n_{27}+n_{28}-n_{29}-n_{30}+n_{31}-n_{32}-n_{51}+n_{52}-n_{53}-n_{54}+n_{55}-n_{56}
   \nonumber \\
   & \quad -n_{58}+n_{60}-n_{79}+n_{80}-n_{81}-n_{82}+n_{83}-n_{84}-n_{86}+n_{88},\nonumber  \displaybreak[0] \\ 
   \alpha_{13} &= -n_
   1-n_2+n_3+n_4-2 n_6+2 n_8+4 n_9-2 n_{10}+2 n_{12}+2 n_{14}-2 n_{15}+2 n_{16}-2 n_{17}
   \nonumber \\
   & \quad +2n_{18}+n_{19}+n_{20}-n_{21}-n_{22}+n_{24}-n_{26}-2 n_{27}+n_{28}-n_{30}-n_{32}-n_{34}+n_{36}
   \nonumber \\
   & \quad +2n_{37}-n_{38}+n_{40}+n_{42}-n_{43}+n_{44}-n_{45}+n_{46}+n_{47}-n_{48}+n_{49}-n_{50}+n_{65}-n_{66}
   \nonumber \\
   & \quad +n_{67}-n_{68}+n_{69}+n_{71}-n_{73}+2 n_{74}+n_{75}-n_{77}-n_{79}-n_{81}+n_{83}-2
   n_{84}-n_{85}+n_{87}
   \nonumber \\
   & \quad -n_{89}-n_{90}+n_{91}+n_{92},\nonumber  \displaybreak[0] \\ 
   \alpha_{14} &= n_{15}-n_{16}+n_{17}-n_{18}-2 n_{23}+n_{24}+2
   n_{25}-n_{26}-n_{28}+2 n_{29}+n_{30}-2 n_{31}+n_{32}
   \nonumber \\
   & \quad +n_{43}-n_{44}+n_{45}-n_{46}+n_{51}-2 n_{52}+n_{53}+2
   n_{54}-n_{55}-n_{57}+2 n_{58}+n_{59}-2 n_{60}
   \nonumber \\
   & \quad -n_{65}+n_{66}-n_{67}+n_{68}+n_{79}-2 n_{80}+n_{81}+2
   n_{82}-n_{83}-n_{85}+2 n_{86}+n_{87}-2
   n_{88},\nonumber  \displaybreak[0] \\ 
   \alpha_{15} &= n_5-n_7-n_9+n_{10}-n_{11}-n_{12}+n_{13}-n_{14}+n_{23}-n_{24}-n_{25}+n_{26}+n_{27}-n_{29}
   \nonumber \\
   & \quad +n_{31}+n_{33}-n_{35}-n_{37}+n_{38}-n_{39}-n_{40}+n_{41}-n_{42}+n_{52}-n_{54}+n_{56}+n_{57}-n_{58}
   \nonumber \\
   & \quad -n_{59}+n_{60}-n_{69}+n_{70}-n_{7
   1}-n_{72}+n_{73}-n_{74}-n_{76}+n_{78}+n_{80}-n_{82}+n_{84}+n_{85}
   \nonumber \\
   & \quad -n_{86}-n_{87}+n_{88},\nonumber  \displaybreak[0] \\ 
   \alpha_{16} &= -n_{15}+n_{16}-n_{17}+n_{18}+n_{19}+n_{20}-n_{21}-n_{22}-n_{43}+n_{44}-n_{45}+n_{46}-n_{61}
   \nonumber \\
   & \quad -n_{62}+n_{63}+n_{64}+n_{65}-n_{66}+n_{67}-n_{68}-n_{89}-n_{90}+n_{91}+n_{92},\nonumber  \displaybreak[0] \\ 
   \alpha_{17} &= -2 n_5+n_6+2 n_7-n_8-n_{10}+2 n_{11}+n_{12}-2
   n_{13}+n_{14}-n_{19}-n_{20}+n_{21}+n_{22}
   \nonumber \\
   & \quad -2 n_{33}+n_{34}+2 n_{35}-n_{36}-n_{38}+2 n_{39}+n_{40}-2
   n_{41}+n_{42}+n_{61}+n_{62}-n_{63}-n_{64}
   \nonumber \\
   & \quad +n_{69}-2 n_{70}+n_{71}+2 n_{72}-n_{73}-n_{75}+2 n_{76}+n_{77}-2
   n_{78}+n_{89}+n_{90}-n_{91}-n_{92},\nonumber  \displaybreak[0] \\ 
   \alpha_{18} &= n_5-n_6-n_7+n_8+n_9-n_{11}+n_{13}+n_{33}-n_{34}-n_{35}+n_{36}+n_{37}-n_{39}
   +n_{41}
   \nonumber \\
   & \quad +n_{70}-n_{72}+n_{74}+n_{75}-n_{76}-n_{77}+n_{78}.
   \label{FFuw}
 \end{align}
 Similarly to \eqref{aalL} and \eqref{aalR},
 The well-definedness of 
$(F^\Psi_R)^{-1}F^\Psi_L$ as a formal Laurent series in 
$\e^{u_i}$ and $\e^{w_i}$ is shown in two steps as
\begin{equation}\label{FFuw2}
\begin{split}
\alpha&: 1,2,3,\\
n&:1,2,3,4,6,8,9,10,12,14,15,16,17,18,19,20,21,22,23,24,25,26,27,28,29,\\
&\quad 30,31,32,47,48,49,50,51,52,53,54,55,56,57,58,59,60,61,62,63,64,79,80,\\
&\quad 81,82,83,84,85,86,87,88,89,90,91,92,
\\
\alpha&: 4,5,6,\\
n&: 5,7,11,13,33,34,35,36,37,38,39,40,41,42,43,44,45,46,65,66,67,68,69,70,\\
&\quad 71,72,73,74,75,76,77,78.
\end{split}
\end{equation}

\section{Reduction to the $R$-operators for FG quiver}\label{app:R}

\subsection{$R$-operators for the FG quiver}

Recall the transformation $\qR_{123}$ of  the Fock-Goncharov (FG) quivers and the corresponding wiring diagrams:
\begin{align}
\label{quiver-d:A2}
\begin{tikzpicture}
\begin{scope}[>=latex,xshift=0pt]
{\color{red}
\fill (1,0.5) circle(2pt) coordinate(A) node[below]{$1$};
\fill (2,1.5) circle(2pt) coordinate(B) node[above]{$2$};
\fill (3,0.5) circle(2pt) coordinate(C) node[below]{$3$};
\draw [-] (0,2) to [out = 0, in = 135] (B) --(C) to [out = -45, in = 180] (4,0);
\draw [-] (0,1) to [out = 0, in = 135] (A) to [out = -45, in = -135] (C) to [out = 45, in = 180] (4,1);
\draw [-] (0,0) to [out = 0, in = -135] (A) -- (B) to [out = 45, in = 180] (4,2);
}
\coordinate (P1) at (4.5,1);
\coordinate (P2) at (5.5,1);
\draw[->] (4.5,1) -- (5.5,1);
\draw (5,1) circle(0pt) node[below]{$\qR_{123}$};
%
%
\draw (0,0.5) circle(2pt) coordinate(B1) node[below]{$3$};
\draw (2,0.5) circle(2pt) coordinate(C1) node[below]{$4$};
\draw (4,0.5) circle(2pt) coordinate(F1) node[below]{$5$};
\draw (1,1.5) circle(2pt) coordinate(D1) node[above]{$1$};
\draw (3,1.5) circle(2pt) coordinate(E1) node[above]{$2$};
\qarrow{B1}{C1}
\qarrow{C1}{D1}
\qarrow{D1}{E1}
\qarrow{E1}{C1}
\qarrow{C1}{F1}
\qdarrow{D1}{B1}
\qdarrow{F1}{E1}
\end{scope}
\begin{scope}[>=latex,xshift=175pt]
{\color{red}
\fill (3,1.5) circle(2pt) coordinate(A) node[above]{$1$};
\fill (2,0.5) circle(2pt) coordinate(B) node[below]{$2$};
\fill (1,1.5) circle(2pt) coordinate(C) node[above]{$3$};
\draw [-] (0,0) to [out = 0, in = -135] (B);
\draw [-] (B) -- (A); 
\draw [->] (A) to [out = 45, in = 180] (4,2);
\draw [-] (0,1) to [out = 0, in = -135] (C); 
\draw [-] (C) to [out = 45, in = 135] (A);
\draw [->] (A) to [out = -45, in = 180] (4,1);
\draw [-] (0,2) to [out = 0, in = 135] (C); 
\draw [-] (C) -- (B);
\draw [->] (B) to [out = -45, in = 180] (4,0);
}
%
%
\draw (0,1.5) circle(2pt) coordinate(B1) node[above]{$1$};
\draw (2,1.5) circle(2pt) coordinate(C1) node[above]{$4$};
\draw (4,1.5) circle(2pt) coordinate(D1) node[above]{$2$};
\draw (1,0.5) circle(2pt) coordinate(E1) node[below]{$3$};
\draw (3,0.5) circle(2pt) coordinate(F1) node[below]{$5$};
\qarrow{B1}{C1}
\qarrow{C1}{D1}
\qarrow{E1}{F1}
\qarrow{F1}{C1}
\qarrow{C1}{E1}
\qdarrow{E1}{B1}
\qdarrow{D1}{F1}
\end{scope}
\end{tikzpicture}
\end{align}
where the cluster transformation is given by $\qR_{123} = \mu_4$. 
The induced transformation $\widehat{R}_{123} = \mu_4^\ast$ of quantum $y$-variables is decomposed in two ways:
$$
  \mu_4^\ast = \Ad(\Phi_q(\mathcal{Y}_4)) \circ \tau_{4,+} = \Ad(\Phi_q(\mathcal{Y}_4^{-1})^{-1}) \circ \tau_{4,-}. 
$$

Recall the embeddings $\phi_\mathrm{FG}\colon \mathcal{Y}(B_\mathrm{FG}) \hookrightarrow
\mathrm{Frac}\,\mathcal{W}(A_2)$ and 
$\phi'_\mathrm{FG}\colon \mathcal{Y}(B'_\mathrm{FG}) \hookrightarrow
\mathrm{Frac}\,\mathcal{W}(A_2)$ involving the parameters $\theta_i \in \C~(i=1, 2, 3)$ \cite[eq.\ (3.6)]{IKT1} given by 
\begin{align}
&\phi_\mathrm{FG}:
\begin{cases}
\mathscr{Y}_1 \mapsto \e^{-\theta_2-w_2-{u}_2+w_1},
\\
\mathscr{Y}_2 \mapsto \e^{\theta_2-w_2+{u}_2+w_3},
\\
\mathscr{Y}_3 \mapsto  \e^{-\theta_1-w_1-{u}_1},
\\
\mathscr{Y}_4 \mapsto  \e^{\theta_1-\theta_3-w_1
+{u}_1-w_3-{u}_3+w_2},
\\
\mathscr{Y}_5 \mapsto  \e^{\theta_3-w_3+{u}_3},
\end{cases}
\qquad 
\phi'_\mathrm{FG}:
\begin{cases}
\mathscr{Y}'_1 \mapsto  \e^{-\theta_3 -w_3-{u}_3},
\\
\mathscr{Y}'_2 \mapsto  \e^{\theta_1-w_1+{u}_1},
\\
\mathscr{Y}'_3 \mapsto  \e^{-\theta_2-w_2-{u}_2+w_3},
\\
\mathscr{Y}'_4 \mapsto  \e^{-\theta_1+\theta_3-w_3
+{u}_3-w_1-{u}_1+w_2},
\\
\mathscr{Y}'_5 \mapsto  \e^{\theta_2-w_2+{u}_2+w_1}.
\end{cases}
\label{kapp}
\end{align}

Let $\pi_{+}$ denote the isomorphism $\pi_{123}$
of $\mathcal{W}(A_2)$ \cite[eq.\ (3.7)]{IKT1} (in the sense
of exponentials), and define another isomorphism $\pi_{-}$ as follows:
\begin{align}
&\pi_{+}:
\begin{cases}
w_1 \mapsto w_1 -  \theta_2+\theta_3,
\quad
w_2  \mapsto w_1+w_3,
\quad
w_3 \mapsto w_2-w_1+ \theta_2-\theta_3, 
\\
{u}_1  \mapsto {u}_1+{u}_2-{u}_3,
\quad 
{u}_2 \mapsto {u}_3,
\quad 
{u}_3  \mapsto {u}_2,
\end{cases}
\\
&\pi_{-}:
\begin{cases}
w_1 \mapsto w_2-w_3+ \theta_1-\theta_2, 
\quad
w_2  \mapsto w_1+w_3,
\quad
w_3 \mapsto w_3-  \theta_1+\theta_2,
\\
{u}_1 \mapsto {u}_2,
\quad 
{u}_2 \mapsto {u}_1,
\quad 
{u}_3 \mapsto -{u}_1+{u}_2+{u}_3. 
\end{cases}
\end{align}

We now consider two $R$-operators ${R}_+$ and ${R}_-$.
The operator ${R}_+$ is that introduced in \cite[eq.\ (4.14)]{IKT1}:
\begin{align}\label{r+fg}
{R}_+ = 
\Psi_q(\e^{\theta_1-\theta_3-w_1+{u}_1-w_3-{u}_3+w_2}) {P}_+,
\qquad 
{P}_{+} = \e^{\tfrac{1}{\hbar}w_1({u}_3-{u}_2)}
\e^{\tfrac{\theta_2-\theta_3}{\hbar}({u}_2-{u}_1)} \rho_{23},
\end{align}
and we introduce a new operator $R_-$, defined analogously  by
\begin{align}\label{r-fg}
{R}_- = 
\Psi_q(\e^{-\theta_1+\theta_3+w_1-{u}_1+w_3+{u}_3-w_2})^{-1} {P}_-,
\qquad 
{P}_- = \e^{\tfrac{1}{\hbar}w_3({u}_1-{u}_2)}
\e^{\tfrac{\theta_1-\theta_2}{\hbar}({u}_1-{u}_2)}\rho_{12}.
\end{align}

\begin{proposition} 
For $\ve\in\{+,-\}$, the following statements hold.\\
(i) The following diagram is commutative:	
\begin{align*}
\xymatrix{
\mathcal{Y}(B'_\mathrm{FG}(A_2)) \ar[r]^{\phi'_\mathrm{FG}} \ar[d]_{\tau_{4,\ve}} 
& \mathrm{Frac}\,\mathcal{W}(A_2) \ar[d]_{\pi_{\ve}}
\\
\mathcal{Y}(B_\mathrm{FG}(A_2)) \ar[r]^{\phi_\mathrm{FG}}& \mathrm{Frac}\,\mathcal{W}(A_2)
}
\end{align*}
(ii) The isomorphism $\pi_{\ve}$ is realized by the adjoint action $\Ad({P}_\ve)$. 
In particular, we have $\phi_\mathrm{FG} \circ \widehat{R}_{123} = \mathrm{Ad}({R}_\ve) \circ \phi'_\mathrm{FG}$.
\end{proposition}

The case $\ve = +$ was proved in \cite[Proposition 3.1]{IKT1}.
Note that the $R$-operators \eqref{r+fg} and \eqref{r-fg} may be written in the unified form
\begin{align}
{R}_\ve = 
\Psi_q(\phi_\mathrm{FG}(\mathscr{Y}_4^\ve))^\ve {P}_{\ve}~ \text{ for } \ve \in \{+,-\}. 
\end{align}

\subsection{Limit of $R$-operators}

Let $\alpha$ and $\alpha'$ be the ring homomorphisms of skewfields,
$\alpha\colon \mathcal{Y}(B_{\mathrm{FG}}) \rightarrow
\mathcal{Y}(B_{\mathrm{SB}})$ and
$\alpha'\colon \mathcal{Y}(B'_{\mathrm{FG}}) \rightarrow
\mathcal{Y}(B'_{\mathrm{SB}})$, given by
\begin{align}\label{Ymor}
\alpha: \begin{cases}
\mathscr{Y}_1 \mapsto Y_1,
\\
\mathscr{Y}_2 \mapsto q Y_2Y_3,
\\
\mathscr{Y}_3 \mapsto Y_4,
\\
\mathscr{Y}_4 \mapsto q Y_5Y_6,
\\
\mathscr{Y}_5 \mapsto qY_7Y_8,
\end{cases}
\qquad 
\alpha': \begin{cases}
\mathscr{Y}'_1 \mapsto Y'_1,
\\
\mathscr{Y}'_2 \mapsto q Y'_5Y'_3,
\\
\mathscr{Y}'_3 \mapsto Y'_4,
\\
\mathscr{Y}'_4 \mapsto q Y'_7Y'_6,
\\
\mathscr{Y}'_5 \mapsto qY'_2Y'_8.
\end{cases}
\end{align}

We consider diagrams
\begin{align}\label{cd3}
\xymatrix{
\mathcal{Y}(B_\mathrm{FG}(A_2)) \ar[d]_{\phi_\mathrm{FG}} \ar[r]^{\alpha} & \mathcal{Y}(B(A_2)) \ar[d]_{\phi}
\\
\mathrm{Frac}\,\mathcal{W}(A_2) \ar[d]_{\pi_+} \ar[r]^{\mathrm{id}} & \mathrm{Frac}\,\mathcal{W}(A_2) \ar[d]_{\eta_{123}^{(+)}}
\\
\mathrm{Frac}\,\mathcal{W}(A_2) \ar[r]^{\mathrm{id}} & \mathrm{Frac}\,\mathcal{W}(A_2)
\\
\mathcal{Y}(B'_\mathrm{FG}(A_2)) \ar[r]^{\alpha'} \ar[u]^{\phi'_\mathrm{FG}} & \mathcal{Y}(B'(A_2)) \ar[u]^{\phi'}
}
\qquad 
\xymatrix{
\mathcal{Y}(B_\mathrm{FG}(A_2)) \ar[d]_{\phi_\mathrm{FG}} \ar[r]^{\alpha} & \mathcal{Y}(B(A_2)) \ar[d]_{\phi}
\\
\mathrm{Frac}\,\mathcal{W}(A_2) \ar[d]_{\pi_{-}} \ar[r]^{\mathrm{id}} & \mathrm{Frac}\,\mathcal{W}(A_2) \ar[d]_{\eta_{123}^{(-)}}
\\
\mathrm{Frac}\,\mathcal{W}(A_2) \ar[r]^{\mathrm{id}} & \mathrm{Frac}\,\mathcal{W}(A_2)
\\
\mathcal{Y}(B'_\mathrm{FG}(A_2)) \ar[r]^{\alpha'} \ar[u]^{\phi'_\mathrm{FG}} & \mathcal{Y}(B'(A_2)) \ar[u]^{\phi'}
}
\end{align}
See \S 3 for the definitions of $\phi$ \eqref{Yw}, $\phi'$ \eqref{Ypw}, $\eta_{123}^{(-)}$ \eqref{uw-1} 
and $\eta_{123}^{(+)}$ \eqref{uw-2}.
\\

\begin{proposition}\label{pr:cd3}
(i) (Cf. \cite[Remark 8.3]{IKSTY}) The left diagram of \eqref{cd3} is commutative if and only if the parameters $\theta_i$ ($i=1$, $2$, $3$) 
and $(a_i,b_i, c_i, d_i, e_i)$ subject to \eqref{econ-a} further satisfy the relations
\begin{align}
&e_2 = e_3, 
\label{pare1}
\\
\begin{split}
&a_1=-a_3=c_3=-c_1, \quad a_2=c_2=0,
\label{pare1-2}
\\
&b_1 + e_1 = -d_1 = \theta_1, \quad b_2+e_2-a_1 = -a_1-d_2 = \theta_2, \quad b_3+e_3 = -d_3 = \theta_3. 
\end{split} 
\end{align}
(ii) The right diagram of \eqref{cd3} is commutative if and only if the parameters $\theta_i$ ($i=1$, $2$, $3$) 
and $(a_i,b_i, c_i, d_i, e_i)$ subject to \eqref{econ-a} satisfy \eqref{pare1-2} together with
\begin{align}\label{pare1-3}
e_1=e_2.
\end{align}
\end{proposition}
The proof of {\rm(ii)} is identical to that of {\rm(i)}, as carried out
in \cite{IKSTY}.

\begin{theorem}(Cf. \cite[\S 8]{IKSTY})
\label{th:fg}
(i) Assume \eqref{pare1} and \eqref{pare1-2}.
The $R$-operator ${R}_+$ \eqref{r+fg} is reproduced from the specialized $R$-operator $R_{123}^{(+)}$ \eqref{R-+-+} as
\begin{align}
\lim R_{123}^{(+)} = {R}_+,
\end{align}
where the limit is taken as
\begin{align}\label{elim}
e_1 \rightarrow -\infty, \;\;
e_2=e_3 \rightarrow -\infty, \;\;
e_1-e_3 \rightarrow -\infty,
\;\;
e_i+b_i= \mathrm{finite} \ (i=1,2,3).
\end{align}
(ii) Assume \eqref{pare1} and \eqref{pare1-3}. The $R$-operator ${R}_-$  \eqref{r-fg} 
is reproduced from the specialized $R$-operator $R_{123}^{(-)}$ \eqref{R--++} as
\begin{align}
\lim R_{123}^{(-)} = {R}_-,
\end{align}
where the limit is taken as
\begin{align}\label{elim2}
e_3 \rightarrow -\infty, \;\;
e_1=e_2 \rightarrow -\infty, \;\;
e_3-e_1 \rightarrow -\infty,
\;\;
e_i+b_i= \mathrm{finite} \ (i=1,2,3).
\end{align}
\end{theorem}

\begin{proof}
We present the proof for the first case. The second case is done in the similar manner.
Recall the parameters $\kappa_i$ \eqref{kad} for $P^{(+)}_{123}$ \eqref{P-+-+}.
From \eqref{pare1} it follows that $\kappa_0=0$, 
and from \eqref{pare1-2} it follows that $\kappa_1 = - \kappa_2 = \theta_3-\theta_2$ and $\kappa_3= 0$.
Hence $P^{(+)}_{123}$ reduces to $P_+$ \eqref{r+fg}.
By applying the relations \eqref{pare1} and \eqref{pare1-2}
to the dilogarithm part of the $R$-operator ${R}_{123}^{(+)}$ \eqref{R-+-+},
we obtain 
\begin{align}
{R}_{123}^{(+)}
&= \Psi_q(\e^{-\theta_{13}+e_1+u_1+u_3+w_1-w_2+w_3})^{-1}
\underline{\Psi_q(\e^{\theta_{13}+u_1-u_3-w_1+w_2-w_3})}
\nonumber \\
& \qquad \cdot 
\Psi_q(\e^{-\theta_{13}+e_1-e_3+u_1-u_3+w_1-w_2+w_3})^{-1}
\Psi_q(\e^{\theta_{13}+e_2-u_3+2u_2+u_1-w_1+w_2-w_3})\underline{{P}_+}.
\end{align}
In the limit \eqref{elim} the underlined parts survive. 
\end{proof}


\end{document}